\documentclass[pdflatex,sn-mathphys-num]{sn-jnl}

\usepackage{graphicx}%
\usepackage{multirow}%
\usepackage{amsmath,amssymb,amsfonts}%
\usepackage{amsthm}%
\usepackage{mathrsfs}%
\usepackage[title]{appendix}%
\usepackage{xcolor}%
\usepackage{textcomp}%
\usepackage{manyfoot}%
\usepackage{booktabs}%
\usepackage{algorithm}%
\usepackage{algorithmicx}%
\usepackage{algpseudocode}%
\usepackage{listings}%

\usepackage{bbm}
\usepackage{diagbox}
\usepackage{enumitem}
\usepackage{xr}

\theoremstyle{thmstyleone}%
\newtheorem{theorem}{Theorem}[section]
\newtheorem{proposition}[theorem]{Proposition}%

\newtheorem{corollary}[theorem]{Corollary}

\theoremstyle{thmstyletwo}%

\theoremstyle{thmstylethree}%
\newtheorem{definition}[theorem]{Definition}%

\begin{document}

\title[On the Well-posedness of HJB Equations]{On the Well-posedness of Hamilton-Jacobi-Bellman Equations of the Equilibrium Type}

\author[1]{\fnm{Qian} \sur{Lei}}

\author*[1]{\fnm{Chi Seng} \sur{Pun}}\email{cspun@ntu.edu.sg}

\affil[1]{\orgdiv{School of Physical \& Mathematical Sciences}, \orgname{Nanyang Technological University}, \orgaddress{\city{Singapore}, \country{Singapore}}}

\abstract{This paper studies the well-posedness of a class of nonlocal parabolic partial differential equations (PDEs), including as a special case the equilibrium Hamilton-Jacobi-Bellman equations, which has a strong tie with the characterization of the equilibrium strategies and the associated value functions for time-inconsistent stochastic control problems. Specifically, we consider nonlocality in both time and space, which allows for modelling of the stochastic control problems with initial-time-and-state dependent objective functionals. We leverage the method of continuity to show the global well-posedness within our proposed Banach space with our established Schauder prior estimate for the linearized nonlocal PDE. Then, we adopt a linearization method and Banach's fixed point arguments to show the local well-posedness of the nonlocal fully nonlinear case, while the global well-posedness is attainable provided that a sharp a-priori estimate is available. The well-posedness results contribute to advancing the understanding of long-standing open problems in equilibrium Hamilton–Jacobi–Bellman equations and time-inconsistent controls. Finally, we present a globally solvable financial example of time-inconsistency to validate our theoretical findings.}

\keywords{Existence and Uniqueness, Time-inconsistent stochastic control problems, Equilibrium Hamilton-Jacobi-Bellman equation, Nonlocal partial differential equation, Method of Continuity, Linearization}

\maketitle

\section{Introduction}
Stochastic control problems can be categorized as time-consistent or time-inconsistent, depending on whether Bellman's principle of optimality (BPO) holds. Classical stochastic control problems are time-consistent, and their solution methods are well documented in \cite{Yong1999}. However, violations of BPO are common in many decision-making problems, especially in behavioral finance and economics, whenever objective functionals depend on the initial time or state. For example, hyperbolic discounting involves initial-time dependence; see \cite{Laibson1997}. Endogenous habit formation and portfolio selection with state-dependent risk aversion involve initial-state dependence; see \cite{Bjoerk2017,Bjoerk2014a}.
This paper develops PDE theory to advance the understanding of long-standing problems in time-inconsistent (TIC) stochastic control.

When Bellman’s Principle of Optimality (BPO) fails to hold, the globally optimal solution may lose its optimality as time evolves, raising fundamental questions about the definition of ``optimal control” and how to characterize such controls. A straightforward approach is the pre-commitment policy, which fixes the decision at the initial time and thus reduces the time-inconsistent (TIC) problem to a classical stochastic control problem. However, maintaining this policy dynamically over multiple periods can incur additional costs, as the controller tends to deviate from the predetermined strategy in favor of the currently optimal control. In contrast, this paper adopts the widely used multi-person differential game approach to formulate TIC stochastic control problems. Within this game-theoretic framework, the solutions to TIC problems are identified as pure strategy Nash equilibria, whose inherent subgame-perfect consistency naturally avoids the dynamic implementability issues faced by pre-commitment policies. For a comprehensive overview of treatments for time inconsistency, we refer readers to \cite{He2022}, with further technical details provided in Section \ref{Sec:Pre}.

Following the PDE approach leads to two closely related methodologies for characterizing equilibrium solutions of TIC stochastic control problems: the extended Hamilton–Jacobi–Bellman (HJB) system and the equilibrium HJB equation, developed respectively in \cite{Basak2010, Bjoerk2017, Bjoerk2014a} and \cite{Yong2012, Wei2017, Yan2019}. The first approach introduces auxiliary functions and adjustment terms to restore the BPO from the perspective of subgame perfect equilibrium. Although verification theorems are available, its derivation and the definition of equilibrium policies remain largely heuristic. The second method, as demonstrated in \cite{Yong2012, Yan2019}, addresses these limitations by employing a discretization scheme that partitions the decision horizon into arbitrary subproblems and proves the convergence of the resulting recursive equations to an equilibrium HJB equation for the value function.

Hundreds of works have adopted the frameworks of \cite{Bjoerk2017,Bjoerk2014} or \cite{Yong2012,Wei2017} to analyze equilibrium policies and to explore the relation between the equilibrium HJB equation and TIC stochastic control problems. Specifically, the literature has examined two complementary aspects: \textbf{Sufficiency:} if a regular solution to the equilibrium HJB equation exists, then an equilibrium value function and the associated equilibrium policy can be identified; \textbf{Necessity:} conversely, any equilibrium policy admits a value function that solves the equilibrium HJB equation. For studies on sufficiency and necessity, we refer to \cite{Bjoerk2017,Hernandez2020} and \cite{Lindensjoe2019,Hernandez2020,He2021,Hamaguchi2021,He2022}, respectively. However, these discussions basically rest on the solvability (well-posedness) of the equilibrium HJB equation, which itself constitutes a standalone mathematical problem. The equilibrium HJB equation, as reformulated in the next section, is a nonlocal fully nonlinear PDE whose well-posedness lies beyond the scope of classical PDE theory.

\subsection{Related Literature and Challenges}
This paper aims to address the well-posedness issues for a general class of nonlocal fully nonlinear PDEs, which include as special cases both the extended HJB system and the equilibrium HJB equation derived from time-inconsistent stochastic control problems, of the form 
\begin{equation} \label{Intro:nonlocalfullynonlinearPDE}
	\left\{
	\begin{aligned}
		u_s(t,s,x,y) &= F\big(t,s,x,y,(\partial_I u)_{|I|\leq 2}(t,s,x,y), (\partial_I u)_{|I|\leq 2}(s,s,x,y)\big|_{x=y}\big), \\
		u(t,0,x,y) &= g(t,x,y),\quad t,s\in[0,T],\quad x,y\in\mathbb{R}^d, 
	\end{aligned}
	\right. 
\end{equation}
where the mapping (nonlinearity) $F$ could be nonlinear with respect to all its arguments, and both $s$ and $y$ are dynamical variables while $(t,x)$ should be considered as an external space-time parameter. Here, $I=(i_1,\ldots,i_j)$ is a multi-index with $j=|I|$, and $\partial_I u:=\frac{\partial^{|I|}u}{\partial y_{i_1}\cdots\partial y_{i_j}}$. The nonlocality comes from the dependence on the unknown function $u$ and its derivatives evaluated at not only the local point $(t,s,x,y)$ but also at the diagonal line of the space-time domain $(s,s,y,y)$. A more specific and relevant application of \eqref{Intro:nonlocalfullynonlinearPDE} is the equilibrium HJB equation \eqref{EquilibriumHJBequation} in Section \ref{Sec:Pre} that characterizes the equilibrium solution to a TIC stochastic control problem.

Most literature on \eqref{Intro:nonlocalfullynonlinearPDE} or \eqref{EquilibriumHJBequation} predominantly focuses on linear dependence on second-order derivatives and excludes the diagonal term $\left(\partial_I u\right)_{|I|= 2}(s,s,y,y)$. These cases correspond to TIC stochastic control with uncontrolled diffusion—effectively nonlocal quasilinear PDEs \eqref{QuasilinearPDE} solvable via fundamental solutions \cite{Yong2012,Wei2017,Yan2019}. However, leaving the diffusion uncontrolled severely limits modeling, especially for risk-sensitive tasks where direct control over uncertainty is essential to distinguish stochastic problems from deterministic ones. Furthermore, existing results often ignore the complex $x$-dependence of the objective functional \cite{Wei2017} or, despite addressing full nonlinearity \cite{Lei2023,Lei2021}, restrict nonlocality to the temporal dimension. This work thus extends the well-posedness framework to encompass nonlocality in both time and state.

To elaborate the challenges of extension from \cite{Yong2012,Wei2017,Yan2019} and from \cite{Lei2023,Lei2021}, we may revisit the classical contraction mapping approach to the well-posedness issue in \cite{Wei2017}. Specifically, we attempt to 
construct a nonlinear operator from $u$ to $U$ defined by the solution to the PDE
\begin{equation} \label{PreviousPDEs}
        U_s = \sum\limits_{|I|= 2}a^I(s,y)\partial_I U + 
        F\big(t,s,x,y,\left(\partial_I U\right)_{|I|\leq 1}, \left(\partial_I u\right)_{|I|\leq 1}\big|_{\begin{subarray}{l}
        t=s \\
        x=y
        \end{subarray}}\big), \quad U|_{s=0}=g.
\end{equation}
Replacing diagonal terms $(s,s,y,y)$ with a function $u$ allows classical PDE theory to yield a unique fixed point via contraction for \eqref{PreviousPDEs}. However, for fully nonlinear nonlocal terms as in \eqref{Intro:nonlocalfullynonlinearPDE}, the corresponding mapping is merely continuous, not contractive. While \cite{Lei2023, Lei2021} addressed time-only nonlocality using integral representations, advancing to nonlocality in both time and state introduces two primary challenges:
\begin{enumerate}
    \item Space and Norm Selection: Representing diagonal terms via $\int^t_s \cdot$ and $\int^x_y \cdot$ results in a definite temporal integral but a potentially indefinite spatial integral over an infinite domain. This necessitates identifying suitable function spaces, norms, and topologies for \eqref{Intro:nonlocalfullynonlinearPDE}.
    \item Lack of Schauder Estimates: The methods in \cite{Lei2023, Lei2021} do not provide Schauder a-priori estimates for the linearized PDE \eqref{NonlocalLinearPDE}, which are vital for the compactness required by fixed-point theorems. Unlike $x$-independent cases where Cauchy problems near $t$ suffice, the spatial dependency lacks boundary conditions to ensure well-posedness in local neighborhoods of $x$.
\end{enumerate}


Notably, even disregarding the fully nonlinear dependence of $F$ on the highest-order derivatives in \eqref{Intro:nonlocalfullynonlinearPDE}, the nonlocal dependence on the highest-order (diagonal) term $\left(\partial_I u\right)_{|I|=2}(s,s,y,y)$ introduces substantial analytical difficulties. In particular, even for the degenerate linearized equation \eqref{GenerallinearizedPDE}, well-posedness remains highly nontrivial.

\subsection{Our Approach}
The methods in the highly related literature \cite{Yong2012,Wei2017,Yan2019,Lei2023,Lei2021} are not feasible to address the well-posedness of the general 
nonlocal PDE \eqref{Intro:nonlocalfullynonlinearPDE}. In this paper, we provide a new approach for proving the well-posedness, which is compatible with all previous results. With the designs of norms and function spaces tailored for \eqref{Intro:nonlocalfullynonlinearPDE}, the main procedure of our analysis is outlined as follows:  

\textbf{Step 1a.} We first study a linearized version of \eqref{Intro:nonlocalfullynonlinearPDE} of the form 
\begin{equation} \label{SpeciallinearizedPDE}
    L_0u := u_s - \sum_{|I| \leq 2} a^I(s,y) \partial_I u + \sum_{|I| \leq 2} b^I(s,y) (\partial_I u) \Big|_{\substack{t=s \\ x=y}} = f, \quad u|_{s=0} = g.
\end{equation}
where both $a^I$ and $b^I$ are independent of $(t,x)$. It turns out that \eqref{SpeciallinearizedPDE} is mathematically equivalent to a decoupled system of PDEs (see \eqref{SystemforW} below) for a unknown vector-valued function. By proving that the system admits a regular enough solution and satisfies some important properties, we can show that there also exists a unique classical solution satisfying \eqref{SpeciallinearizedPDE} in $[0,T]^2\times\mathbb{R}^{d;d}$. Noteworthy is that due to the appearance of $\left(\partial_I u\right)_{|I|= 2}(s,s,x,y)\big|_{x=y}$, \eqref{SpeciallinearizedPDE} is not a special case of \eqref{PreviousPDEs};   

\textbf{Step 1b.} We then investigate a linearized PDE of \eqref{Intro:nonlocalfullynonlinearPDE} with general coefficients
\begin{equation} \label{GenerallinearizedPDE}
    Lu:=u_s-\sum_{|I|\leq 2} A^I(t,s,x,y)\partial_Iu+\sum_{|I|\leq 2} B^I(t,s,x,y)(\partial_Iu)\big|_{\begin{subarray}{l}
        t=s \\
        x=y
        \end{subarray}}=f, \quad u|_{s=0}=g,
\end{equation}
where $A^I$ and $B^I$ depend on both $(s,y)$ and $(t,x)$. The primary challenge in proving solvability lies in establishing a Schauder estimate. By utilizing specifically designed norms and function spaces, we provide quantitative regularity results and demonstrate that solutions to \eqref{GenerallinearizedPDE} are controlled by the non-homogeneous term $f$ and initial data $g$.


\textbf{Step 1c.} Let us consider a family of operators parameterized by $\tau\in[0,1]$: 
\begin{equation*}
    L_\tau u:=(1-\tau)L_0 u+\tau L_1 u
\end{equation*}
where $L_1u:=Lu$. Thanks to the Schauder estimate of solutions of \eqref{GenerallinearizedPDE}, we will take advantage of the method of continuity to prove the global well-posedness of \eqref{GenerallinearizedPDE} in $[0,T]^2\times\mathbb{R}^{d;d}$;  

\textbf{Step 2.} We analyze the operator $\Lambda(u)=U$, where $U$ is the solution of  \begin{equation} \label{utoU}
    U_s=LU+F\big(t,s,x,y,\left(\partial_I u\right)_{|I|\leq 2}, \left(\partial_I u\right)_{|I|\leq 2}\big|_{\begin{subarray}{l}
        t=s \\
        x=y
        \end{subarray}}\big)-Lu, \quad U|_{s=0}=g, 
\end{equation}
which is well-defined, provided that the nonlocal linear PDE \eqref{GenerallinearizedPDE} is well-posed. Moreover, it is obvious that each fixed point of \eqref{utoU} solves \eqref{Intro:nonlocalfullynonlinearPDE}. Thanks again to the Schauder estimate of solutions of \eqref{GenerallinearizedPDE}, we first prove that $\Lambda$ is a contraction and then make use of Banach's fixed point theorem to justify the local well-posedness of \eqref{Intro:nonlocalfullynonlinearPDE}. Subsequently, we show its global solvability, provided that a very sharp a prior estimate is available.



\subsection{Contributions and Organization of Our Paper}
Our contributions are mainly twofold. First, for such kind of nonlocal PDEs with initial-dynamic space-time structure arising from TIC stochastic control problems, we devise an analytical framework under which nonlocal linear/nonlinear PDEs are well-posed in the sense that we can establish the existence, uniqueness, and stability of their solutions. 
This paper has a detailed exploration about the underlying space of functions as well as mathematical properties of mappings between these spaces. Second, our framework allows the control variate entering the diffusion of state process, which breaks successfully through the existing bottleneck of TIC stochastic control problems. Together with the sufficiency and necessity analysis in the existing literature, our well-posedness results indicate directly the solvability of TIC control problems at least in a maximally-defined time interval. Thanks to our well-posedness and regularity results, some long-standing open problems in TIC stochastic control theory can be effectively analyzed and advanced within our analytical framework; see our Proposition \ref{SolvabilityofTIC} and its discussion following it. 
 
The rest of this paper is organized as follows. Section \ref{Sec:Pre} is devoted to the preliminaries for our study. We review the concepts of equilibrium controls and the associated equilibrium HJB equations for time-consistent stochastic optimal control problem. Section \ref{Sec:Linear} studies the linerized version of the nonlocal PDEs. We first establish the Schauder's prior estimate of solutions of the nonlocal linear PDEs, then take advantage of the method of continuity to prove its global well-posedness. In Section \ref{Sec:Nonlinear}, by the linearization method and Banach's fixed point theorem, we show that the nonlocal fully nonlinear PDE is locally solvable in a small time interval. Subsequently, we investigate extending the local well-posedness results to a larger time interval and a broader function space. Moreover, the fully nonlinear PDE is also globally solvable provided that a sufficiently sharp a priori estimate holds. As a corollary, we establish the global solvability of nonlocal quasilinear PDEs by leveraging these newly obtained well-posedness results.
In Section \ref{Sec:TIC}, we apply our PDE results to equilibrium HJB equations for the analysis of TIC stochastic control problems. 
Moreover, we provide financial TIC examples that are globally solvable. Finally, Section \ref{Sec:Conclusion} concludes. 


\section{Time-Inconsistent Stochastic Control Problems and Equilibrium Solutions} \label{Sec:Pre}

Let $(\Omega,\mathcal{F},\mathbb{F},\mathbb{P})$ be a complete filtered probability space that supports a $n$-dimensional standard Brownian motion, whose natural filtration augmented by all the $\mathbb{P}$-null sets is given by $\mathbb{F}=\{\mathcal{F}_s\}_{s\geq 0}$. Let $T>0$ be a finite horizon and $U\subseteq\mathbb{R}^m$ be a non-empty set that could be unbounded. The set of all admissible stochastic control processes over $[t,T]$ for $t\in [0,T)$ is defined as 
\begin{equation*}
    \begin{split}
        \mathcal{U}[t,T]:=\bigg\{\alpha:[t,T]\times\Omega\to U:\alpha(\cdot) \text{~is~} \mathbb{F}\text{-progressively measurable} 
        \text{~with~} \mathbb{E}\int^T_t|\alpha(\cdot)|^2ds<\infty\bigg\}. 
    \end{split}
\end{equation*}
To define a TIC problem, we often fix the time $t\in[0,T]$ and consider a time variable $s\in[t,T]$. It is convenient to introduce a set notation for the time pair $(t,s)$: $\nabla[0,T]:=\{(t,s):0\leq t\leq s\leq T\}$; similarly, we also define $\Delta[0,T]:=\{(t,s):0\leq s\leq t\leq T\}$. To ease notational burden, we also introduce $\mathbb{R}^{d;d}:=\mathbb{R}^d\times \mathbb{R}^d$.

\subsection{Stochastic Controls with Time-and-State-Varying Objectives} For a given pair $(t,x)\in[0,T]\times\mathbb{R}^d$, we aim to find an $\overline{\alpha}(\cdot)\in\mathcal{U}[t,T]$ such that 
\begin{equation} \label{TICproblem}
    J(t,x;\overline{\alpha}(\cdot)):=\inf\limits_{\alpha(\cdot)\in\mathcal{U}[t,T]}J(t,x;\alpha(\cdot))
\end{equation}
where the cost functional $J(t,x;\alpha(\cdot)):=Y(t;t,x,\alpha(\cdot))$ with $(X(\cdot),Y(\cdot),Z(\cdot))$ (in greater detail, $(X(\cdot;t,x,\alpha(\cdot)),Y(\cdot;t,x,\alpha(\cdot)),Z(\cdot;t,x,\alpha(\cdot)))$) being the adapted solution to the following controlled forward-backward stochastic differential equations (FBSDEs): 
\begin{equation} \label{ControlledFBSDE}
    \left\{
    \begin{aligned}
        dX(s) &= b(s,X(s),\alpha(s))ds + \sigma(s,X(s),\alpha(s))dW(s), && s \in [t,T], \\
        dY(s) &= -h(t,s,X(t),X(s),\alpha(s),Y(s),Z(s))ds + Z(s)dW(s), && s \in [t,T], \\
        X(t) &= x, \quad Y(T) = g(t,X(t),X(T)),
    \end{aligned}
    \right.
\end{equation}
where $b:[0,T]\times\mathbb{R}^d\times U\to\mathbb{R}^d$ and $\sigma:[0,T]\times\mathbb{R}^d\times U\to\mathbb{R}^{d\times n}$ are the drift and volatility of the state process $X(\cdot)$, respectively, $h:\nabla[0,T]\times\mathbb{R}^{d;d}\times  U\times\mathbb{R}\times\mathbb{R}^{1\times n}\to\mathbb{R}$ and $g:[0,T]\times\mathbb{R}^{d;d}$ are the generator and terminal condition of the utility process $(Y(\cdot),Z(\cdot))$, respectively, and they are all deterministic functions. Under some suitable conditions (see \cite[Proposition 3.3]{Ma1999}), for any $(t,x)\in[0,T]\times\mathbb{R}^d$ and $\alpha(\cdot)\in\mathcal{U}[t,T]$, the controlled FBSDEs \eqref{ControlledFBSDE} admit a unique adapted solution $(X(\cdot),Y(\cdot),Z(\cdot))$. Moreover, \cite{Karoui1997} reveals that it admits a probabilistic representation: 
\begin{equation*}
    J(t,x; \alpha(\cdot)) = \mathbb{E}_{t,x} \left[ g(t, X(T), X(T)) + \int_t^T h(t, s, X(t), X(s), \alpha(s), Y(s), Z(s)) \, ds \right]
\end{equation*}
where $\mathbb{E}_{t,x}[\cdot]$ is the conditional expectation $\mathbb{E}[\cdot|\mathcal{F}_t]$ under $X(t)=x$. One can easily observe that the BPO for \eqref{TICproblem} is not available as the $h$ and the $g$ in \eqref{ControlledFBSDE} depend on the current time $t$ and the current state $X(t)$. As a result, 
even the agent can find an optimal control, denoted by $\overline{\alpha}(\cdot):=\overline{\alpha}(\cdot;t,x)$, for the problem over $[t,T]$ with any initial pair $(t,x)\in[0,T]\times\mathbb{R}^d$, we can anticipate that at a later time point $s\in(t,T]$,
due to the time-and-state-dependence of objectives,
\begin{equation} \label{DPPviolation}
    J \bigl( s, \overline{X}(s); \overline{\alpha}(\cdot; t, x) \big|_{[s, T]} \bigr) > J \bigl( s, \overline{X}(s); \overline{\alpha}(\cdot; s, \overline{X}(s)) \bigr) \quad \text{almost surely}.
\end{equation}
where $\overline{X}(\cdot)$ is the adapted solution to \eqref{ControlledFBSDE} with $\overline{\alpha}(\cdot)$ and $(t,x)$ fixed. Hence, there is an incentive to deviate from the optimal control policy derived at $(t,x)$, $\overline{\alpha}(\cdot;t,x)$, as time evolves. Such problems are called TIC problems.

This paper adopts the widely used game-theoretical approach to tackle TIC problems.
We partition $[0,T]$ into $N$ subintervals $[t_{0},t_{1}),\ldots,[t_{N-1},t_{N})$ via $\mathcal{P}:0=t_0<t_1<\cdots<t_N=T$, and interpret the TIC stochastic control problem as an $N$-person stochastic differential game.
Player $k$ $(1\le k\le N)$ controls the system over $[t_{k-1},t_k)$ with her own admissible control $\alpha^k(\cdot)\in\mathcal{U}[t_{k-1},t_k)$.
The individual problems are linked through sophisticated cost functionals: the cost of Player $k$ is determined by the solution to \eqref{ControlledFBSDE} on $[t_k,t_{k+1}]$ ($t=t_k$) under the action $\alpha^{k+1}(\cdot)$ of Player $k+1$, namely $Y(t_{k+1})=Y(t_{k+1};t_{k+1},X(t_{k+1}),\alpha^{k+1}(\cdot))$.
Thus, cost functionals and actions are resolved backward, while each Player $k$ still solves a conventional, time-consistent problem over $[t_{k-1},t_k)$ based on her $(t_{k-1},X(t_{k-1}))$-dependent preference.
Agents with such sophisticated cost functionals are called sophisticated agents—thinking globally but acting locally.
We next detail the definition of equilibrium strategies.

\begin{definition}[\cite{Yong2012,Wei2017,Yan2019}] \label{Def:EquilibriumControl}
A continuous map $\mathbbm{e}:[0,T]\times\mathbb{R}^d\to U$ is called a closed-loop equilibrium strategy of the TIC stochastic control problem \eqref{TICproblem} if the following two conditions hold:

 1. For any $x\in\mathbb{R}^d$, the dynamics equation
\begin{equation*} 
    \left\{
    \begin{aligned}
        d\overline{X}(s) &= b(s,\overline{X}(s),\mathbbm{e}(s,\overline{X}(s)))ds+\sigma(s,\overline{X}(s),\mathbbm{e}(s,\overline{X}(s)))dW(s), \quad s\in[0,T], \\
        \overline{X}(0) &= x, 
    \end{aligned}
    \right.
\end{equation*}
    admits a unique solution $\overline{X}(\cdot)$; 

 2. For each $(s,a)\in[0,T)\times U$, let $X^\epsilon(\cdot)$ satisfy 
    \begin{equation*}
    \left\{
    \begin{aligned}
        dX^\epsilon(s) &= b(s,X^\epsilon(s),a)ds+\sigma(s,X^\epsilon(s),a)dW(s), && s\in[t,t+\epsilon), \\
        dX^\epsilon(s) &= b(s,X^\epsilon(s),\mathbbm{e}(s,X^\epsilon(s)))ds+\sigma(s,X^\epsilon(s),\mathbbm{e}(s,X^\epsilon(s)))dW(s), && s\in[t+\epsilon,T], \\
        X^\epsilon(t) &= \overline{X}(t),
    \end{aligned}
    \right. 
\end{equation*}
then the following inequality holds: 
\begin{equation} \label{Localoptimality} 
    \underset{\epsilon\downarrow 0}{\underline{\lim}}\frac{J\left(t,\overline{X}(t);a\cdot\mathbf{1}_{[t,t+\epsilon)}\oplus\mathbbm{e}\right)-J\left(t,\overline{X}(t);\mathbbm{e}\right)}{\epsilon}\geq 0, 
\end{equation}
where  
\begin{equation} \label{Def:epsilonpolicy}
\left(a \cdot \mathbf{1}_{[s, s+\epsilon)} \oplus \mathbbm{e}\right)(s) = 
\begin{cases}
    a, & s \in [t, t+\epsilon), \\
    \mathbbm{e}(s, X^\epsilon(s)), & s \in [t+\epsilon, T].
\end{cases}
\end{equation}
Furthermore, $\big\{\overline{X}(s)\big\}_{s\in[0,T]}$ and $V(t,\overline{X}(t)):= J\big(t,\overline{X}(t);\mathbbm{e}(s,\overline{X}(s))\}_{\tau\in[t,T]}\big)$ are called the equilibrium state process and the equilibrium value function, respectively. 
\end{definition}

Condition \eqref{Localoptimality} characterizes a subgame perfect equilibrium (SPE) solution to a game played by the incarnations of the agent at different time points. Hence, the closed-loop equilibrium strategy achieves local optimality in a proper sense. Considering the violation of the BPO and the deviation of optimal controls as time evolves in TIC problems, such a locally optimal control revives the recursive relationship between two sub-problems initiating at $(t,\overline{X}(t))$ and $(t+\epsilon,X^\epsilon(t+\epsilon))$, respectively. As a result, the closed-loop equilibrium strategy is time-consistent and free of \eqref{DPPviolation}. For more on the game-theoretic interpretation of TIC problems, see \cite{Bjoerk2017, Yong2012, Yan2019, Wei2017}.



\subsection{Equilibrium HJB Equations} \label{Subsec:EqHJB}
The equilibrium solution (Definition \ref{Def:EquilibriumControl})—comprising strategy, state process, and value function—was pioneered in \cite{Bjoerk2017} using game-theoretic ideas. However, \cite{Yong2012,Wei2017,Yan2019} noted a lack of rigor regarding the state process solvability under the $\epsilon$-policy \eqref{Def:epsilonpolicy} and the heuristic nature of extended dynamic programming. We adopt the more rigorous framework of \cite{Yong2012,Wei2017,Yan2019}, noting that both approaches ultimately converge on the same equilibrium HJB equation. 

Briefly, the derivation involves partitioning the time interval to define piecewise approximate strategies and value functions, which are then stitched into time-consistent solutions. Under Assumption \eqref{Selection}, we construct an $N$-person equilibrium. Taking the mesh limit $|\mathcal{P}|\to 0$ yields the continuous-time strategy $\mathbbm{e}(s,y)=\Psi(s,y)$ and value function $V(s,y)=u(s,s,y,y)$, where $u(t,s,x,y)$ satisfies the parabolic PDE \eqref{EquilibriumHJBequation} with its initial-dynamic space–time structure:
\begin{equation} \label{EquilibriumHJBequation}
    \left\{
    \begin{aligned}
        &u_s(t,s,x,y) + \mathcal{H}\bigl(t,s,x,y,\Psi(s,y),u(t,s,x,y),u_y(t,s,x,y),u_{yy}(t,s,x,y)\bigr) = 0, \\
        &u(t,T,x,y) = g(t,x,y), 
    \end{aligned}
    \right.
\end{equation}
with the Hamiltonian given by
\begin{equation} \label{Hamiltonian}
	\begin{split}
		\mathcal{H}(t,s,x,y,a,u,p,q) &= \frac{1}{2} \mathrm{tr} \bigl[ q \cdot (\sigma\sigma^\top)(s,y,a) \bigr] + p^\top b(s,y,a) \\
		&\quad + h \bigl( t, s, x, y, a, u, p^\top \cdot \sigma(s,y,a) \bigr)
	\end{split}
\end{equation}
for $(t,s,x,y,a,u,p,q)\in\nabla[0,T]\times\mathbb{R}^{d;d}\times U\times\mathbb{R}\times\mathbb{R}^d\times\mathbb{S}^d$, in which the superscript $\top$ denotes the transpose of vectors or matrices and $\mathbb{S}^d\subseteq\mathbb{R}^{d\times d}$ denotes the set of all $d\times d$-symmetric matrices and
$$\Psi(s,y) = \psi \Bigl( s, s, y, y, u(s, s, y, y), u_y(s, s, x, y) \big|_{x=y}, u_{yy}(s, s, x, y) \big|_{x=y} \Bigr)$$
for $(s,y)\in[0,T]\times\mathbb{R}^d$, in which we assume that there exists a map $\psi:\nabla[0,T]\times\mathbb{R}^{d;d}\times\mathbb{R}\times\mathbb{R}^d\times\mathbb{S}^d\to U$ with all needed smoothness and boundedness of its derivatives such that 
\begin{equation} \label{Selection}
    \psi(t,s,x,y,u,p,q) \in \left\{ \overline{a} \in U : \mathcal{H}(t,s,x,y,\overline{a},u,p,q) = \min_{a \in U} \mathcal{H}(t,s,x,y,a,u,p,q) \right\}
\end{equation}
holds for all $(t,s,x,y,u,p,q)\in\nabla[0,T]\times\mathbb{R}^{d;d}\times\mathbb{R}\times\mathbb{R}^d\times\mathbb{S}^d$.

Equation \eqref{EquilibriumHJBequation} is an \textit{equilibrium HJB equation} with an initial–dynamic structure, involving both $(t,x)$ and $(s,y)$. The variables $(t,x)$ are not mere parameters due to nonlocal terms such as $u(s,s,y,y)$, $u_y(s,s,x,y)\big|_{x=y}$, and $u_{yy}(s,s,x,y)\big|_{x=y}$, making \eqref{EquilibriumHJBequation} fully nonlinear and nonlocal. Heuristically, “global thinking” corresponds to terms at $(t,s,x,y)$, while “local acting” corresponds to those at $(s,s,y,y)$. When \eqref{EquilibriumHJBequation} is well-posed, the limit as $|\mathcal{P}|\to 0$ becomes rigorous. Moreover, as shown in \cite{Yong2012,Wei2017,Yan2019}, $(\mathbbm{e},V)(s,y):=(\Psi(s,y),u(s,s,y,y))$ gives the closed-loop equilibrium strategy and value function. Hence, well-posedness of \eqref{EquilibriumHJBequation} is essential for TIC control problems.

By a standard change of time variables as discussed in \cite{Lei2021,Lei2023}, \eqref{EquilibriumHJBequation} can be reformulated as an initial value problem in a forward form of \eqref{Intro:nonlocalfullynonlinearPDE}, which can ease the notational burden compared to the terminal value problem. Moreover, the order relation $t\leq s$ between the initial time point $t$ and the running time $s$ in \eqref{EquilibriumHJBequation} can be removed in the study of nonlocal PDEs, since it is natural to extend the solutions of \eqref{Intro:nonlocalfullynonlinearPDE} from the triangular time zone $\nabla[0,T]$ to a rectangular one $[0,T]^2$. In the next two sections, we will establish the well-posedness of \eqref{Intro:nonlocalfullynonlinearPDE}. 


\section{Nonlocal Linear PDEs} \label{Sec:Linear}
To establish the well-posedness of the fully nonlinear PDE \eqref{Intro:nonlocalfullynonlinearPDE}, we first study its linearized version \eqref{GenerallinearizedPDE}.
While the nonlocal linear PDE is simpler by right, it plays a crucial role in the study of nonlocal fully nonlinear PDEs with a linearization method in Section \ref{Sec:Nonlinear}.  

\subsection{Function Spaces and Nonlocal Differential Operators} 

Let $C([a,b]\times\mathbb{R}^d;\mathbb{R})$ denote the set of continuous, bounded real functions on $[a,b]\times\mathbb{R}^d$ endowed with the supremum norm $|\cdot|^{(0)}$. To ensure the existence of classical solutions for second-order parabolic equations, we utilize H"{o}lder spaces \cite{Eidelman1969,Lei2023}. Specifically, $C^{\frac{l}{2},l}({[a,b]\times\mathbb{R}^d};\mathbb{R})$ is the Banach space of continuous functions $\varphi(s,y)$ with existing derivatives $D^i_sD^j_y\varphi$ for $2i+j<l$ and a finite norm defined by: 
\begin{equation*}
	\begin{split}
        |\varphi|^{(l)}_{[a,b] \times \mathbb{R}^d} &:= \sum_{k \le [l]} \sum_{2i+j=k} |D^i_s D^j_y \varphi|^{(0)} + \sum_{2i+j=[l]} \langle D^i_s D^j_y \varphi \rangle^{(l-[l])}_y  + \sum_{0 < l-2i-j < 2} \langle D^i_s D^j_y \varphi \rangle^{(\frac{l-2i-j}{2})}_s,
    \end{split}
\end{equation*}
where $l$ is a non-integer positive number with $[l]$ being its integer part and for $\alpha\in(0,1)$, 
\begin{equation*}
    \langle\varphi\rangle^{(\alpha)}_y := \sup_{\substack{a \leq s \leq b \\ 0 < |y-y^\prime| \leq 1}} \frac{|\varphi(s,y)-\varphi(s,y^\prime)|}{|y-y^\prime|^\alpha}, \quad \langle\varphi\rangle^{(\alpha)}_s := \sup_{\substack{a \leq s < s^\prime \leq b \\ y \in \mathbb{R}^d}} \frac{|\varphi(s,y)-\varphi(s^\prime,y)|}{|s-s^\prime|^\alpha}.
\end{equation*}
Wherever no confusion arises, we do not distinguish between $|\varphi|^{(l)}_{[a,b]\times\mathbb{R}^d}$ and $|\varphi|^{(l)}_{\mathbb{R}^d}$ for functions $\varphi(y)$ independent of $s$. 

Considering the pair of space-time arguments (i.e. $(t,x)$ and $(s,y)$) of solutions of nonlocal PDEs, we can similarly introduce the space $C([a,b]^2\times\mathbb{R}^{d;d};\mathbb{R})$. For a real-valued function $\psi(t,s,x,y)$ and a vector-valued $\Psi=(\psi^1,\psi^2,\cdots,\psi^m)(t,s,x,y)$, we introduce the following norms: 
\begin{equation*}
    \begin{aligned}
    |\Psi|^{(l)}_{[a,b]} &:= \sum_m |\psi^m(t, \cdot, x, \cdot)|^{(l)}_{[a,b]}, \quad
    [\Psi]^{(l)}_{[a,b]} &:= \sup_{(t,x) \in [a,b] \times \mathbb{R}^d} \Bigl\{ \sum_m |\psi^m(t, \cdot, x, \cdot)|^{(l)}_{[a,b]} \Bigr\}
\end{aligned}
\end{equation*}
\begin{equation*}
    \begin{aligned}
        \lVert\psi\rVert^{(l)}_{[a,b]} &:= \sup_{(t,x) \in [a,b] \times \mathbb{R}^d} \Bigl\{ |(\psi, \psi_t, \psi_x, \psi_{xx})(t, \cdot, x, \cdot)|^{(l)}_{[a,b]} \Bigr\} \\
        &:= \sup_{(t,x) \in [a,b] \times \mathbb{R}^d} \Bigl\{ |\psi(t, \cdot, x, \cdot)|^{(l)}_{[a,b]} + |\psi_t(t, \cdot, x, \cdot)|^{(l)}_{[a,b]} + |\psi_x(t, \cdot, x, \cdot)|^{(l)}_{[a,b]}  + |\psi_{xx}(t, \cdot, x, \cdot)|^{(l)}_{[a,b]} \Bigr\},
    \end{aligned}
\end{equation*}
which induces the following Banach space
\begin{equation*}
    \Omega^{(l)}_{[a,b]}:=\Big\{\psi\in C([a,b]^2\times\mathbb{R}^{d;d};\mathbb{R}):
    \Vert\psi\rVert^{(l)}_{[a,b]}<\infty\Big\}. 
\end{equation*}
To ease the notational burden, we introduce the following vector functions:
\begin{gather*}
    \overline{\psi}(t,s,x,y) = (\psi, \psi_t, \psi_x, \psi_{xx})^\top, \quad \overleftarrow{\psi}(t,s,x,y) = (\psi, \psi_t, \psi_x)^\top, \quad
    \overrightarrow{\psi}(t,s,x,y) = (\psi_t, \psi_x, \psi_{xx})^\top
\end{gather*}
Hence, $\lVert\psi\rVert^{(l)}_{[a,b]}$ can be rewritten as $[\overline{\psi}]^{(l)}_{[a,b]}$. Now, we turn to regulate the nonlocal linear differential operator: 
\begin{equation} \label{Nonlocallinearoperator}
    \begin{split}
        Lu := u_s(t,s,x,y) &- \sum_{|I| \leq 2} A^I \partial_I u(t,s,x,y) + \sum_{|I| \leq 2} B^I \partial_I u(s,s,x,y) \big|_{x=y}
    \end{split}
\end{equation}
where $I=(I_1,I_2,\cdots,I_d)$ is a multi-index of non-negative integers, $|I|=I_1+I_2+\cdots+I_d$. The operator $\partial_I$ is interpreted as the partial derivative $\frac{\partial^{|I|}}{\partial y^{I_1}_1\partial y^{I_2}_2\cdots\partial y^{I_d}_d}$ of order $I_i$ in $y_i$ and $u_s$ as the derivative of $u$ in $s$. Moreover, for each $I$, the coefficients $A^I(t,s,x,y)$, $B^I(t,s,x,y)\in\Omega^{(\alpha)}_{[0,T]}$ satisfy the uniform ellipticity conditions, i.e. there exists some $\lambda>0$ such that 
\begin{align} 
	\sum_{|I|= 2} A^I(t,s,x,y) \xi^I &\geq \lambda |\xi|^2, \label{Uniformellipticitycondition1} \\[1ex]
	\sum_{|I|= 2} \bigl( A^I(t,s,x,y) + B^I(t,s,x,y) \bigr) \xi^I &\geq \lambda |\xi|^2, \label{Uniformellipticitycondition2}
\end{align}
for any $(t,s,x,y)\in[0,T]^2\times\mathbb{R}^{d;d}$ and $\xi\in\mathbb{R}^d$, where $\xi^I=\xi^{I_1}_1\xi^{I_2}_2\cdots\xi^{I_d}_d$. It is noteworthy that the nonlocal operator \eqref{Nonlocallinearoperator} and the uniformly ellipticity conditions \eqref{Uniformellipticitycondition1}-\eqref{Uniformellipticitycondition2} reduce to the classical counterparts when $B^I=0$.   

With the introduction of \eqref{Nonlocallinearoperator} and its regularity conditions, we study the nonlocal linear PDE of the form 
\begin{equation} \label{NonlocalLinearPDE}
    \left\{
    \begin{aligned}
        Lu(t,s,x,y) &= f(t,s,x,y), \\
        u(t,0,x,y) &= g(t,x,y), \quad t,s \in [0,T], \quad x,y \in \mathbb{R}^d.
    \end{aligned}
    \right.
\end{equation}
where the non-homogeneous term $f\in\Omega^{(\alpha)}_{[0,T]}$ and the initial condition $g\in\Omega^{(2+\alpha)}_{[0,T]}$. Roughly speaking, the existence, uniqueness, and stability of solutions of nonlocal linear PDE \eqref{NonlocalLinearPDE} corresponds to the surjection, injection, and continuity (boundedness) properties of nonlocal operator \eqref{Nonlocallinearoperator} within $\Omega^{(l)}_{[a,b]}$, respectively. Next, we will find that \eqref{NonlocalLinearPDE} is well-posed in $\Omega^{(l)}_{[0,T]}$ under some mild conditions. Consequently, the nonlocal linear operator $L:\big\{u\in\Omega^{(2+\alpha)}_{[0,T]}:u|_{s=0}=g\big\}\rightarrow\Omega^{(\alpha)}_{[0,T]}$ is bijective and continuous. 


\subsection{Schauder's Estimate of Solutions to Nonlocal Linear PDEs}


We use the method of continuity to establish global existence for the nonlocal linear PDE \eqref{NonlocalLinearPDE} by embedding it into a family of parameterized problems. Instead of tackling \eqref{NonlocalLinearPDE} directly, we study a simpler equation in this family and then transfer solvability to the original problem via continuation. The key step is a parameter-independent Schauder estimate for the whole family. This a priori estimate controls solution behavior, yields regularity, and ensures compactness of the solution set. Such compactness is crucial for both the continuity method in the linear case and fixed-point arguments in the fully nonlinear setting.

In what follows, we establish the prior estimate of solutions to nonlocal linear PDEs \eqref{NonlocalLinearPDE}. First of all, let us rewrite \eqref{NonlocalLinearPDE} as the following form 
\begin{equation} \label{ExpandedlinearPDE}
    \left\{
    \begin{aligned}
        u_s(t,s,x,y) &= \sum_{|I|\leq 2} A^I(t,s,x,y) \partial_I u(t,s,x,y) \\
        &\quad + \sum_{|I|\leq 2} B^I(t,s,x,y) \partial_I u(s,s,x,y) \big|_{x=y} + f(t,s,x,y), \\
        u(t,0,x,y) &= g(t,x,y), \quad t,s \in [0,T], \quad x,y \in \mathbb{R}^d.
    \end{aligned}
    \right.
\end{equation}
Suppose that \eqref{ExpandedlinearPDE} admits a solution $u(t,s,x,y)\in\Omega^{(2+\alpha)}_{[0,T]}$. Then, for any $(t,x)\doteq(x_0,x_1,\cdots,x_d)^\top\in[0,T]\times\mathbb{R}^d$, we have 
\begin{equation} \label{Firstderivatives}
\left\{
    \begin{aligned}
        \left(\frac{\partial u}{\partial x_i}\right)_s(t,s,x,y) &= \sum_{|I|\leq 2} A^I \partial_I \left(\frac{\partial u}{\partial x_i}\right)(t,s,x,y) + \sum_{|I|\leq 2} A^I_{x_i} \partial_I u(t,s,x,y) \\
        &\quad + \sum_{|I|\leq 2} B^I_{x_i} \partial_I u(s,s,x,y) \big|_{x=y} + f_{x_i}, \quad i=1,\ldots,d, \\
        \left(\frac{\partial u}{\partial x_i}\right)(t,0,x,y) &= g_{x_i}(t,x,y), \quad t,s \in [0,T], \quad x,y \in \mathbb{R}^d.
    \end{aligned}
\right. 
\end{equation}
where the dependence of $A^I$, $B^I$, and $f$ on their arguments $(t,s,x,y)$ is suppressed here. Furthermore, we also have
\begin{equation} \label{Secondderivatives} 
\left\{
    \begin{aligned}
        \left(\frac{\partial^2 u}{\partial x_i\partial x_j}\right)_s(t,s,x,y) &= \sum_{|I|\leq 2} A^I \partial_I \left(\frac{\partial^2 u}{\partial x_i \partial x_j}\right)(t,s,x,y) + \sum_{|I|\leq 2} A^I_{x_j} \partial_I \left(\frac{\partial u}{\partial x_i}\right)(t,s,x,y) \\
        &\quad + \sum_{|I|\leq 2} A^I_{x_i} \partial_I \left(\frac{\partial u}{\partial x_j}\right)(t,s,x,y) + \sum_{|I|\leq 2} A^I_{x_ix_j} \partial_I u(t,s,x,y) \\
        &\quad + \sum_{|I|\leq 2} B^I_{x_ix_j} \partial_I u(s,s,x,y) \big|_{x=y} + f_{x_ix_j}, \quad i,j=1,\ldots,d, \\
        \left(\frac{\partial^2 u}{\partial x_i\partial x_j}\right)(t,0,x,y) &= g_{x_ix_j}(t,x,y), \quad t,s \in [0,T], \quad x,y \in \mathbb{R}^d.
    \end{aligned}
\right. 
\end{equation}
To simplify, \eqref{ExpandedlinearPDE}-\eqref{Secondderivatives} can be reorganized in a compact way as a parabolic system for a vector-valued function $\overline{u}^\top=\left(u,\frac{\partial u}{\partial t},\frac{\partial u}{\partial x},\frac{\partial^2 u}{\partial x\partial x}\right)^\top(t,s,x,y)$: 
        
          
\begin{equation} \label{Systemforoverlineu}
    \left\{
    \begin{aligned}
        \overline{u}^\top_s(t,s,x,y) &= \sum_{|I|\leq 2} P^I \partial_I \overline{u}^\top(t,s,x,y) + \sum_{|I|\leq 2} (\overline{B}^I)^\top \partial_I u(s,s,x,y) \big|_{x=y} + \overline{f}^\top, \\
        \overline{u}^\top(t,0,x,y) &= \overline{g}^\top(t,x,y), \quad t,s \in [0,T], \quad x,y \in \mathbb{R}^d.
    \end{aligned}
    \right.
\end{equation}
where each of $P^I=P^I(t,s,x,y)$ for $|I|\leq 2$ is a lower-triangular matrix, whose diagonal elements are exactly $A^I$ while the off-diagonal elements do not matter the subsequent analyses. Moreover, $\overline{B}^I$, $\overline{f}$, and $\overline{g}$ are all vector-valued, consisting of themselves and their derivatives in $t$ and $x$. Thanks to the structure of such a matrix $P^I$, the existence and regularity of the fundamental solution of the parabolic operator $Du:=u_s-\sum P^I\partial_Iu$ is promised by \eqref{Uniformellipticitycondition1} of $A^I$; see \cite{Friedman1964,Ladyzhanskaya1968,Eidelman1969}.  

Next, we take advantage of the integral representations below to replace all diagonal terms $\partial_I u(s,s,x,y)\big|_{x=y}$ in \eqref{ExpandedlinearPDE}-\eqref{Firstderivatives} with a relatively manageable $\partial_I u(t,s,x,y)$. 
\begin{equation} \label{Intergralrepresentations}
    \begin{aligned}
        & \partial_I u(t,s,x,y) - \partial_I u(s,s,x,y)\big|_{x=y} \\
        &\quad = \int^t_s \partial_{I}\left(\frac{\partial u}{\partial t}\right)(\theta_t,s,x,y) \, d\theta_t + \int^{x_1}_{y_1} \partial_{I}\left(\frac{\partial u}{\partial x_1}\right)(s,s,\theta_1,x_2,\cdots,x_d,y) \, d\theta_1 \\
        &\qquad + \int^{x_2}_{y_2} \partial_{I}\left(\frac{\partial u}{\partial x_2}\right)(s,s,x_1,\theta_2,x_3,\cdots,x_d,y) \Big|_{x_1=y_1} \, d\theta_2 \\
        &\qquad + \int^{x_3}_{y_3} \partial_{I}\left(\frac{\partial u}{\partial x_3}\right)(s,s,x_1,x_2,\theta_3,x_4,\cdots,x_d,y) \Big|_{\begin{subarray}{c} x_1=y_1 \\ x_2=y_2\end{subarray}} \, d\theta_3 + \cdots \\
        &\qquad + \int^{x_{d-1}}_{y_{d-1}} \partial_{I}\left(\frac{\partial u}{\partial x_{d-1}}\right)(s,s,x_1,x_2,\cdots,x_{d-2},\theta_{d-1},x_d,y) \Big|_{\begin{subarray}{c} x_i=y_i \\ i=1,2,\cdots,d-2\end{subarray}} \, d\theta_d \\  
        &\qquad + \int^{x_d}_{y_d} \partial_{I}\left(\frac{\partial u}{\partial x_d}\right)(s,s,x_1,x_2,\cdots,x_{d-1},\theta_d,y) \Big|_{\begin{subarray}{c} x_i=y_i \\ i=1,2,\cdots,d-1\end{subarray}} \, d\theta_d \\
    \end{aligned}
\end{equation}
\begin{flalign*} 
    & \doteq -\mathcal{I}^I\left[\frac{\partial u}{\partial t}, \frac{\partial u}{\partial x}\right](t,s,x,y) &
\end{flalign*}
Note that the $(d+1)$-dimensional $\left(\frac{\partial u}{\partial t},\frac{\partial u}{\partial x}\right)^\top$ constitutes a conservative vector field and the potential function of which is $u(t,s,x,y)$. Hence, $\partial_I u(s,s,x,y)\big|_{x=y}$ has various integral representations when we alter the integral paths from $(s,y)$ to $(t,x)$.   

Thanks to the integral representation \eqref{Intergralrepresentations}, the equations \eqref{ExpandedlinearPDE} and \eqref{Firstderivatives} can be rewritten as the following coupled system of PDEs: 
\begin{equation} \label{IntegralRep}
    \left\{
    \begin{aligned}
        u_s(t,s,x,y) &= \sum_{|I|\leq 2}(A^I+B^I)\partial_I u(t,s,x,y)  + \sum_{|I|\leq 2}B^I\mathcal{I}^I\left[\frac{\partial u}{\partial t},\frac{\partial u}{\partial x}\right](t,s,x,y) + f, \\
        \left(\frac{\partial u}{\partial x_i}\right)_s(t,s,x,y) &= \sum_{|I|\leq 2}A^I\partial_I \left(\frac{\partial u}{\partial x_i}\right)(t,s,x,y) + \sum_{|I|\leq 2}(A^I_{x_i}+B^I_{x_i})\partial_I u(t,s,x,y) \\
        &\quad + \sum_{|I|\leq 2}B^I_{x_i}\mathcal{I}^I\left[\frac{\partial u}{\partial t},\frac{\partial u}{\partial x}\right](t,s,x,y) + f_{x_i}, \quad i=0,1,\dots,d, \\
        \left(u,\frac{\partial u}{\partial x_i}\right)(t,0,x,y) &= (g,g_{x_i})(t,x,y), \quad t,s\in[0,T], \quad x,y\in\mathbb{R}^d.
    \end{aligned}
    \right.
\end{equation}
which is equivalent to a parabolic system for $\overleftarrow{u}^\top=\left(u,\frac{\partial u}{\partial t},\frac{\partial u}{\partial x}\right)^\top(t,s,x,y)$: 
\begin{equation} \label{Systemforoverleftarrowu}
    \left\{
    \begin{aligned}
        \overleftarrow{u}^\top_s(t,s,x,y) &= \sum_{|I|\leq 2} Q^I \partial_I \overleftarrow{u}^\top(t,s,x,y) + \sum_{|I|\leq 2} \overleftarrow{B}^\top \mathcal{I}^I\left[\frac{\partial u}{\partial t}, \frac{\partial u}{\partial x}\right](t,s,x,y) + \overleftarrow{f}^\top, \\
        \overleftarrow{u}^\top(t,0,x,y) &= \overleftarrow{g}^\top(t,x,y), \quad t,s \in [0,T], \quad x,y \in \mathbb{R}^d.
    \end{aligned}
    \right.
\end{equation}
where each $Q^I$ for $|I|\leq 2$ is a lower-triangular matrix, whose off-diagonal elements do not matter the subsequent analyses while the diagonal elements are either $A^I$ or $A^I+B^I$; specifically, the coefficients in front of $\partial_Iu(t,s,x,y)$ are $A^I+B^I$ while all other coefficients related to $\partial_Iu_{x_i}(t,s,x,y)$ are $A^I$. Consequently, by the classical theory of PDE systems \cite{Friedman1964,Ladyzhanskaya1968,Eidelman1969}, the differential operator $D^\prime u:=u_s-\sum Q^I\partial_Iu$ admits a fundamental solution $Z(s,\tau,y,\xi;t,x)$, which is ensured by the uniformly ellipticity conditions \eqref{Uniformellipticitycondition1}-\eqref{Uniformellipticitycondition2} of $A^I$ and $A^I+B^I$.

After showing a variety of equations/systems, we are ready to prove the Schauder prior estimate of solutions to nonlocal linear PDE \eqref{NonlocalLinearPDE}.
\begin{theorem} \label{Priorestimate}
  Suppose that $u$ is a solution of \eqref{NonlocalLinearPDE} (i.e. \eqref{ExpandedlinearPDE}) in $\Omega^{(2+\alpha)}_{[0,T]}$. Then we have 
  \begin{enumerate}
      \item $\overline{u}$ and $\overleftarrow{u}$ solve \eqref{Systemforoverlineu} and \eqref{Systemforoverleftarrowu} on $[0,T]^2\times\mathbb{R}^{d;d}$, respectively; 
      \item there exists a constant $C$ depending only on $\lambda$, $\alpha$, $d$, $T$, $\lVert A^I\rVert^{(\alpha)}_{[0,T]}$, and $\lVert B^I\rVert^{(\alpha)}_{[0,T]}$ such that 
      \begin{equation} \label{Schauderestimate}
    \| u \|^{(2+\alpha)}_{[0,T]} \leq C \left( \| f \|^{(\alpha)}_{[0,T]} + \| g \|^{(2+\alpha)}_{[0,T]} \right). 
\end{equation}
  \end{enumerate}
\end{theorem}
\begin{proof}
	The first claim is straightforward as it follows by our introductions of the systems \eqref{Systemforoverlineu} and \eqref{Systemforoverleftarrowu} before. Next, we focus on the proof of the second claim. 
	
	We first show that the inequality \eqref{Schauderestimate} holds for a suitably small $\delta\in[0,T]$ and then the conclusion can be extended to the case of $\delta=T$. According to the classical theory of parabolic system \cite{Friedman1964,Ladyzhanskaya1968,Eidelman1969,Lunardi1995}, for any fixed $(t,x)\in[0,\delta]\times\mathbb{R}^d$ and system \eqref{Systemforoverlineu}, there exists a constant $C>0$ such that 
    \begin{align}
    \Big|\overline{u}(t,s,x,y)\Big|^{(2+\alpha)}_{(s,y)\in[0,\delta]\times\mathbb{R}^d} &\leq C \bigg( \sum_{|I|\leq 2} \Big|\partial_I u(s,s,x,y)|_{x=y}\Big|^{(\alpha)}_{(s,y)\in[0,\delta]\times\mathbb{R}^d} \label{Estimateofoverlineu} \\
    &\quad + \Big|\overline{f}(t,s,x,y)\Big|^{(\alpha)}_{(s,y)\in[0,\delta]\times\mathbb{R}^d} + \Big|\overline{g}(t,s,x,y)\Big|^{(2+\alpha)}_{(s,y)\in[0,\delta]\times\mathbb{R}^d} \bigg). \nonumber
\end{align}
	
	Next, we estimate $\big|\partial_I u(s,s,x,y)|_{x=y}\big|^{(\alpha)}_{(s,y)\in[0,\delta]\times\mathbb{R}^d}$ for $|I|=0,1,2$. In addition to the estimates of $\big|\partial_I u(s,s,x,y)|_{x=y}\big|^{(0)}_{(s,y)\in[0,\delta]\times\mathbb{R}^d}$, we need to evaluate the difference between $\partial_I u(s,s,x,y)|_{x=y}$ and $\partial_I u(s^\prime,s^\prime,x,y^\prime)|_{x=y^\prime}$ for any $0\leq s<s^\prime\leq\delta$ and $y,y^\prime\in\mathbb{R}^d$ with $0<|y-y^\prime|\leq 1$. It is obvious that the evaluation requires not only $\partial_Iu$ but also the partial derivatives $\partial_Iu_t$ and $\partial_Iu_x$. All of them are characterized by $\eqref{Systemforoverleftarrowu}$ for $\overleftarrow{u}$. As usual (without loss of generality), we assume that $g=0$; otherwise, we consider $L'v:=f-Lg$ with $v|_{s=0}=0$ noting that the problems of \eqref{NonlocalLinearPDE}, \eqref{Systemforoverlineu}, and \eqref{Systemforoverleftarrowu} are all of linear-type. By the classical theory of parabolic systems, the vector-valued classical solution $\overleftarrow{u}$ of \eqref{Systemforoverleftarrowu} can be represented as 
    \begin{equation} \label{Integralforoverleftarrowu}
    \begin{split}
        \overleftarrow{u}(t,s,x,y) &= \int^s_0 d\tau \int_{\mathbb{R}^d} Z(s,\tau,y,\xi;t,x) \sum_{|I|\leq 2} \overleftarrow{B}^I(t,\tau,x,\xi) \mathcal{I}^I\left[\frac{\partial u}{\partial t}, \frac{\partial u}{\partial x}\right](t, \tau, x, \xi) \, d\xi \\
        &\quad + \int^s_0 d\tau \int_{\mathbb{R}^d} Z\overleftarrow{f}(t,\tau,x,\xi) \, d\xi
    \end{split}
\end{equation}
    where $Z$ is the fundamental solution for $D^\prime u := u_s - \sum Q^I \partial_I u$. Given $B, f \in \Omega^{(\alpha)}_{[0,T]}$ and $u \in \Omega^{(2+\alpha)}_{[0,T]}$, the derivatives for $|I|=1,2$ satisfy:
    \begin{equation} \label{Integralforpartialoverleftarrowu} 
    \begin{split}
        \partial_I\overleftarrow{u}(t,s,x,y) &= \int^s_0 d\tau \int_{\mathbb{R}^d} \partial_I Z \sum_{|I|\leq 2} \left[ \left(\overleftarrow{B}^I\mathcal{I}^I\right)(t,\tau,x,\xi) - \left(\overleftarrow{B}^I\mathcal{I}^I\right)(t,\tau,x,y) \right] \, d\xi \\
        &\quad + \int^s_0 \left( \int_{\mathbb{R}^d} \partial_I Z \, d\xi \right) \sum_{|I|\leq 2} \left(\overleftarrow{B}^I\mathcal{I}^I\right)(t,\tau,x,y) \, d\tau \\
    \end{split}
\end{equation}
\begin{equation*} 
    \begin{split}
        &\quad + \int^s_0 d\tau \int_{\mathbb{R}^d} \partial_I Z \left[ \overleftarrow{f}(t,\tau,x,\xi) - \overleftarrow{f}(t,\tau,x,y) \right] \, d\xi \\
        &\quad + \int^s_0 \left( \int_{\mathbb{R}^d} \partial_I Z \, d\xi \right) \overleftarrow{f}(t,\tau,x,y) \, d\tau
    \end{split}
\end{equation*}
	Generally speaking, in order to obtain $\big|\partial_I u(s,s,x,y)|_{x=y}\big|^{(\alpha)}_{(s,y)\in[0,\delta]\times\mathbb{R}^d}$ for $|I|=0,1,2$, we need to evaluate not only the absolute value $\big|\partial_I u(s,s,x,y)|_{x=y}\big|^{(0)}_{(s,y)\in[0,\delta]\times\mathbb{R}^d}$ but also the difference
    \begin{equation} \label{Analysisofdifference}
    \begin{split}
        & |\partial_I u(s^\prime,s^\prime,x,y^\prime)|_{x=y^\prime} - \partial_I u(s,s,x,y)|_{x=y} | \\
        &\leq |\partial_I u(s^\prime,s^\prime,x,y^\prime)|_{x=y^\prime} - \partial_I u(s,s^\prime,x,y^\prime)|_{x=y} | + |\partial_I u(s,s^\prime,x,y^\prime)|_{x=y} - \partial_I u(s,s,x,y)|_{x=y} | \\
        &\leq \left| \partial_I \left( \frac{\partial u}{\partial t}, \frac{\partial u}{\partial x} \right) (\eta_t, s^\prime, x, y^\prime) \big|_{x=\eta_x} \right| \big( |s^\prime-s| + |y^\prime-y| \big) \\
    \end{split}
\end{equation}
\begin{equation*} 
    \begin{split}
        &\quad + \frac{|\partial_I u(s,s^\prime,x,y^\prime)|_{x=y} - \partial_I u(s,s,x,y)|_{x=y} |}{|s^\prime-s|^\frac{\alpha}{2} + |y^\prime-y|^\alpha} \big( |s^\prime-s|^\frac{\alpha}{2} + |y^\prime-y|^\alpha \big)
    \end{split}
\end{equation*}
	for any $0\leq s<s^\prime\leq\delta$ and $y$, $y^\prime\in\mathbb{R}^d$ with $0<|y-y^\prime|\leq 1$, where $\eta=(\eta_t,\eta_x)=(1-c)(s,y)+c(s^\prime,y^\prime)$ for some $c\in(0,1)$ for the mean value theorem in several variables. We denote by $\rho$ the parabolic distance $\sqrt{(s^\prime-s)+|y-y^\prime|^2}$ between $(s,y)$ and $(s^\prime,y^\prime)$. Hence, we need to estimate the eight terms $E_i$ ($i=1,2,\cdots,8$) in Table \ref{TableforlinearPDE}. The estimation is standard but lengthy and is therefore given in Appendix \ref{App:A}. 
	
	\begin{table}[!ht] \label{TableforlinearPDE} 
		\centering
		\begin{tabular}{cccc}
			\toprule
			\multicolumn{2}{l}{\quad\quad~ Estimates of $\partial_Iu(s,s,x,y)|_{x=y}$} & \\
			\cmidrule{1-1}
			$0\leq s<s^\prime\leq \delta$, $0<|y-y^\prime|\leq 1$ & & $|I|=0$ & $|I|=1,2$ \\
			\midrule
			$\big|\partial_Iu(s,s,x,y)|_{x=y}\big|$ & & $E_1$ & $E_2$ \\
			$\big|\partial_I\left(u_t,u_x\right)(\eta_t,s^\prime,x,y^\prime)|_{x=\eta_x}\big|$ & & $E_3$ & $E_4$ \\
			\cmidrule{2-4}
			{\multirow{2}{*}{$\big|\partial_I u(s,s^\prime,x,y^\prime)|_{x=y}-\partial_I u(s,s,x,y)|_{x=y}\big|$}} & $s\leq \rho^2$ & $E_5$ & $E_6$ \\
			& $s> \rho^2$ & $E_7$ & $E_8$ \\
			\bottomrule
		\end{tabular}
		\caption{H\"{o}lder regularities of $\partial_Iu(s,s,x,y)|_{x=y}$ in $(s,y)$} 
		\label{tab:table1} 
	\end{table} 
	
	The estimates of $E_i$-terms and \eqref{Estimateofoverlineu} imply that for a suitably small $\delta\in[0,T]$ and any fixed $(t,x)\in[0,\delta]\times\mathbb{R}^d$,
    \begin{equation} \label{Estimateofu} 
    \begin{split}
        \Big|\overline{u}(t,s,x,y)\Big|^{(2+\alpha)}_{(s,y)\in[0,\delta]\times\mathbb{R}^d} & \leq \frac{1}{2}[\overrightarrow{u}]^{(2+\alpha)}_{[0, \delta]} + C \left( \| f \|^{(\alpha)}_{[0,T]} + \| g \|^{(2+\alpha)}_{[0,T]} \right) \\
        & \leq \frac{1}{2}[\overline{u}]^{(2+\alpha)}_{[0, \delta]} + C \left( \| f \|^{(\alpha)}_{[0,T]} + \| g \|^{(2+\alpha)}_{[0,T]} \right).
    \end{split}
\end{equation}
	Thanks to the integral representation of $\mathcal{I}^I[u_t,u_x]$ in \eqref{IntegralRep}, we can set the coefficient $\frac{1}{2}$ in front of $[\overrightarrow{u}]^{(2+\alpha)}_{[0,\delta]}$ in \eqref{Estimateofu}
	by choosing a small enough $\delta$. Consequently, we have
	\begin{equation} \label{delta-Estimateofu}
		\lVert u\rVert^{(2+\alpha)}_{[0,\delta]}\leq C\left(\lVert f\rVert^{(\alpha)}_{[0,\delta]}+\lVert g\rVert^{(2+\alpha)}_{[0,\delta]}\right). 
	\end{equation}
	
	To complete the proof, we ought to show that the small $\delta$ in \eqref{delta-Estimateofu} can be extended to an arbitrarily large $T<\infty$. It can simply follow the proof of Theorem 3.3 in \cite{Lei2021}. Essentially, since we can obtain prior estimates similar to \eqref{delta-Estimateofu} in any subinterval, we can extend the horizon by solving the same PDE with the initial condition updated by the upper bound of the current interval. It follows that \eqref{Schauderestimate} for any finite $T$ holds as well.
\end{proof}


\subsection{Global Well-Posedness of Nonlocal Linear PDEs}
By the Schauder prior estimate for the solutions to \eqref{NonlocalLinearPDE} in $\Omega^{(2+\alpha)}_{[0,T]}$ in Theorem \ref{Priorestimate}, we apply the method of continuity to prove the global well-posedness of \eqref{NonlocalLinearPDE}. To this end, we ought to show the global solvability of a simplied version of \eqref{NonlocalLinearPDE} with constant coefficients and $(t,x)$-independent variable coefficients:
\begin{equation} \label{SimplfiednonlocallinearPDE}
    \left\{
    \begin{aligned}
        L_0u(t,s,x,y) &= f(t,s,x,y), \\
        u(t,0,x,y) &= g(t,x,y), \quad t,s \in [0,T], \quad x,y \in \mathbb{R}^d.
    \end{aligned}
    \right.
\end{equation}
with a nonlocal parabolic differential operator of the form
\begin{equation} \label{Nonlocallinearoperatorwithconstantcoefficients} 
    L_0u := u_s(t,s,x,y) - \sum_{|I|\leq 2} a^I(s,y) \partial_I u(t,s,x,y) + \sum_{|I|\leq 2} \partial_I b^I(s,y) u(s,s,x,y) \big|_{x=y}
\end{equation}
where both of $a^I$ and $b^I$ belong to $\Omega^{(\alpha)}_{[0,T]}$ and satisfy the uniformly ellipticity conditions \eqref{Uniformellipticitycondition1}-\eqref{Uniformellipticitycondition2}. We aim to establish the global existence, uniqueness, and stability of solutions to \eqref{SimplfiednonlocallinearPDE} and from which, we make use of the method of continuity and the Schauder prior estimate to transfer the well-posedness results to \eqref{NonlocalLinearPDE}.

\begin{theorem} \label{Well-posednessofL0} 
  If $f\in\Omega^{(\alpha)}_{[0,T]}$ and $g\in\Omega^{(2+\alpha)}_{[0,T]}$, then the simplified nonlocal linear PDE \eqref{SimplfiednonlocallinearPDE} admits a unique solution in $\Omega^{(2+\alpha)}_{[0,T]}$.  
\end{theorem}
\begin{proof}
  In order to show the global existence of solutions of \eqref{SimplfiednonlocallinearPDE} in $\Omega^{(2+\alpha)}_{[0,T]}$, we directly construct a regular enough solution $w$ for it by studying the following decoupled system \eqref{SystemforW} of PDEs for the unknown vector-valued function $W(t,s,x,y)=(w,w_0,w_1,\cdots,w_d,w_{11},w_{12},\cdots,w_{dd})(t,s,x,y)$,    
        
        
\begin{equation} \label{SystemforW}
    \left\{
    \begin{aligned}
        W^\top_s(t,s,x,y) &= \sum_{|I|\leq 2} \widetilde{ Q}^I(s,y) \partial_I W^\top(t,s,x,y) \\
        &\quad + \sum_{|I|\leq 2} (B^I, O)^\top(s,y) \mathcal{I}^I[w_0, w_1, \dots, w_d](t,s,x,y) + \overline{f}^\top, \\[1ex]
        W^\top(t,0,x,y) &= \overline{g}^\top(t,x,y), \quad t,s \in [0,T], \quad x,y \in \mathbb{R}^d.
    \end{aligned}
    \right.
\end{equation}
where $O$ is $(d^2+d+1)$-dimensional zero vector and $\widetilde{Q}^I$ is a lower triangular matrix, whose main diagonal elements are $a^I(s,y)$ or $(a^I+b^I)(s,y)$. The coefficients in front of $\partial_Iw$ is $(a^I+b^I)(s,y)$ and other coefficients are all $a^I(s,y)$. The construction of the system \eqref{SystemforW} for $W$ is inspired by \eqref{Systemforoverlineu} and \eqref{Systemforoverleftarrowu}. 

Next, we will show that 
\begin{enumerate}
    \item[1.)] the system \eqref{SystemforW} admits a unique classical solution $W$; 
    \item[2.)] $(\partial_Iw_0,\partial_Iw_1,\cdots,\partial_Iw_d)$ of $W$ is a conservative vector field, the potential function of which is just $\partial_Iw$. Furthermore, $\partial_Iw_0=\partial_Iw_t$, $\partial_Iw_i=\partial_Iw_{x_i}$, and ${\partial_Iw_{ij}}=\partial_Iw_{x_ix_j}$ for $i,j=1,\ldots,d$ and $|I|=0,1,2$; 
    \item[3.)] the first component $w$ of $W$ solves the simplified nonlocal linear PDE \eqref{SimplfiednonlocallinearPDE}; 
    \item[4.)] the estimate $[W]^{(2+\alpha)}_{[0,T]}<\infty$ holds such that $w\in\Omega^{(2+\alpha)}_{[0,T]}$; 
    \item[5.)] the nonlocal PDE \eqref{SimplfiednonlocallinearPDE} is solvable in $\Omega^{(2+\alpha)}_{[0,T]}$. 
\end{enumerate}


\textbf{1.)} We are to prove that the system \eqref{SystemforW} admits a unique solution $W$. Note that $\eqref{SystemforW}$ is a decoupled system and the PDEs of $(w_0,w_1,\cdots,w_d,w_{11},\cdots,w_{dd})$ are all classical equations. Hence, by the classical PDE theory \cite{Friedman1964,Ladyzhanskaya1968,Eidelman1969,Lunardi1995}, we can find a unique classical solution $(w_0,w_1,\cdots,w_{dd})$ satisfying 
\begin{equation*}
    \sup_{(t,x)\in[0,T]\times\mathbb{R}^d} \left| (w_0,w_1,\dots,w_{dd})(t,s,x,y) \right|^{(2+\alpha)}_{(s,y)\in[0,T]\times\mathbb{R}^d} \leq C \left( \| f \|^{(\alpha)}_{[0,T]} + \| g \|^{(2+\alpha)}_{[0,T]} \right) < \infty.
\end{equation*}

Moreover, after solving for $(w_0,w_1,\cdots,w_{d})$, the nonlocal term $\mathcal{I}^I\left[w_0,w_1,\cdots,w_d\right]$ in \eqref{SystemforW} is known as well. Consequently, the nonlocal PDE of $w$ reduces to a classical equation. Considering the boundedness of $(w_0,w_1,\cdots,w_{d})$ and the integral structures of $\mathcal{I}^I\left[w_0,w_1,\cdots,w_d\right]$, there exists a unique classical solution $w$ although it possibly increases with $x$, $y\to\infty$. We will show that $w$ is bounded later. Now, we have shown that the decoupled system \eqref{SystemforW} exists a unique classical solution $W$.


\textbf{2.)} We are to prove that $(\partial_Iw_0,\partial_Iw_1,\cdots,\partial_Iw_d)$ of $W$ is a conservative vector field, the potential function of which is just $\partial_Iw$. Here, we only consider the case $|I|=0$ while the other two cases for $|I|=1,2$ can be proved similarly. For \eqref{SystemforW}, it is clear that the solution $(w_0,w_1,\cdots,w_d)$ can be represented with a fundamental solution in an integral form 
\begin{equation} \label{Conservativevectorfield} 
    \begin{split}
        (w_0,w_1,\dots,w_d)(t,s,x,y) &= \int^s_0 d\tau \int_{\mathbb{R}^d} Z(s,\tau,y,\xi) (f_t, f_{x_1}, \dots, f_{x_d})(t,\tau,x,\xi) \, d\xi \\
        &\quad + \int_{\mathbb{R}^d} Z(s,0,y,\xi) (g_t, g_{x_1}, \dots, g_{x_d})(t,x,\xi) \, d\xi
    \end{split}
\end{equation}
where the real-valued fundamental solution $Z$ is independent of $(t,x)$ since $a^I=a^I(s,y)$. Next, in order to show that it is a conservative vector field, we need to prove that a line integral of the vector field $(w_0,w_1,\cdots,w_d)$ is path-independent. Let us consider any two paths $r_a(\theta)=(t^a(\theta),x^a(\theta))$ and $r_b(\theta)=(t^b(\theta),x^b(\theta))$ connecting between two fixed endpoints $(t^\prime,x^\prime)$ and $(t^{\prime\prime},x^{\prime\prime})$, both of which are parameterized by $\theta\in[0,1]$ such that $r_1(0)=r_2(0)=(t^\prime,x^\prime)$ and $r_1(1)=r_2(1)=(t^{\prime\prime},x^{\prime\prime})$. Then, we have  
\begin{equation*}
    \begin{split}
        & \int_{r_a}(w_0,w_1,\dots,w_d)\cdot dr \\
        &= \int^1_0(w_0,w_1,\dots,w_d)(t^a(\theta),s,x^a(\theta),y)\cdot\left(\frac{d t^a(\theta)}{d\theta},\frac{d x^a(\theta)}{d\theta}\right)^\top d\theta \\
        &= \int^s_0d\tau\int_{\mathbb{R}^d}Z(s,\tau,y,\xi) \int^1_0(f_t,f_x)(t^a(\theta),\tau,x^a(\theta),\xi)\cdot\left(\frac{d t^a(\theta)}{d\theta},\frac{d x^a(\theta)}{d\theta}\right)^\top d\theta \, d\xi \\
    \end{split}
\end{equation*}
\begin{equation*}
    \begin{split}
        &\quad + \int_{\mathbb{R}^d}Z(s,0,y,\xi)\int^1_0(g_t,g_x)(t^a(\theta),x^a(\theta),\xi)\cdot\left(\frac{d t^a(\theta)}{d\theta},\frac{d x^a(\theta)}{d\theta}\right)^\top d\theta \, d\xi \\
        &= \int^s_0d\tau\int_{\mathbb{R}^d}Z(s,\tau,y,\xi) \Big(f(t^{\prime\prime},\tau,x^{\prime\prime},\xi)-f(t^{\prime},\tau,x^{\prime},\xi)\Big) \, d\xi \\
        &\quad + \int_{\mathbb{R}^d}Z(s,0,y,\xi)\Big((g_t,g_x)(t^{\prime\prime},x^{\prime\prime},\xi)-(g_t,g_x)(t^{\prime},x^{\prime},\xi)\Big) \, d\xi \\
        &= \int_{r_b}(w_0,w_1,\dots,w_d)\cdot dr.
    \end{split}
\end{equation*}
which shows that the choice of paths between two points does not change the value of the line integral. Hence, we obtain that $(w_0,w_1,\cdots,w_d)$ is a conservative vector field. Similarly, $(\partial_Iw_0,\partial_Iw_1,\cdots,\partial_Iw_d)$ for $|I|=1,2$ is also a conservative vector field. 

From the claims above, there exist some (continuously differentiable) scalar fields $\phi^I$ (i.e. real-valued functions) such that $\nabla_{t,x}\phi^I=(\partial_Iw_0,\partial_Iw_1,\cdots,\partial_Iw_d)$. Next, we will prove that $\partial_Iw$ ($|I|=0,1,2$) are simply the corresponding potential functions, i.e. $\phi^I=\partial_Iw$. Since $(\partial_Iw_0,\partial_Iw_1,\cdots,\partial_Iw_d)$ is a conservative vector field, we will show that the nonlocal term $\mathcal{I}^I\left[w_0,w_1,\cdots,w_d\right](t,s,x,y)$ satisfies the following properties: 
\begin{equation} \label{Commutationproperty}
    \begin{split}
        \frac{\partial\mathcal{I}^I\left[w_0,w_1,\dots,w_d\right](t,s,x,y)}{\partial t} &= -\partial_{I}w_0(t,s,x,y), \\
        \frac{\partial\mathcal{I}^I\left[w_0,w_1,\dots,w_d\right](t,s,x,y)}{\partial x_k} &= -\partial_{I}w_k(t,s,x,y)
    \end{split}
\end{equation}
for any $k=1,2,\cdots,d$ and $|I|=0,1,2$. From the definition \eqref{Intergralrepresentations} of $\mathcal{I}^I$, the first equation of \eqref{Commutationproperty} is clear. As for the second equation, we can rearrange the order of $w_i$ in $\mathcal{I}^I\left[w_0,w_1,\cdots,w_d\right]$ such that the integral of $w_i$ appears in the first position. Thanks to the property of path-independence, we have 
\begin{equation*}  
    \begin{split}
        &-\mathcal{I}^I\left[w_0,w_1,\dots,w_d\right](t,s,x,y)  \\
        &= \int^t_s \partial_{I}w_0(\theta_t,s,x,y) \, d\theta_t + \int^{x_1}_{y_1} \partial_{I}w_1(s,s,\theta_1,x_2,\dots,x_d,y) \, d\theta_1 \\
        &\quad + \int^{x_2}_{y_2} \partial_{I}w_2(s,s,x_1,\theta_2,\dots,x_d,y) \big|_{x_1=y_1} \, d\theta_2 \\
        &\quad + \int^{x_3}_{y_3} \partial_I w_3(s,s,x_1,x_2,\theta_3,\dots,x_d,y) \big|_{\begin{subarray}{c} x_1=y_1 \\ x_2=y_2\end{subarray}} \, d\theta_3 \\
        &\quad \dots + \int^{x_d}_{y_d} \partial_{I}w_d(s,s,x_1,x_2,\dots,x_{d-1},\theta_d,y) \big|_{\begin{subarray}{c} x_i=y_i \\ i=1,2,\dots,d-1\end{subarray}} \, d\theta_d \\
        &= \phi^I(t,s,x,y) - \phi^I(s,s,y,y) \\
        &= \int^{x_k}_{y_k} \partial_Iw_k(t,s,x_1,x_2,\dots,x_{k-1},\theta_{k},x_{k+1},\dots,x_d,y) \, d\theta_{k} \\
        &\quad + \int^t_s \partial_Iw_0(\theta_t,s,x,y) |_{x_k=y_k} \, d\theta_t + \int^{x_1}_{y_1} \partial_I w_1(s,s,\theta_1,x_2,\dots,x_d,y) |_{x_k=y_k} \, d\theta_1 \\
        &\quad \dots + \int^{x_{k-1}}_{y_{k-1}} \partial_Iw_{k-1}(s,s,x_1,\dots,x_{k-2},\theta_{k-1},x_{k},\dots,x_d,y) |_{\begin{subarray}{c} x_i=y_i,x_k=y_k \\ i=1,2,\dots,k-2\end{subarray}} \, d\theta_{k-1} \\
        &\quad + \int^{x_{k+1}}_{y_{k+1}} \partial_Iw_{k-1}(s,s,x_1,\dots,x_{k},\theta_{k+1},x_{k+2},\dots,x_d,y) |_{\begin{subarray}{c} x_i=y_i \\ i=1,2,\dots,k\end{subarray}} \, d\theta_{k+1} \\
        &\quad \dots + \int^{x_d}_{y_d} \partial_{I}w_d(s,s,x_1,x_2,\dots,x_{d-1},\theta_d,y) \big|_{\begin{subarray}{c} x_i=y_i \\ i=1,2,\dots,d-1\end{subarray}} \, d\theta_d.
    \end{split}
\end{equation*}
which directly indicates that the second equation of \eqref{Commutationproperty} holds.

Next, we will show that the potential function of $(\partial_Iw_0,\partial_Iw_1,\cdots,\partial_Iw_d)$ is just $\partial_Iw$. Furthermore, $\partial_Iw_0=\partial_Iw_t$, $\partial_Iw_i=\partial_Iw_{x_i}$, and ${\partial_Iw_{ij}}=\partial_Iw_{x_ix_j}$ for $i,j=1,2,\cdots,d$ and $|I|=0,1,2$. Note that $\partial_Iw_i$ is the $(i+1)$-th component of $\partial_IW$ while $\partial_Iw_{x_i}$ is the partial derivative of the first component $\partial_I w$ of $\partial_I W$ with respect to $x_i$. Hence, it is not trivial to check if they are identical. 

With the differentiability of coefficients $a^I$ and $b^I$, the nonhomogeneous terms $\mathcal{I}^I\left[w_0,w_1,\cdots,w_d\right](t,s,x,y)$ and $f$, and the initial condition $g$ in $(t,x)$, the implicit function theorem guarantees that the solution $w$ of the first PDE of $\eqref{SystemforW}$ is also differentiable in $(t,x)$. Thanks to \eqref{Commutationproperty}, we first differentiate the first PDE of \eqref{SystemforW} for $w$ and then subtract the equation of \eqref{SystemforW} for $w_i$ from it. We find that the difference $w_{x_i}-w_i$ satisfies the following classical PDE 
\begin{equation*} 
    \left\{
    \begin{aligned}
        (w_{x_i}-w_i)(t,s,x,y) &= \sum_{|I|\leq 2}(A^I+B^I)(s,y)\partial_I(w_{x_i}-w_i)(t,s,x,y) \\
        (w_{x_i}-w_i)(t,0,x,y) &= 0, \quad t,s\in[0,T], \quad x,y\in\mathbb{R}^d.  
    \end{aligned}
    \right. 
\end{equation*}
By the classical PDE theory \cite{Friedman1964,Ladyzhanskaya1968,Eidelman1969,Lunardi1995}, we have $\partial_Iw_{x_i}=\partial_Iw_i$ for $|I|=0,1,2$. Hence, the potential function of $(\partial_Iw_0,\partial_Iw_1,\cdots,\partial_Iw_d)$ is just $\partial_Iw$. 


\textbf{3.)} We are to prove that $w$ solves the simplified nonlocal linear PDE \eqref{SimplfiednonlocallinearPDE}. Since $\partial_Iw_i=\partial_Iw_{x_i}$ for $i=0,1,2,\cdots,d$ and $|I|=0,1,2$, we replace all $\partial_Iw_i$ of the nonlocal terms $\mathcal{I}^I\left[w_0,w_1,\cdots,w_d\right]$ by $\partial_Iw_{x_i}$. Then the first PDE of \eqref{SystemforW} is exactly the simplified nonlocal linear PDE \eqref{SimplfiednonlocallinearPDE}. 


\textbf{4.)} We are to show that the estimate $[W]^{(2+\alpha)}_{[0,T]}<\infty$, i.e., $w\in\Omega^{(2+\alpha)}_{[0,T]}$. First of all, it is obvious that $[\overrightarrow{w}]^{(2+\alpha)}_{[0,T]}<\infty$ with the regularities of $f$ and $g$. Similar to the proof of the Schauder prior estimate \eqref{Schauderestimate}, by the simplified nonlocal PDE \eqref{SimplfiednonlocallinearPDE} and the system \eqref{SystemforW}, we have
\begin{equation*} 
    \begin{split}
        \Big|w(t,s,x,y)\Big|^{(2+\alpha)}_{(s,y)\in[0,T]\times\mathbb{R}^d} & \leq C\left([\overrightarrow{w}]^{(2+\alpha)}_{[0,T]}+\lVert f\rVert^{(\alpha)}_{[0,T]}+\lVert g\rVert^{(2+\alpha)}_{[0,T]}\right) <\infty
    \end{split}
    \end{equation*}
for any $(t,x)\in[0,T]\times\mathbb{R}^d$, which implies $[w]^{(2+\alpha)}_{[0,T]}<\infty$. Hence, we have $[W]^{(2+\alpha)}_{[0,T]}<\infty$ and $w\in\Omega^{(2+\alpha)}_{[0,T]}$. Furthermore, from \textbf{3.)} and \textbf{4.)}, the nonlocal PDE \eqref{SimplfiednonlocallinearPDE} is solvable in $\Omega^{(2+\alpha)}_{[0,T]}$. 


\textbf{5.)} Finally, for the uniqueness and stability of solutions of \eqref{SimplfiednonlocallinearPDE}, both of them come directly from the Schauder estimate \eqref{Schauderestimate}. Suppose that $u_1$, $u_2\in\Omega^{(2+\alpha)}_{[0,T]}$ are two solutions of \eqref{SimplfiednonlocallinearPDE}, then we have
\begin{equation*} 
    \left\{
    \begin{aligned}
        L_0(u_1-u_2)(t,s,x,y) &= 0, \\
        (u_1-u_2)(t,0,x,y) &= 0, \quad t,s\in[0,T], \quad x,y\in\mathbb{R}^d. 
    \end{aligned}
    \right. 
\end{equation*}
which shows that $u_1=u_2$.

Similarly, we can also show that the map from data $(f,g)$ to solutions of \eqref{SimplfiednonlocallinearPDE} is continuous in the $\Omega^{(l)}_{[0,T]}$-topology. Specifically, let $u$ and $\widehat{u}$ correspond to $(f,g)$ and $(\widehat{f},\widehat{g})$ satisfying the
assumptions of Theorem \ref{Well-posednessofL0}, respectively. Then, we have
\begin{equation} \label{StabilityofnonlocallinearPDE} 
      \lVert u-\widehat{u}\rVert^{(2+\alpha)}_{[0,T]}\leq C\left(\lVert f-\widehat{f}\rVert^{(\alpha)}_{[0,T]}+\lVert g-\widehat{g}\rVert^{(2+\alpha)}_{[0,T]}\right). 
\end{equation}

 
With the claims \textbf{1.)}-\textbf{5.)}, the proof is completed. 
\end{proof}

With the Schauder estimate \eqref{Schauderestimate} and the well-posedness of the simplified version \eqref{SimplfiednonlocallinearPDE} of \eqref{NonlocalLinearPDE}, we are ready to prove the global solvability of \eqref{NonlocalLinearPDE} by the method of continuity. 

\begin{theorem} \label{Well-posednessofL} 
  If $f\in\Omega^{(\alpha)}_{[0,T]}$ and $g\in\Omega^{(2+\alpha)}_{[0,T]}$, then the nonlocal linear PDE \eqref{NonlocalLinearPDE} admits a unique solution in $\Omega^{(2+\alpha)}_{[0,T]}$.  
\end{theorem}
\begin{proof}
  Since the problem \eqref{NonlocalLinearPDE} is of linear type, we assume $g=0$ without loss of generality. Consider the family of equations: 
  $$L_\tau u:=(1-\tau)L_0 u+\tau L_1 u$$
  where $L_1u:=Lu$. It is clear that 
  \begin{equation}
      \lVert L_\tau u\rVert^{(\alpha)}_{[0,T]}\leq C\lVert u\rVert^{(2+\alpha)}_{[0,T]}
  \end{equation}
where $C$ is a positive constant depending only on $d$, $\alpha$, $T$, $\lVert A^I\rVert^{(\alpha)}_{[0,T]}$, and $\lVert B^I\rVert^{(\alpha)}_{[0,T]}$. Hence, for each $\tau\in[0,1]$, the nonlocal parabolic $L_\tau$ is a bounded linear operator from $B:=\left\{u\in\Omega^{(2+\alpha)}_{[0,T]}: u(t,0,x,y)=0\right\}$ to $V:=\Omega^{(\alpha)}_{[0,T]}$. We know that $L_0$ is solvable (i.e., $L_0$ is surjective) by Theorem \ref{Well-posednessofL0}. Moreover, there exists a constant $C$ such that the following a-priori estimate holds for all $u\in B$ and $\tau\in[0,1]$ 
\begin{equation}
      \lVert u\rVert^{(2+\alpha)}_{[0,T]}\leq C\lVert L_\tau u\rVert^{(\alpha)}_{[0,T]}, 
  \end{equation}
since $u$ solves the equation with $L_\tau u$ as the nonhomogeneous term. By the method of continuity, $L_1$ is also solvable (i.e., surjective). Furthermore, the uniqueness directly follows from the Schauder estimate for the homogeneous, linear, and strongly parabolic PDE with zero initial value, which is satisfied by the difference of any two solutions in $\Omega^{(2+\alpha)}_{[0,T]}$ to the equation. For the stability of solutions of \eqref{NonlocalLinearPDE}, one can refer to the counterpart of Theorem \ref{Well-posednessofL0}.    
\end{proof}


Theorem \ref{Well-posednessofL} implies the following properties of $L$.
\begin{corollary}
   Given $g(t,x,y)\in\Omega^{(2+\alpha)}_{[0,T]}$, the nonlocal operator $L$ from $\big\{u\in\Omega^{(2+\alpha)}_{[0,T]}:u|_{s=0}=g\big\}$ to $\Omega^{(\alpha)}_{[0,T]}$, defined in \eqref{Nonlocallinearoperator}, is linear, bijective, continuous, and bounded. 
\end{corollary}


\section{Nonlocal Fully Nonlinear PDE} \label{Sec:Nonlinear}
In this section, we make use of the linearization method and Banach's fixed point theorem to prove the local existence, uniqueness, and stability of solutions to nonlocal fully nonlinear PDE: 
\begin{equation} \label{NonlocalfullynonlinearPDE} 
    \left\{
    \begin{aligned}
        u_s(t,s,x,y) &= F\big(t,s,x,y,\left(\partial_I u\right)_{|I|\leq 2}(t,s,x,y),  \left(\partial_I u\right)_{|I|\leq 2}(s,s,x,y)\big|_{x=y}\big), \\
        u(t,0,x,y) &= g(t,x,y), \quad t,s\in[0,T], \quad x,y\in\mathbb{R}^d. 
    \end{aligned}
    \right. 
\end{equation}
where the mapping (nonlinearity) $F$ could be nonlinear with respect to all its arguments. With the local well-posedness, we then extend the results to the largest possible time horizon, resulting in the maximally defined solution. Finally, in regards of the global solvability, we will show that it holds if a very sharp a-priori estimate is available. Especially for a special case of \eqref{NonlocalfullynonlinearPDE}, called nonlocal quasilinear PDEs, the global solvability can be achieved. Furthermore, the results will be extended to a broader function space to enhance their applicability in practical settings.

\subsection{Local Well-posedness of Fully Nonlinear PDEs}
To take advantage of the results of nonlocal linear PDEs in Section \ref{Sec:Linear}, we impose some regularity assumptions on the nonlinearity $F$ and the initial data $g$. In addition to $g\in\Omega^{(2+\alpha)}_{[0,T]}$, it is required that the nonlinear mapping $(t,s,x,y,z)\to F(t,s,x,y,z)$ is defined in $\Pi=[0,T]^2\times\mathbb{R}^{d;d}\times B(\overline{z},R_0)$ for a positive constant $R_0$, where $\overline{z}\in\mathbb{R}\times\mathbb{R}^{d}\times\mathbb{R}^{d^{2}}\times\mathbb{R}\times\mathbb{R}^{d}\times\mathbb{R}^{d^{2}}$ and satisfies that 
\begin{enumerate} 
    \item \textbf{(Uniformly ellipticity condition)} for any $\xi=(\xi_1,\dots,\xi_d)^\top\in\mathbb{R}^d$, there exists a constant $\lambda>0$ such that     
    \begin{align}
    \sum_{|I|=2} \partial_I F(t,s,x,y,z)\xi^I &\geq \lambda |\xi|^2, \label{UniformellipticityconditionofF1} \\
    \sum_{|I|=2} (\partial_I F + \partial_I \overline{F})(t,s,x,y,z)\xi^I &\geq \lambda |\xi|^2 \label{UniformellipticityconditionofF2} 
\end{align}
    hold uniformly with respect to $(t,s,x,y,z)\in\Pi$; 
    \item \textbf{(Locally H\"{o}lder continuity)} for every $\delta\geq 0$ and $z\in B(\overline{z},R_0)$, there exists a constant $K>0$ such that
    \begin{equation} \label{HoldercontinuityofF}
        \sup_{(t,x,z)}\left\{\left| \mathcal{F}(t,\cdot,x,\cdot,z)\right|^{(\alpha)}_{[0,\delta]\times\mathbb{R}^d}\right\}=K;
    \end{equation}
    \item \textbf{(Locally Lipschitz continuity)} for any $(t,s,x,y,z_1),(t,s,x,y,z_2)\in\Pi$, there exists a constant $L>0$ such that
    \begin{equation} \label{LipschitzcontinuityofF} 
        |\mathcal{F}(t,s,x,y,z_1)-\mathcal{F}(t,s,x,y,z_2)|\leq L|z_1-z_2|, 
    \end{equation}
\end{enumerate}
where $\partial_I F$ denotes the partial derivative of $F$ with respect to $\partial_Iu (t,s,x,y)$ while $\partial_I \overline{F}$ denotes the derivative of $F$ with respect to $\partial_I u(s,s,x,y)|_{x=y}$. Moreover, the generic notation $\mathcal{F}$ in the conditions of \eqref{HoldercontinuityofF} and \eqref{LipschitzcontinuityofF} represents $F$ itself and some of its first-, second, third-order partial derivatives, which are indicated by ``$\surd$" or ``$\overline{\surd}$" in Tables \ref{tab:table2}, \ref{tab:table3}, and \ref{tab:table4}. For these second-order derivatives denoted by ``$\overline{\surd}$" in Table \ref{tab:table3}, we require further regularities listed in Table \ref{tab:table4}, where $\partial^3_{\overline{\mathcal{X}}\overline{\mathcal{Y}}Z}F$ represents the first partial derivative of $\partial^2_{\mathcal{X}\mathcal{Y}}F$ with ``$\overline{\surd}$" in Table \ref{tab:table3} with respect to the argument $\mathcal{Z}$. 
\begin{table}[!ht] 
    \centering
    \small
    \setlength{\tabcolsep}{3.5pt}
    
    \begin{minipage}[t]{0.48\textwidth}
        \centering
        \begin{tabular}{c| c c c c c c} 
            \hline
            $\mathcal{X}$ & $t$ & $s$ & $x$ & $y$ & $\partial_I u$ & $\partial_I u \big| \substack{\scalebox{0.75}{$t=s$} \\ \scalebox{0.75}{$x=y$}}$ \\ 
            \hline 
            $\partial_\mathcal{X} F$ & $\surd$ & & $\surd$ & & $\surd$ & $\surd$ \\ 
            \hline 
        \end{tabular}
        \caption{First-order derivatives of $F$} 
        \label{tab:table2}
    \end{minipage}
    \hfill
    \begin{minipage}[t]{0.48\textwidth}
        \centering
        \begin{tabular}{c| c c c c c c} 
            \hline
            $\mathcal{Z}$ & $t$ & $s$ & $x$ & $y$ & $\partial_I u$ & $\partial_I u \big| \substack{\scalebox{0.75}{$t=s$} \\ \scalebox{0.75}{$x=y$}}$ \\ 
            \hline 
            $\partial^3_{\overline{\mathcal{X}}\overline{\mathcal{Y}}\mathcal{Z}}F$ & & & $\surd$ & & $\surd$ & \\ 
            \hline 
        \end{tabular}
        \caption{Third-order derivatives of $F$} 
        \label{tab:table4} 
    \end{minipage}

    \vspace{0.4cm}

    \begin{minipage}{1\textwidth}
        \centering
        \begin{tabular}{c| c c c c c c}
            \hline
            \diagbox[width=6em]{$\mathcal{X}$}{$\partial^2 F$}{$\mathcal{Y}$} & $t$ & $s$ & $x$ & $y$ & $\partial_I u$ & $\partial_I u \big| \substack{\scalebox{0.75}{$t=s$} \\ \scalebox{0.75}{$x=y$}}$ \\ 
            \hline 
            $t$ &  &  &  &  & $\surd$ & $\surd$ \\ 
            $s$ &  &  &  &  &  & \\
            $x$ &  &  & $\surd$ &  &  $\overline{\surd}$ & $\overline{\surd}$ \\
            $y$ &  &  &  &  &  & \\
            $\partial_I u$ & $\surd$ &  & $\overline{\surd}$ &  & $\overline{\surd}$ & $\overline{\surd}$ \\
            $\partial_I u \big| \substack{\scalebox{0.75}{$t=s$} \\ \scalebox{0.75}{$x=y$}}$ & $\surd$ &  & $\overline{\surd}$ &  & $\overline{\surd}$ & \\
            \hline 
        \end{tabular}
        \caption{Second-order derivatives of $F$} 
        \label{tab:table3}
    \end{minipage}
\end{table}

After introducing the regularities for $F$ and $g$, we are now in position of stating a local existence and uniqueness result for the nonlocal fully nonlinear PDE \eqref{NonlocalfullynonlinearPDE}. 
\begin{theorem}\label{Localwell-posednessoffullynonlinearPDE} 
	Suppose that $F$ satisfies the conditions \eqref{UniformellipticityconditionofF1}-\eqref{LipschitzcontinuityofF}, $g\in\Omega^{{(2+\alpha)}}_{[0,T]}$, and that the range of $(\partial_I g(t,x,y),\partial_I g(s,x,y)|_{x=y})$ is contained in the ball centered at $\overline{z}$ with radius $R_0/2$ for a positive constant $R_0$. Then, there exist $\delta>0$ and a unique $u\in\Omega^{{(2+\alpha)}}_{[0,\delta]}$ satisfying \eqref{NonlocalfullynonlinearPDE} in $[0,\delta]^2\times\mathbb{R}^{d;d}$. 
\end{theorem} 
\begin{proof}
We adopt the linearization method and Banach's fixed point argument to prove the local well-posedness of nonlocal fully nonlinear PDE. Overall speaking, we search for the solution of \eqref{NonlocalfullynonlinearPDE} as a fixed point of the operator $\Lambda$, defined by $\Lambda(u)=U$ over the space
    \begin{equation*}
    \mathcal{U} = \left\{ u \in \Omega^{(2+\alpha)}_{[0,\delta]} : u(t,0,x,y) = g(t,x,y), \ \| u - g \|^{(2+\alpha)}_{[0,\delta]} \leq R \right\}  
\end{equation*}
	for two constants $\delta$ and $R$ (determined later), where $U$ is the solution to
    \begin{equation} \label{DefinitionofLamda} 
    \left\{
    \begin{aligned}
        U_s(t,s,x,y) &= \mathcal{L}U + F\big(t,s,x,y,\left(\partial_I u\right)_{|I|\leq 2}(t,s,x,y), \left(\partial_I u\right)_{|I|\leq 2}(s,s,x,y)|_{x=y}\big) - \mathcal{L}u, \\
        U(t,0,x,y) &= g(t,x,y), \quad t,s\in [0,\delta], \quad x,y\in\mathbb{R}^d. 
    \end{aligned}
    \right. 
\end{equation}
	in which
    \begin{equation} 
    \begin{split}
        \mathcal{L}u(t,s,x,y) &= \sum_{|I|\leq 2} \partial_I F \big(t,0,x,y,\theta_0(t,x,y)\big) \cdot \partial_I u(t,s,x,y) \\
        &\quad + \sum_{|I|\leq 2} \partial_I \overline{F} \big(t,0,x,y,\theta_0(t,x,y)\big) \cdot \partial_I u(s,s,x,y) \big|_{x=y}
    \end{split}
\end{equation}
	with $\theta_0(t,x,y):=\big(\left(\partial_I g\right)_{|I|\leq 2}(t,x,y),  \left(\partial_I g\right)_{|I|\leq 2}(0,x,y)|_{x=y}\big)$. Note that the partial derivative $\partial_I F\big(t,0,x,y,\theta_0(t,x,y)\big)$
	is meant to be evaluated at $(t,0,x,y,\theta_0(t,x,y))$, i.e. $\big(t,0,x,y,\left(\partial_I g\right)_{|I|\leq 2}(t,x,y),  \left(\partial_I g\right)_{|I|\leq 2}(0,x,y)\big)$. Similarly, the same convention applies to $\partial_I \overline{F}(t,0,x,y,\theta_0(t,x,y))$. Remarkably, the nonlinear operator $\Lambda$ defined by \eqref{DefinitionofLamda} is well-defined given the well-posedness of nonlocal linear PDE \eqref{NonlocalLinearPDE}.

In order to apply the Banach's fixed point theorem, we need to strike a balance between $\delta$ and $R$ such that they satisfy the following three conditions: 
	\begin{enumerate}
	    \item To validate $F\big(t,s,x,y,\left(\partial_I u\right)_{|I|\leq 2}(t,s,x,y),  \left(\partial_I u\right)_{|I|\leq 2}(s,s,x,y)|_{x=y}\big)$, we require that the range of various derivatives of $u$ in $\mathcal{U}$ is contained in $B(\overline{z},R_0)$. Noting that
        \begin{equation} \label{Balance1}
    \begin{split}
        & \sup_{(t,s,x,y)\in[0,\delta]^2\times\mathbb{R}^{2d}} \sum_{|I|\leq 2} \big( \left|\partial_I u(t,s,x,y) - \partial_I g(t,x,y)\right| \\
        &\qquad\qquad\qquad\qquad + \left|\partial_I u(s,s,x,y)|_{x=y} - \partial_I g(s,x,y)|_{x=y}\right| \big) \leq C\delta^\frac{\alpha}{2}R,
    \end{split}
\end{equation}
		it should hold that $C\delta^\frac{\alpha}{2}R\leq \frac{R_0}{2}$;
		\item After a rather lenghty verification (see Appendix B), we can show the core inequality 
		\begin{equation} 
			\lVert\Lambda(u)-\Lambda(\widehat{u})\rVert^{(2+\alpha)}_{[0,\delta]}\leq C(R)\delta^\frac{\alpha}{2}\lVert u-\widehat{u}\rVert^{(2+\alpha)}_{[0,\delta]},   
		\end{equation} 
		and thus a small enough $\delta$ can be chosen to ensure $C(R)\delta^\frac{\alpha}{2r}\leq\frac{1}{2}$ such that $\Lambda$ is a $\frac{1}{2}$-contraction; 
		\item Before applying a fixed point argument, the last step is to prove that $\Lambda$ maps $\mathcal{U}$ into itself, i.e. $\lVert \Lambda(u)-g\rVert^{(2+\alpha)}_{[0,\delta]}\leq R$. Hence, $R$ should be suitably large such that $\lVert \Lambda(g)-g\rVert^{(2+\alpha)}_{[0,\delta]}\leq R/2$. 
	\end{enumerate}
	 
	\noindent \textbf{(Contractility of $\Lambda$)} Let us consider the equation for $U-\widehat{U}:=\Lambda(u)-\Lambda(\widehat{u})$:
    \begin{equation} 
    \left\{
    \begin{aligned}
        \big(U-\widehat{U}\big)_s(t,s,x,y) &= \mathcal{L}\big(U-\widehat{U}\big)- \mathcal{L}(u-\widehat{u}) \\
        &\quad + F\big(t,s,x,y,(\partial_I u)_{|I|\leq 2}(t,s,x,y), (\partial_I u)_{|I|\leq 2}(s,s,x,y)|_{x=y}\big) \\
        &\quad - F\big(t,s,x,y,(\partial_I\widehat{u})_{|I|\leq 2}(t,s,x,y), (\partial_I\widehat{u})_{|I|\leq 2}(s,s,x,y)|_{x=y}\big), \\ 
        \big(U-\widehat{U}\big)(t,0,x,y) &= 0, \qquad t,s \in [0,\delta], \quad x,y \in \mathbb{R}^d.
    \end{aligned}
    \right. 
\end{equation}
    According to the prior estimates \eqref{Schauderestimate} and \eqref{StabilityofnonlocallinearPDE} of nonlocal linear PDEs \eqref{NonlocalLinearPDE}, we have 
	\begin{equation} \label{varphi}
		\lVert U-\widehat{U}\rVert^{(2+\alpha)}_{[0,\delta]}\leq C\lVert\varphi\rVert^{(\alpha)}_{[0,\delta]},
	\end{equation}
    where the constant $C$ is independent of $\delta$ and the inhomogeneous term $\varphi$ is given by
    \begin{equation*}
    \begin{aligned}
        \varphi(t,s,x,y) &= F \big( t,s,x,y, (\partial_I u)_{|I|\leq 2}(t,s,x,y), (\partial_I u)_{|I|\leq 2}(s,s,x,y)|_{x=y} \big) \\
        &\quad - F \big( t,s,x,y, (\partial_I\widehat{u})_{|I|\leq 2}(t,s,x,y), (\partial_I\widehat{u})_{|I|\leq 2}(s,s,x,y)|_{x=y} \big)  - \mathcal{L} (u - \widehat{u}),
    \end{aligned}
\end{equation*}
for convenience, which can be rewritten as an integral representation: 
    \begin{equation} \label{Integralrepresentationofvarphi}
    \begin{aligned}
        &\int^1_0\frac{d}{d\sigma}F\big(t,s,x,y,\theta_\sigma(t,s,x,y)\big)d\sigma-\mathcal{L}\left(u-\widehat{u}\right) \\
        &= \int^1_0\sum_{|I|\leq 2}\partial_I F\big(t,s,x,y,\theta_\sigma(t,s,x,y)\big) \cdot \left(\partial_I u(t,s,x,y)-\partial_I\widehat{u}(t,s,x,y)\right)d\sigma \\
        &\quad + \int^1_0\sum_{|I|\leq 2}\partial_I \overline{F}\big(t,s,x,y,\theta_\sigma(t,s,x,y)\big)  \cdot\left(\partial_I u(s,s,x,y)|_{x=y}-\partial_I\widehat{u}^b(s,s,x,y)|_{x=y}\right)d\sigma - \mathcal{L}\left(u-\widehat{u}\right) \\
        &= \int^1_0\sum_{|I|\leq 2} \Big( \partial_I F\big(t,s,x,y,\theta_\sigma(t,s,x,y)\big) - \partial_I F\big(t,0,x,y,\theta_0(t,x,y)\big) \Big)  \cdot \partial_I\left(u-\widehat{u}\right)(t,s,x,y) \, d\sigma \\
        &\quad + \int^1_0\sum_{|I|\leq 2} \Big( \partial_I \overline{F}\big(t,s,x,y,\theta_\sigma(t,s,x,y)\big) - \partial_I \overline{F}\big(t,0,x,y,\theta_0(t,x,y)\big) \Big)  \cdot \partial_I\left(u-\widehat{u}\right)(s,s,x,y)|_{x=y} \, d\sigma
    \end{aligned}
\end{equation}
in which
    \begin{equation*}
    \begin{aligned}
        \theta_\sigma(t,s,x,y) &:= \sigma \big( (\partial_I u)_{|I|\leq 2}(t,s,x,y), \, (\partial_I u)_{|I|\leq 2}(s,s,x,y)|_{x=y} \big) \\
        &\quad + (1-\sigma) \big( (\partial_I\widehat{u})_{|I|\leq 2}(t,s,x,y), \, (\partial_I\widehat{u})_{|I|\leq 2}(s,s,x,y)|_{x=y} \big).
    \end{aligned}
\end{equation*}
To estimate $\lVert\varphi\rVert^{(\alpha)}_{[0,\delta]}$, we need to estimate the H\"{o}lder regularities of $|\varphi(t,s,x,y)|^{(\alpha)}_{(s,y)\in[0,\delta]\times\mathbb{R}^d}$, $|\varphi_t(t,s,x,y)|^{(\alpha)}_{(s,y)\in[0,\delta]\times\mathbb{R}^d}$, $|\varphi_x(t,s,x,y)|^{(\alpha)}_{(s,y)\in[0,\delta]\times\mathbb{R}^d}$, and $|\varphi_{xx}(t,s,x,y)|^{(\alpha)}_{(s,y)\in[0,\delta]\times\mathbb{R}^d}$ for any fixed $t\in[0,\delta]$ and $x\in\mathbb{R}^d$, all of which are listed in Table \ref{tab:table5}. After a lenghty but straightforward investigation of $K_1$-$K_{12}$ in Table \ref{tab:table5} (see Appendix \ref{App:B}), for a small enough $\delta$, we have 
\begin{equation} \label{Contraction} 
	\lVert U-\widehat{U}\rVert^{(2+\alpha)}_{[0,\delta]}\leq C\lVert\varphi\rVert^{(\alpha)}_{[0,\delta]}\leq C(R)\delta^{\frac{\alpha}{2}}\lVert u-\widehat{u}\rVert^{(2+\alpha)}_{[0,\delta]}\leq \frac{1}{2}\lVert u-\widehat{u}\rVert^{(2+\alpha)}_{[0,\delta]}. 
\end{equation}

\begin{table}[!ht]
\centering 
\begin{tabular}{cccc}
    \toprule
    \multicolumn{2}{l}{\quad\quad~ Estimates of $\lVert\varphi\rVert^{(\alpha)}_{[0,\delta]}$} & \\
    \cmidrule{1-1}
    $0\leq s<s^\prime\leq \delta$, $0<|y-y^\prime|\leq 1$ & ~\quad $ |~\cdot~|$ ~\quad & $\triangle_{s^\prime-s,s}$ & $\triangle_{y^\prime-y,y}$ \\
    \midrule
    $\varphi(t,s,x,y)$ & ~\quad $K_1$ ~\quad & $K_2$ & $K_3$ \\
    $\varphi_t(t,s,x,y)$ & ~\quad $K_4$ ~\quad & $K_5$ & $K_6$ \\
    $\varphi_x(t,s,x,y)$ & ~\quad $K_7$ ~\quad & $K_8$ & $K_9$ \\
    $\varphi_{xx}(t,s,x,y)$ & ~\quad $K_{10}$ ~\quad & $K_{11}$ & $K_{12}$ \\
    \bottomrule
\end{tabular}
\caption{Estimate of $\lVert\varphi\rVert^{(\alpha)}_{[0,\delta]}$} 
	\label{tab:table5} 
\end{table} 	

	
\noindent \textbf{(Self-mapping of $\Lambda$)} Before applying the Banach's fixed point theorem, we need to choose a suitably large $R$ such that $\Lambda$ maps $\mathcal{U}$ into itself. Letting $\delta$ and $R$ satisfy 
	\begin{equation*}
		C(R)\delta^{\frac{\alpha}{2}}\leq \frac{1}{2},
	\end{equation*} 
	then $\Lambda$ is a $\frac{1}{2}$-contraction and for any $u\in\mathcal{U}$, we have
	\begin{equation*}
		\lVert\Lambda(u)-g\rVert^{(2+\alpha)}_{[0,\delta]}\leq \frac{R}{2}+\lVert \Lambda(g)-g\rVert^{(2+\alpha)}_{[0,\delta]}. 
	\end{equation*} 
	Define the function $G:=\Lambda(g)-g$ as the solution of the equation 
    \begin{equation*}
    \left\{
    \begin{aligned}
        G_s(t,s,x,y) &= \mathcal{L}G + F\big(t,s,x,y, (\partial_I g)_{|I|\leq 2}(t,x,y), \, (\partial_I g)_{|I|\leq 2}(s,x,y)|_{x=y} \big), \\
        G(t,0,x,y) &= 0, \qquad t,s \in [0,\delta], \quad x,y \in \mathbb{R}^d. 
    \end{aligned}
    \right. 
\end{equation*}
	By \eqref{Schauderestimate}, there is $C>0$ independent of $\delta$ such that 
	\begin{equation*}
		\lVert G\rVert^{(2+\alpha)}_{[0,\delta]}\leq C\lVert \psi\rVert^{(\alpha)}_{[0,\delta]}=:C^\prime, 
	\end{equation*}
	where $\psi(t,s,x,y)=F\big(t,s,x,y,\left(\partial_Ig\right)_{|I|\leq 2}(t,x,y),  \left(\partial_Ig\right)_{|I|\leq 2}(s,x,y)|_{x=y}\big)$. Hence, we have 
	\begin{equation*}
		\lVert\Lambda(u)-g\rVert^{(2+\alpha)}_{[0,\delta]}\leq \frac{R}{2}+C^\prime.
	\end{equation*}
	Therefore for a suitably large $R$, $\Lambda$ is a contraction mapping $\mathcal{U}$ into itself and it has a unique fixed point $u$ in $\mathcal{U}$ satisfying  
    \begin{equation} \label{NonlocalfullynonlinearPDEfrom0todelta}
    \left\{
    \begin{aligned}
        u_s(t,s,x,y) &= F\big(t,s,x,y, (\partial_I u)_{|I|\leq 2}(t,s,x,y), \, (\partial_I u)_{|I|\leq 2}(s,s,x,y)|_{x=y} \big), \\
        u(t,0,x,y) &= g(t,x,y), \qquad t,s \in [0,\delta], \quad x,y \in \mathbb{R}^d.  
    \end{aligned}
    \right. 
\end{equation}


\noindent (\textbf{Uniqueness}) To complete the proof, we ought to show that $u$ is the unique solution of \eqref{NonlocalfullynonlinearPDEfrom0todelta} in the whole space $\Omega^{(2+\alpha)}_{[0,\delta]}$ rather than just in the subset $\mathcal{U}$. We can directly study a nonlocal PDE satisfied by the difference of any two solutions $u$ and $\overline{u}$ in $\Omega^{(2+\alpha)}_{[0,\delta]}$ to \eqref{NonlocalfullynonlinearPDEfrom0todelta}. By using the mean value theorem and the technique of \eqref{Integralrepresentationofvarphi}, it is clear that the difference $u-\overline{u}$ solves a nonlocal, homogeneous, linear, and strongly parabolic PDE with initial value zero. Consequently, similar to the proof of Theorem \ref{Well-posednessofL0}, we obtain the uniqueness of \eqref{NonlocalfullynonlinearPDEfrom0todelta}. 

Alternatively, the uniqueness of solutions of \eqref{NonlocalfullynonlinearPDEfrom0todelta} can also be proven by establishing a contradiction. Specifically, supposed that \eqref{NonlocalfullynonlinearPDEfrom0todelta} admits two solutions $u$ and $\overline{u}$ and let 
\begin{equation*}
    t_0 = \sup \Big\{ t \in [0, \delta] : u(t,s,x,y) = \overline{u}(t,s,x,y), \ (t,s,x,y) \in [0,t]^2 \times \mathbb{R}^{d;d} \Big\}.
\end{equation*}
If $t_0=\delta$ the proof is completed. Otherwise, we can consider a new nonlocal fully nonlinear PDE with an updated initial condition $g^\prime(t,x,y)=u(t,t_0,x,y)=\overline{u}(t,t_0,x,y)$ in $(t,s,x,y)\in[t_0,T]^2\times\mathbb{R}^{d;d}$. As the proof of Theorem \ref{Localwell-posednessoffullynonlinearPDE} shows, there exist a enough small $\delta^\prime$ and a suitable large $R^\prime$ such that the updated \eqref{NonlocalfullynonlinearPDEfrom0todelta} admits a unique solution in a ball centered at $g^\prime$ with radius $R^\prime$. By choosing a sufficiently large $R^\prime$ such that both $u$ and $\overline{u}$ are contained in this ball, we have $u=\overline{u}$ in $[0,t_0+\delta^\prime]^2\times\mathbb{R}^{d;d}$. This contradicts the definition of $t_0$. Consequently, $t_0=\delta$ and $u=\overline{u}$.       
\end{proof}



\subsection{Extensions to a Larger Time Horizon} \label{Subsec:Extension}
In this subsection, we extend the local well-posedness result in Theorem \ref{Localwell-posednessoffullynonlinearPDE} to the largest possible time horizon, leading to a maximally defined solution. In fact, we have proven that there exists a $\delta_1>0$ such that the nonlocal fully nonlinear PDE \eqref{NonlocalfullynonlinearPDE} is well-posed in $[0,\delta_1]^2\times\mathbb{R}^{d;d}$, i.e. the time region $R_1$ in Figure \ref{fig:extension}.
\begin{figure}[!ht]
	\centering
	\includegraphics[width=0.5\textwidth]{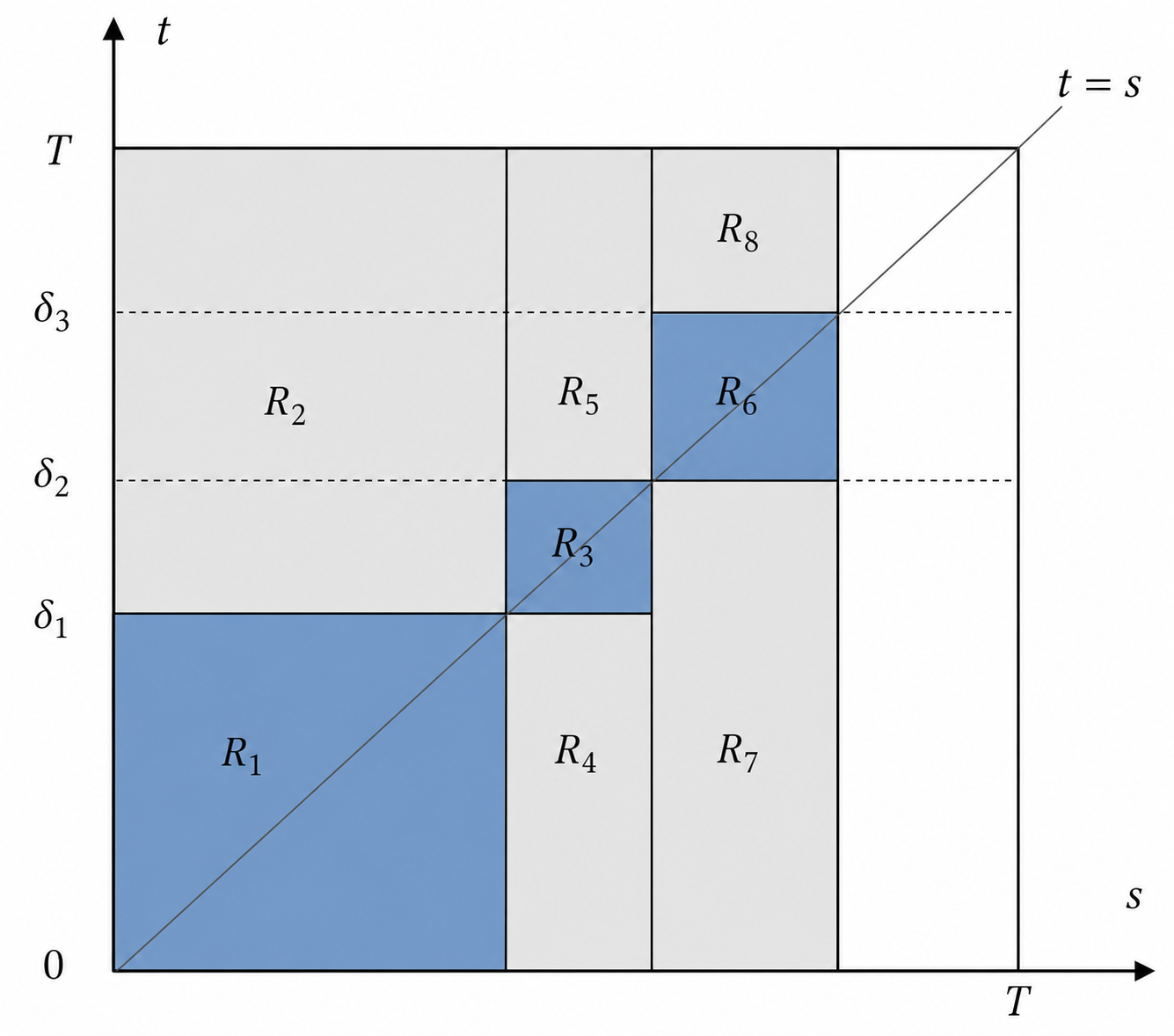}
	\caption{Extension from $[0,\delta_1]^2$ to a larger time region $[0,T]\times[0,\delta_3]$}
	\label{fig:extension}
\end{figure}

Hence, the diagonal condition can be determined for $s\in[0,\delta_1]$, which means all nonlocal terms $\left(\partial_I u\right)_{|I|\leq 2}(s,s,x,y)\big|_{x=y}$ of \eqref{NonlocalfullynonlinearPDE} are known for $s\in[0,\delta]$. Subsequently, \eqref{NonlocalfullynonlinearPDE} reduces to a family of classical fully nonlinear PDEs parameterized by $(t,x)\in[0,T]\times\mathbb{R}^d$. By the classical PDE theory, it is natural to extend the solution of \eqref{NonlocalfullynonlinearPDE} from $R_1$ to $R_1\bigcup R_2$. Next, we take $s=\delta_1$ as initial time and $g^\prime(t,x,y)=u(t,\delta_1,x,y)$ as initial datum, then consider the nonlocal PDE
\begin{equation} \label{UpdatednonlocalfullynonlinearPDE} 
    \left\{
    \begin{aligned}
        u_s(t,s,x,y) &= F\big(t,s,x,y, (\partial_I u)_{|I|\leq 2}(t,s,x,y), \, (\partial_I u)_{|I|\leq 2}(s,s,x,y)\big|_{x=y}\big), \\
        u(t,0,x,y) &= g^\prime(t,x,y), \qquad t,s \in [\delta_1, T], \quad x,y \in \mathbb{R}^d. 
    \end{aligned}
    \right. 
\end{equation}
Since $g^\prime\in\Omega^{{(2+\alpha)}}_{[\delta_1,T]}$, Theorem \ref{Localwell-posednessoffullynonlinearPDE} promises that there exists $\delta_2>0$ such that \eqref{UpdatednonlocalfullynonlinearPDE} is solvable as well in $R_3\times\mathbb{R}^{d;d}$. Similarly, after solving \eqref{UpdatednonlocalfullynonlinearPDE} in the time region $R_3$, the nonlocal term of \eqref{UpdatednonlocalfullynonlinearPDE} are all given for $s\in[\delta_1,\delta_2]$ and \eqref{UpdatednonlocalfullynonlinearPDE} reduces to a family of conventional PDEs. The solution of \eqref{UpdatednonlocalfullynonlinearPDE} can then be extended from $R_3$ to $R_3\bigcup R_4\bigcup R_5$. Until now, we have obtained a solution of \eqref{NonlocalfullynonlinearPDE} in $[0,T]\times[0,\delta_2]\times\mathbb{R}^{d;d}$. One can continue extending the solution to a larger time interval by repeating the procedure above until a maximally defined solution $u(t,s,x,y)$ defined in $[0,T]\times[0,\tau)\times\mathbb{R}^{d;d}$ is reached. The time region $[0,T]\times[0,\tau)$ is maximal in the sense that if $\tau<\infty$, then there does not exist any solution of \eqref{NonlocalfullynonlinearPDE} belonging to $\Omega^{(2+\alpha)}_{[0,\tau]}$. 

It is noteworthy that the problem of existence at large for arbitrary initial data is a difficult task even in the local fully nonlinear case. The difficulty is caused by the fact that a-priori estimate in a very high norm $|\cdot|^{(2+\alpha)}_{[a,b]\times\mathbb{R}^d}$ (or $\lVert\cdot\rVert^{(2+\alpha)}_{[a,b]}$) is needed to establish the existence at large. To this end, there should be severe restrictions on the nonlinearities. More details are discussed in \cite{Krylov1987,Lieberman1996}. Next, let us denote $\tau(g)$ as the maximal time interval associated with $g$. To achieve the global existence, it is desired to prove that $\tau(g)=T$. Inspired by \cite{Lunardi1989,Prato1996,Lei2021}, we will show that the existence of solutions to \eqref{NonlocalfullynonlinearPDE} for an arbitrary large $T>0$, $\epsilon>0$, and initial data $g\in\Omega^{(2+\alpha+\epsilon)}_{[0,T]}$ if a very sharp priori estimate is available. 

\begin{theorem} \label{Globalsolvability}
  Let $F$ and $g$ satisfy the assumptions of Theorem \ref{Localwell-posednessoffullynonlinearPDE} with $\alpha$ replaced by $\alpha+\epsilon$. For a fixed $g\in\Omega^{(2+\alpha+\epsilon)}_{[0,T]}$, let $u$ be the maximally defined solution of problem \eqref{NonlocalfullynonlinearPDE} for $s\in[0,\tau(g))$. Further assume that there exists a constant $M>0$ such that 
  \begin{equation} \label{Upperboundedforglobalexistence}
      \lVert u\rVert^{(2+\alpha+\epsilon)}_{[0,\sigma]}\leq M, ~ \text{for all} ~ \sigma\in[0,\tau(g)),
  \end{equation}
  then we have either $\lim_{s\to\tau}u(t,s,x,y)\in\partial\mathcal{O}$ or $\tau(g)=T$, where $\partial\mathcal{O}$ is the boundary of the open set $\mathcal{O}$ of $\Omega^{(2+\alpha)}_{[0,T]}$ with
  $$
  \mathcal{O}=\left\{u~|~\big(\left(\partial_I u\right)_{|I|\leq 2}(t,s,x,y), \left(\partial_I u\right)_{|I|\leq 2}(s,s,x,y)|_{x=y}\big)\subseteq B(\overline{z},R_0)\right\}.
  $$
\end{theorem}

Remarkably, in order to obtain the $\frac{1}{2}$-contractility of $\Lambda$ defined by \eqref{DefinitionofLamda}, we need to strike a balance between $R$ and $\delta$ in $C(R)\delta^\frac{\alpha}{2}$ of \eqref{Contraction}. In the study of nonlocal fully nonlinear PDEs, the radius $R$ of solution $u$ of \eqref{NonlocalfullynonlinearPDEfrom0todelta} contains its $\lVert\cdot\rVert^{(2+\alpha)}_{[0,\delta]}$-norm information. Hence, a prior estimate $\lVert u\rVert^{(2+\alpha)}_{[0,\sigma]}\leq M$ for $\sigma\in[0,\tau(g))$ is not sufficient for existence in the large. The $\epsilon$ of \eqref{Upperboundedforglobalexistence} provides an additional regularity for the maximally defined solution to ensure the extension. Similar sufficient conditions to obtain a-priori estimates like \eqref{Upperboundedforglobalexistence} for the classical PDEs can be found in \cite{Krylov1987,Lieberman1996}. However, it is not straightforward to express such conditions in terms of coefficients and data of the local and nonlocal fully nonlinear PDEs. 

Next, we will make use of Theorems \ref{Localwell-posednessoffullynonlinearPDE} and \ref{Globalsolvability} to show the global solvability of a special class of \eqref{NonlocalfullynonlinearPDE}, which is called nonlocal quasilinear PDEs, of the form: 
\begin{equation} \label{QuasilinearPDE} 
    \left\{
    \begin{aligned}
        u_s(t,s,x,y) &= \sum_{|I|= 2} A^I(s,y) \partial_I u(t,s,x,y) \\
        &\quad + Q\big(t,s,x,y, (\partial_I u)_{|I|\leq 1}(t,s,x,y), \, (\partial_I u)_{|I|\leq 1}(s,s,x,y)\big|_{x=y}\big), \\
        u(t,0,x,y) &= g(t,x,y), \qquad t,s \in [0,T], \quad x,y \in \mathbb{R}^d,
    \end{aligned}
    \right. 
\end{equation}
where the nonlinearity $Q(t,s,x,y,p^I,q^I)$ satisfies that there exists a constant $K<\infty$ such that $Q$ and its first-/second-derivatives satisfy
\begin{equation} \label{Growthcondition}
    |Q|\leq K\Big(1+\sum|p^I|\Big), 
\end{equation}
\begin{equation} \label{Boundednesscondition}
    \max\limits_{|I|\leq 1, 1\leq i,j\leq d}\left\{|Q_{p^I}|, |Q_t|, |Q_{x_i}|, |Q_{p^Ip^J}|, |Q_{tp^I}|, |Q_{x_ip^I}|, |Q_{x_ix_j}|\right\}\leq K. 
\end{equation}
Such a simplified PDE \eqref{QuasilinearPDE} contains the equilibrium-type HJB equations originated from TIC stochastic control problems driven by state processes controlled solely through the drift. For such cases, we obtain the global well-posedness result. To the best of our knowledge, the result obtained as a degeneration of Theorem \ref{Localwell-posednessoffullynonlinearPDE} and Theorem \ref{Globalsolvability} constitutes the state-of-the-art in the existing TIC literature. 
\begin{theorem} \label{Well-posednessofquasilinearPDE}
  Suppose that $A^I$ satisfies the uniform ellipticity condition \eqref{Uniformellipticitycondition1} and the nonlinearity $Q$ meets the conditions \eqref{Growthcondition}-\eqref{Boundednesscondition}. The nonlocal quasilinear PDE \eqref{QuasilinearPDE} admits a unique solution in $\Omega^{(2+\alpha)}_{[0,T]}$ in $[0,T]^2\times\mathbb{R}^{d;d}$. 
\end{theorem}

The proofs of Theorems \ref{Globalsolvability}–\ref{Well-posednessofquasilinearPDE} follow \cite{Lei2021,Lei2023a}; see also Appendix \ref{App:ProofThm4.2}-\ref{App:ProofThm4.3}.

\subsection{Extensions to a Larger Function Space} \label{Sec:WeightedNormsSpaces}
In this subsection, we extend the main results from the previous sections to a weighted space that allows both functions and their partial derivatives to grow exponentially with respect to the spatial variables \(x\) and \(y\), see also \cite{Lei2021,Lorenzi}. We focus on exponential weights defined by \(\varrho(x,y) = \exp\{1 + \langle Sx, x \rangle^{1/2} + \langle Sy, y \rangle^{1/2}\}\) for any $x$, \(y \in \mathbb{R}^d\), where \(S\) is a symmetric positive-definite matrix with eigenvalues in \([\underline{\lambda}, \overline{\lambda}]\) and \(\underline{\lambda} > 0\). Moreover, $\varrho_{\min}(x,y):=\min\left\{\varrho^{-1}(x,y),\varrho^{-1}(y,y)\right\}$ and $\varrho_{\max}(x,y):=\max\left\{\varrho(x,y),\varrho(y,y)\right\}$. First, we introduce the following weighted norms:
\begin{equation*}
    \begin{aligned}
        & |\psi(t,\cdot,x,\cdot)|^{(l)}_{\varrho,[a,b]\times\mathbb{R}^d} := \sum_{k\leq[l]}\sum_{2i+j=k} \left| \frac{D^i_s D^j_y \psi(t,\cdot,x,\cdot)}{\varrho_{\max}(x,\cdot)} \right|^{(0)}  + \sum_{0<l-2i-j<2} \left\langle \frac{D^i_s D^j_y \psi(t,\cdot,x,\cdot)}{\varrho_{\max}(x,\cdot)} \right\rangle^{(\frac{l-2i-j}{2})}_s \\
        &\quad + \sum_{2i+j=[l]} \left\{ \sup_{\begin{subarray}{c} s\in[a,b], \, y,y^\prime\in\mathbb{R}^d \\ 0<|y-y^\prime|\leq 1 \end{subarray}} \frac{|D^i_s D^j_y \psi(t,s,x,y) - D^i_s D^j_y \psi(t,s,x,y^\prime)|}{|y-y^\prime|^{(l-\lfloor l\rfloor)}} \cdot \vphantom{\sup_{\begin{subarray}{c} s\in[a,b] \end{subarray}}} \min\left\{\varrho_{\min}(x,y), \, \varrho_{\min}(x,y^\prime)\right\} \right\}
    \end{aligned}
\end{equation*}
and \([\Psi]^{(l)}_{\varrho,[0,\delta]}:=\sup_{(t,x)\in[a,b]\times\mathbb{R}^d}|\psi(t,\cdot,x,\cdot)|^{(l)}_{\varrho,[a,b]\times\mathbb{R}^d}\). Similarly, we can also define 
\begin{equation*}
	\begin{split}
		\lVert\psi\rVert^{(l)}_{\varrho,[a,b]}& := \sup\limits_{(t,x)\in[a,b]\times\mathbb{R}^d}\Big\{|(\psi,\psi_t,\psi_x,\psi_{xx})(t,\cdot,x,\cdot)|^{(l)}_{\varrho,[a,b]}\Big\} 
	\end{split}
\end{equation*}
as well as the weighted space \(\Omega^{(l)}_{\varrho,[0,\delta]}:=\Big\{\psi\in C([a,b]^2\times\mathbb{R}^{d;d};\mathbb{R}):\Vert\psi\rVert^{(l)}_{\varrho,[a,b]}<\infty\Big\}\).

Next, we present a class of nonlinearities that generalizes the nonlinearity \(F\) discussed in Subsection \ref{Sec:Nonlinear}, in order to investigate well-posedness within the weighted spaces.
\begin{definition} \label{Def:AppropriatePair}
A pair of $(F,g)$ is appropriate if there exist $\delta$, $R>0$ such that for any $u\in\big\{u\in\Omega^{(2+\alpha)}_{\varrho,[0,\delta]}:u(t,0,x,y)=g(t,x,y),\| u-g\|^{(2+\alpha)}_{\varrho,[0,\delta]}\leq R\big\}$,
\begin{enumerate}[label=(\alph*)]
	\item $F\big(t,s,x,y,\left(\partial_I u\right)_{|I|\leq 2}(t,s,x,y),  \left(\partial_I u\right)_{|I|\leq 2}(s,s,x,y)|_{x=y}\big)\in\Omega^{(\alpha)}_{\varrho,[0,\delta]}$. 
	\item both $\partial_I F$ and $\partial_I\overline{F}$ at $u$ belong to $\Omega^{(\alpha)}_{[0,\delta]}$. 
	\item the uniformly ellipticity condition \eqref{UniformellipticityconditionofF1}-\eqref{UniformellipticityconditionofF2} holds.  
        \item the inequality $\left|\Delta_{s,y}\mathcal{F}\big(t,s,x,y,\left(\partial_I u\right)_{|I|\leq 2}(t,s,x,y),  \left(\partial_I u\right)_{|I|\leq 2}(s,s,x,y)|_{x=y}\big)\cdot \mathcal{K}\right|\leq C(R)\left(|s-s^\prime|^\frac{\alpha}{2}+|y-y^\prime|^\alpha\right)$ holds, 
        \begin{itemize}
        	\item if $ \mathcal{F}\in\big\{\partial_I F,\partial_I\overline{ F},\partial^2_{It} F,\partial^2_{It}\overline{F},\partial^2_{Ix} F,\partial^2_{Ix}\overline{F},\partial^3_{Ixx}F,\partial^3_{Ixx}\overline{F}\big\}$, then $\mathcal{K}=1$;
        	\item if $\mathcal{F}\in\big\{\partial^2_{IJ}F,\partial^2_{IJ}\overline{F},\partial^3_{xIJ}F,\partial^3_{xIJ}\overline{F}\big\}$, then $\mathcal{K}=\partial_J\overline{u}(t,s,x,y)$;
        	\item if $\mathcal{F}\in\big\{\partial^3_{IJK}F,\partial^3_{IJK}\overline{F}\big\}$, then $\mathcal{K}=\left(\partial_J\overline{u}\cdot\partial_K\overline{u}\right)(t,s,x,y)$,
        \end{itemize}
\end{enumerate}
where $\Delta_{s,y}\varphi(s,y):=|\varphi(s^\prime,y^\prime)-\varphi(s,y)|$.
\end{definition}

Conditions (a)-(d) in Definition \ref{Def:AppropriatePair} allow us to employ the methodologies from Sections \ref{Sec:Linear} to \ref{Sec:Nonlinear}, including the linearization method and the fixed-point argument, to study fully nonlinear nonlocal systems in a weighted space. Though this refined framework allows for the nonhomogeneous term \( f \) and the initial data \( g \) that grow exponentially in the spatial variable (specifically, \( f \in \Omega^{{(\alpha)}}_{\varrho,[0,T]} \) and \( g \in \Omega^{{(2+\alpha)}}_{\varrho,[0,T]} \)), all coefficients of \( L \) defined in \eqref{Nonlocallinearoperator} are still required to belong to \( \Omega^{(\alpha)}_{[0,T]} \). This induces Condition (b), which stipulates that \( \partial_I F \) and \( \partial_I \overline{F} \) at \( u \) belong to ordinary normed spaces rather than weighted ones. Conditions (a)-(c) ensure that the mapping \(u \mapsto U\), defined by \(U_s = \mathcal{L}U + F(u) - \mathcal{L} u\), is well-defined, while Condition (d) enables us to demonstrate that this mapping is contractive. These conditions are satisfied by our financial example in Section \ref{Sec:FinEx}.

Under these conditions, all well-posedness results for nonlocal systems in Section \ref{Sec:Nonlinear} can be extended to weighted spaces in Theorem \ref{WellposednessWeighted}, whose proof in the same spirit is omitted.

\begin{theorem} \label{WellposednessWeighted}
We obtain the well-posedness for nonlocal systems in weighted spaces:
\begin{enumerate}
	\item If all coefficients of $L$ defined in \eqref{Nonlocallinearoperator} belong to $\Omega^{(\alpha)}_{[0,T]}$, $f\in\Omega^{{(\alpha)}}_{\varrho,[0,T]}$, and $g\in\Omega^{{(2+\alpha)}}_{\varrho,[0,T]}$, then the linear nPDE \eqref{NonlocalLinearPDE} admits a unique solution $u\in\Omega^{{(2+\alpha)}}_{\varrho,[0,T]}$ in $\Delta[0,T]\times\mathbb{R}^d$. Moreover,   
	\begin{equation} \label{Weighted:Estimatesofsolutionsofnonlocalsystem} 
		\| u\|^{(2+\alpha)}_{\varrho,[0,T]}\leq C\left(\| f\|^{(\alpha)}_{\varrho,[0,T]}+\| g\|^{(2+\alpha)}_{\varrho,[0,T]}\right). 
	\end{equation}  
	\item Suppose that the pair of $(F,g)$ is appropriate in the sense of Definition \ref{Def:AppropriatePair}. Then, there exist $\tau>0$ and a unique maximally-defined solution $u\in\Omega^{{(2+\alpha)}}_{\varrho,[0,\tau]}$ satisfying \eqref{NonlocalfullynonlinearPDE} in $\Delta[0,\tau]\times\mathbb{R}^d$. Assume further that $\| u\|^{(2+\alpha^\prime)}_{\varrho,[0,\sigma]}\leq M$ for some finite constant $M>0$ across all $\sigma\in[0,\tau)$, then either the pair of $(F,\lim_{s\to\tau}u(\cdot,s,\cdot))$ is not appropriate or $\tau=T$. Specially, the nonlocal quasilinear system \eqref{QuasilinearPDE} is globally solvable.  
\end{enumerate}
\end{theorem}



\section{Well-Posedness of Equilibrium HJB Equations and Financial Examples} \label{Sec:TIC}
With the well-posedness results in Section \ref{Sec:Nonlinear}, we now echo back the TIC stochastic control problem of our interest \eqref{TICproblem} and analyze the well-posedness of the equilibrium HJB equation \eqref{EquilibriumHJBequation}. We summarize for the latter and then give a financial example.

\subsection{Analyses of Equilibrium HJB Equations} 
In this subsection, we apply Theorems \ref{Localwell-posednessoffullynonlinearPDE} and \ref{Well-posednessofquasilinearPDE} to equilibrium HJB equations. Moreover, we examine how close in value function between sophisticated and ``na\"{i}ve" controllers in Appendix \ref{app:diffvalue}. 

To this end, we first introduce some useful notations:
\begin{enumerate}
    \item For two fixed bounded subsets $P$, $Q\subseteq\mathbb{R}\times\mathbb{R}^{d}\times\mathbb{R}^{d^{2}}\times\mathbb{R}\times\mathbb{R}^{d}\times\mathbb{R}^{d^{2}}$, then 
    \begin{equation*}
        d(P,Q):=\inf\{|p-q|:p\in P,q\in Q\}.  
\end{equation*}
\item For any $u\in\Omega^{(2+\alpha)}_{[0,T]}$, the range of $u$ is defined by 
\begin{equation*}
    R(u):=\{(\partial_I u(t,s,x,y),\partial_Iu(s,s,x,y)|_{x=y}):t,s\in[0,T]^2,x,y\in\mathbb{R}^d\}. 
\end{equation*}
\end{enumerate}

We have the following conclusions. 
\begin{proposition} \label{SolvabilityofTIC}
    Let $\overline{\mathcal{H}}(t,s,x,y,z)=\mathcal{H}(T-t,T-s,x,y,z)$ defined in \eqref{Hamiltonian}, and $\overline{g}(t,x,y)=g(T-t,x,y)$. Suppose that $\overline{\mathcal{H}}$ satisfies the conditions \eqref{UniformellipticityconditionofF1}-\eqref{LipschitzcontinuityofF} in an open ball $B$ with $d(\partial B,R(\overline{g}))>0$. Then, for the TIC stochastic control problem \eqref{TICproblem}-\eqref{ControlledFBSDE}, we have that
    \begin{enumerate}
    	\item if both the drift and the diffusion of \eqref{ControlledFBSDE} are controlled, there exist $\tau\in(0,T]$ and a unique maximally defined solution $u\in\Omega^{{(2+\alpha)}}_{[T-\tau,T]}$ satisfying the equilibrium HJB equation \eqref{EquilibriumHJBequation} in $\nabla[T-\tau,T]\times\mathbb{R}^{d;d}$. Consequently, one has a feedback equilibrium control and a $C^{1,2}$ equilibrium value function of the form 
        \begin{equation} \label{SolutionpairofTICprobems}
    \left\{
    \begin{aligned}
        \mathbbm{e}(s,y) &:= \psi\big(s,s,y,y, (\partial_I u)_{|I|\leq 2}(s,s,x,y)\big|_{x=y}\big), \\ 
        V(s,y) &:= J(s,y; \mathbbm{e}(s,y)) = u(s,s,y,y), 
    \end{aligned}
    \right. 
\end{equation}
    	at least in the maximally defined time interval. Further assume that the domain of $\overline{\mathcal{H}}$ is large enough and \eqref{Upperboundedforglobalexistence} holds, we have $\tau=T$; 
    	\item in the case where only the drift is controlled while the diffusion of \eqref{ControlledFBSDE} is uncontrolled, i.e. $\sigma(s,y,a)=\sigma(s,y)$, the equilibrium HJB equation \eqref{EquilibriumHJBequation} is solvable globally. 
    \end{enumerate}
\end{proposition}

Proposition \ref{SolvabilityofTIC} directly comes from Theorems \ref{Localwell-posednessoffullynonlinearPDE} and \ref{Well-posednessofquasilinearPDE}. The only important thing to note in the proposition is that the condition $d(\partial B,R(\overline{g}))>0$ is weaker than the counterpart in Theorem \ref{Localwell-posednessoffullynonlinearPDE}, where $d(\partial B,R(\overline{g}))>R_0/2$ was required. In fact, from the inequality \eqref{Balance1} in the proof of Theorem \ref{Localwell-posednessoffullynonlinearPDE}, we can find that the local solution always exists only if the range of initial data is contained in the interior of the domain of nonlinearity.  

Proposition \ref{SolvabilityofTIC} is significant as it partially addressed some open research problems listed in \cite{Bjoerk2017}, i.e.,
\begin{itemize}
	\item to provide conditions on primitives which guarantee that the functions $V$ and $u$ are regular enough to satisfy the extended HJB system; and
	\item to prove existence and/or uniqueness for solutions of the extended HJB system.
\end{itemize}
From our well-posedness and regularities results of nonlocal PDEs \eqref{NonlocalfullynonlinearPDE}, it is clear that these open problems can be solved at least within the maximally defined time interval. Note that the extended HJB equation in \cite{Bjoerk2017} and the equilibrium HJB equation \eqref{EquilibriumHJBequation} are equivalent; see \cite{Wei2017}. Moreover, our function space $\Omega^{(2+\alpha)}_{[a,b]}$ supports the $C^{1,1,2,2}$-regular condition for possible solutions of the extended HJB equation. In addition, our function space requires neither second-order partial derivative in $t$ nor mixed ones between $t$ and $x$. It suits well the formulations of TIC stochastic control problems.

\subsection{Financial Examples} \label{Sec:FinEx}
We present two globally solvable examples of the TIC stochastic control problem \eqref{TICproblem}. The first studies optimal investment under exponential utility and fits our refined (weighted-norm) framework, ensuring global solvability on ([0,T]). The second examines investment–consumption under power utility but falls outside the framework due to degeneracy, motivating extensions; details are deferred to Appendix \ref{App:PowerU}.


Consider a market with a risk-free bond (rate $r>0$) and a risky asset (return $\mu>r$, volatility $\sigma>0$). Let $\alpha(\cdot)$ be the amount invested in the risky asset. Then the wealth $X(\cdot)$ and TIC recursive utility $(Y(\cdot),Z(\cdot,\cdot))$ satisfy the controlled FBSDE:

\begin{equation}\label{ExpUFBSDE}
    \left\{
    \begin{aligned}
        dX(s) &= [rX(s)+(\mu-r)\alpha(s)]ds + \sigma \alpha(s)dW(s), && s\in[t,T], \\
        dY(s) &= -\big[v(t,s,X(t)) - w(t,s,X(t))Y(s)\big]ds + Z(t,s)dW(s), && s\in[t,T], \\
        X(t) &= x, \quad Y(T) = -g(t,X(t))\exp\{-\eta X(T)\} + h(t,X(t)), && t\in[0,T].
    \end{aligned}
    \right.
\end{equation}
Then, we define the recursive utility functional for the investor as
\begin{equation*} \label{Exampleutility}
	J(t,x;\alpha(\cdot)):=Y(t;t,x,\alpha(\cdot)). 
\end{equation*}
Hence, the problem for the investor is to identify the optimal investment such that a sort of exponential utilities of the instantaneous and terminal wealth is maximized. Due to the general $(t,x)$-dependence of the generator and of the terminal condition of the BSDE, the aforementioned stochastic control problem is TIC. Note that we are dealing with a utility maximization problem. With some simple transformation, it can be reformulated as a minimization of a TIC functional, to align with our framework. 

By following the analysis in \cite{Lei2021}, it is necessary to consider the following variant of $\eqref{ExpUFBSDE}$:
\begin{equation} \label{TransExpUFBSDE}
    \left\{
    \begin{aligned}
        dX(s) &= \big[(\mu-r)\exp\{r(T-s)\}\alpha(s)\big] ds + \sigma\exp\{r(T-s)\} \alpha(s)dW(s), \\
        dY(s) &= -\Big[ \gamma w_1(t,s,X(t),X(s))\alpha(s) - w_2(t,s,X(t),X(s))\alpha^2(s) \\
        &\qquad - w_3(t,x,X(t))Y(s) + w_4(t,s,X(t)) \Big] ds + Z(t,s)dW(s), \\
        X(t) &= x\exp\{r(T-s)\}, \\
        Y(T) &= \gamma g_1(t,X(t))\exp\{\eta X(T)\} - g_2(t,X(t))\exp\{-\eta X(T)\} + g_3(t,X(t)).
    \end{aligned}
    \right.
\end{equation}
It is clear that under suitable assumptions, we can obtain the same conclusions for \eqref{ExpUFBSDE} via the well-posedness analysis of \eqref{TransExpUFBSDE} by letting the parameter \(\gamma\) approaches to zero, as with \cite{Lei2021}. Following the derivation in Section \ref{Sec:TIC}, we consider the following Hamiltonian
\begin{equation*}
    \begin{aligned}
        H_\gamma(t,s,x,y,\alpha,u,p,q) &= \frac{1}{2}\big(\widehat{\sigma}(s)\alpha\big)^2 q + \widehat{\mu}(s)\alpha p + \gamma w_1(t,s,x,y)\alpha \\
        &\quad - w_2(t,s,x,y)\alpha^2 - w_3(t,s,x) u + w_4(t,s,x),
    \end{aligned}
\end{equation*}
where $\widehat{\mu}:=(\mu-r)\exp\{r(T-s)\}\alpha(s)$ and $\widehat{\sigma}:=\sigma\exp\{r(T-s)\} \alpha(s)$. Maximizing it with respect to $\alpha$ gives us the equilibrium strategy 
\begin{equation} \label{Example1:Closedloopstrategy} 
    \overline{\alpha}(s,y) = \frac{\widehat{w}_1(s,s,y,y) + (\mu-r) U_y(s,s,x,y)\big|_{x=y}}{\widehat{w}_2(s,s,y,y) - \sigma^2 U_{yy}(s,s,x,y)\big|_{x=y}} \exp\big\{-r(T-s)\big\}
\end{equation}
with $U(t,s,x,y)$ ($\gamma$ is suppressed) being the solution to an equilibrium HJB equation:
		
		
\begin{equation} \label{BackwardHJBExponentialutility}
    \left\{
    \begin{aligned}
        &U_s(t,s,x,y) + \frac{1}{2} \left( \frac{\sigma\widehat{w}_1(s,s,y,y) + \sigma(\mu-r) U_y(s,s,x,y)|_{x=y}}{\widehat{w}_2(s,s,y,y) - \sigma^2 U_{yy}(s,s,x,y)|_{x=y}} \right)^2 U_{yy}(t,s,x,y) \\
        &\quad + \left( \frac{(\mu-r)\widehat{w}_1(s,s,y,y) + (\mu-r)^2 U_y(s,s,x,y)|_{x=y}}{\widehat{w}_2(s,s,y,y) - \sigma^2 U_{yy}(s,s,x,y)|_{x=y}} \right) U_y(t,s,x,y) \\
        &\quad + \exp\{-r(T-s)\} \gamma w_1(t,s,x,y) \left( \frac{\widehat{w}_1(s,s,y,y) + (\mu-r) U_y(s,s,x,y)|_{x=y}}{\widehat{w}_2(s,s,y,y) - \sigma^2 U_{yy}(s,s,x,y)|_{x=y}} \right) \\
        &\quad - \exp\{-2r(T-s)\} w_2(t,s,x,y) \left( \frac{\widehat{w}_1(s,s,y,y) + (\mu-r) U_y(s,s,x,y)|_{x=y}}{\widehat{w}_2(s,s,y,y) - \sigma^2 U_{yy}(s,s,x,y)|_{x=y}} \right)^2 \\
        &\quad - w_3(t,s,x) U(t,s,x,y) + w_4(t,s,x) = 0, \\[2.5ex]
        &U(t,T,x,y) = \gamma g_1(t,x) \exp\{\eta y\} - g_2(t,x) \exp\{-\eta y\} + g_3(t,x), \quad 0 \leq t \leq s \leq T.
    \end{aligned}
    \right.
\end{equation}
With similar arguments in Appendix B of \cite{Lei2021}, Theorem \ref{WellposednessWeighted} ensures that there exist $\delta\in(0,T]$ and a unique solution satisfying \eqref{BackwardHJBExponentialutility} in $\nabla[T-\delta,T]$.

Next, we consider the following ansatz for $U$: for $(t,s)\in\nabla[T-\delta,T],~x,y\in\mathbb{R}$,
\begin{equation*} 
	U(t,s,x,y)=\varphi_1(t,s,x)\exp\{\eta y\}-\varphi_2(t,s,x)\exp\{-\eta y\}+\varphi_3(t,s,x)
\end{equation*}
for some suitable $\varphi_1(\cdot,\cdot,\cdot)$, $\varphi_2(\cdot,\cdot,\cdot)$, and $\varphi_3(\cdot,\cdot,\cdot)$ to be solved. Consider also that  
\begin{equation*} 
    \begin{aligned}
        \widehat{w}_1(t,s,x,y) &= \gamma w_1(t,s,x,y) \exp\{-r(T-s)\} = \gamma W_1(t,s,x) \exp\{\eta y\} \exp\{-r(T-s)\}, \\
        \widehat{w}_2(t,s,x,y) &= 2w_2(t,s,x,y) \exp\{-2r(T-s)\} = \frac{\sigma^2\eta}{\mu-r} \widehat{w}_1(t,s,x,y) + 2\sigma^2\eta^2 \varphi_1(t,s,x) \exp\{\eta y\}.
    \end{aligned}
\end{equation*}
Under the assumptions, \eqref{Example1:Closedloopstrategy} admits the form 
\begin{equation} \label{ExpPolicy}
	\overline{\alpha}(s,y)=\frac{1}{\eta}\frac{\mu-r}{\sigma^2}\exp\{-r(T-s)\}. 
\end{equation}
Moreover, we obtain an ordinary differential equation (ODE) system parameterized by $(t,x)$: 
\begin{equation} \label{ExpODESys}
    \left\{
    \begin{aligned}
        &\varphi_s(t,s,x) + N(t,s,x)\varphi(t,s,x) + M(t,s,x) = 0, \\
        &\varphi(t,T,x) = g(t,x), \qquad (t,s,x) \in \nabla[T-\delta, T] \times \mathbb{R}, 
    \end{aligned}
    \right. 
\end{equation}
where $N=\mathrm{diag}\{N_1,N_2,N_3\}=\mathrm{diag}\left\{\frac{(\mu-r)^2}{2\sigma^2}-w_3,-\frac{(\mu-r)^2}{2\sigma^2}-w_3,-w_3\right\}$ and $M=\left(M_1,M_2,M_3\right)^\top=\left(\gamma\frac{(\mu-r)\exp\{-r(T-s)\}}{2\sigma^2\eta}W_1,0,w_4\right)^\top$. By the classical theory of ODEs, the system admits a unique solution represented by variation of constants formula. Consequently, the unique solution of the equilibrium HJB equation \eqref{BackwardHJBExponentialutility} has the following explicit representation: 
\begin{equation*} 
    \begin{aligned}
        U(t,s,x,y) &= \left[ \gamma \digamma_1(t,s,x) \digamma^{-1}_1(t,T,x) g_1(t,x) + \gamma \int^T_s \digamma_1(t,s,x) \digamma^{-1}_1(t,\tau,x) \overline{M}_1(t,\tau,x) d\tau \right] \exp\{\eta y\} \\
        &\quad - \digamma_2(t,s,x) \digamma^{-1}_2(t,T,x) g_2(t,x) \exp\{-\eta y\} \\
        &\quad + \left[ \digamma_3(t,s,x) \digamma^{-1}_3(t,T,x) g_3(t,x) + \int^T_s \digamma_3(t,s,x) \digamma^{-1}_3(t,\tau,x) w_4(t,\tau,x) d\tau \right]
    \end{aligned}
\end{equation*}
where $\digamma_i$ ($i=1,2,3$) is the fundamental matrix of the $i$-th ODE of \eqref{ExpODESys} and $\digamma^{-1}_i$ the associated inverse matrix. Note that this solution does not explode at $s=T-\delta$ such that we can update a new terminal condition at $s=T-\delta$. Consequently, one can repeat indefinitely the solving procedure up to a global solution for \eqref{BackwardHJBExponentialutility} over $\nabla[0,T]$. Furthermore, by sending $\gamma\to 0$, $t=s$, $x=y$, and denoting $\widetilde{\varphi}(t,s,x)=\varphi(t,s,x\exp\{r(T-t)\})$, one has  
\begin{equation} \label{ExpValueF}
    \begin{aligned}
        V(s,y) &= -\widetilde{\digamma}_2(s,s,y) \widetilde{g}(s,y) \exp\big\{ -\eta y \exp\{r(T-s)\} \big\} \\
        &\quad + \left[ \widetilde{\digamma}_3(s,s,y) \widetilde{h}(s,y) + \int^T_s \widetilde{\digamma}_3(s,s,y) \widetilde{\digamma}^{-1}_3(s,\tau,y) \widetilde{v}(s,\tau,y) d\tau \right]. 
    \end{aligned}
\end{equation}

\begin{proposition} \label{Exp1Propostion}
	Suppose that $v$, $w$, $g$, and $h$ are smooth enough, then the TIC stochastic control problem \eqref{ExpUFBSDE} admits a unique solution in $\nabla[0,T]$, and the closed-loop equilibrium strategy and the associated value function are given in \eqref{ExpPolicy} and \eqref{ExpValueF}.   
\end{proposition}


\section{Conclusion} 
\label{Sec:Conclusion}
For the TIC stochastic control problems with initial-time and -state dependent objectives, their well-posedness issues are shown to be equivalent to that of a class of nonlocal fully nonlinear PDEs, provided that the optimum of the Hamiltonian is attainable. We sequentially establish the global well-posedness of the linear and the fully nonlinear nonlocal PDEs. While the fully nonlinear case would require a sharp prior estimate for the global well-posedness, we show that its special case of nonlocal quasilinear PDEs, which correspond to the state processes with only drift being controlled, possess global well-posedness with mild conditions. On top of the well-posedness results, we also provide the probabilistic representation of the solution to the nonlocal fully nonlinear PDEs and an estimate on the difference between the value functions of the sophisticated and na\"{i}ve controllers (in the supplementary materials).

This work advances our understanding of the open problems raised in \cite{Bjoerk2017} and also provides new progress in the study of equilibrium-type HJB equations. Along this research direction, the following future research is promising: 1) extending our results from second-order PDE to a higher-order system (inspired by \cite{Lei2021}); 2) extending our results from Markovian to non-Markovian setting (referring to \cite{Hernandez2020} and Appendix \ref{app:probrep}); 3) expressing conditions \eqref{Upperboundedforglobalexistence} in terms of coefficients and data of the local and nonlocal fully nonlinear PDEs \eqref{NonlocalfullynonlinearPDE}.

\backmatter

\ \ 

\bmhead{Acknowledgments}
The second author, Chi Seng Pun, was supported in part by the Ministry of Education, Singapore under its AcRF Tier 2 grant (Reference No: MOE-T2EP20220-0013).

\ \

\appendix
\allowdisplaybreaks
\section{H\"{o}lder regularities of \texorpdfstring{$\boldsymbol{\partial_I u(s,s,x,y)|_{x=y}}$}{∂I u(s,s,x,y)|x=y}}
\label{App:A}
We investigate the H\"{o}lder continuities of $\big|\partial_I u(s,s,x,y)|_{x=y}\big|^{(\alpha)}_{(s,y)\in[0,\delta]\times\mathbb{R}^d}$ in $s$ and $y$ for $|I|=0,1,2$, which are required by Theorem \ref{Priorestimate} to obtain the Schauder estimate \eqref{Schauderestimate} of solutions of nonlocal linear PDE \eqref{NonlocalLinearPDE}. Upon the analysis of \eqref{Analysisofdifference}, we need to estimate all terms ($E_1$-$E_8$) in Table \ref{TableforlinearPDE}. Let us begin with $\big|\partial_I u(s,s,x,y)|_{x=y}\big|^{(0)}_{(s,y)\in[0,\delta]\times\mathbb{R}^d}$ for $|I|=0,1,2$. 

\ \ 

\noindent\textbf{(Estimate for $\boldsymbol{E_1}$)} By making use of integral representations \eqref{Integralforoverleftarrowu}-\eqref{Integralforpartialoverleftarrowu}, for any fixed $(t,x)\in[0,\delta]\times\mathbb{R}^d$ and $|I|=0$, we have 
\begin{equation*}
        \begin{aligned}\bigl| \overleftarrow{u}(t,s,x,y) \bigr| 
        \leq \;& C \int_0^s \int_{\mathbb{R}^d} (s-\tau)^{-\frac{d}{2}} \exp \Bigl\{ -c \varpi(s,\tau,y,\xi) \Bigr\} \sum_{|I| \leq 2} \left| \mathcal{I}^I \left[ \frac{\partial u}{\partial t}, \frac{\partial u}{\partial x} \right] (t,\tau,x,\xi) \right| d\xi d\tau \\
        & + Cs \left\| \overleftarrow{f}(t,s,x,y) \right\|^{(\alpha)}_{(s,y) \in [0,\delta] \times \mathbb{R}^d}
    \end{aligned}
\end{equation*}
where the constants $C$ and $c$ are independent of $(t,s,x,y)$ while they only depend on $\lVert A^I\rVert^{(\alpha)}_{[0,T]}$ and $\lVert B^I\rVert^{(\alpha)}_{[0,T]}$, and $\varpi(s,\tau,y,\xi)=\sum^d_{i=1}|y_i-\xi_i|^2(s-\tau)^{-1}$. Consequently, when $(t,x)=(s,y)\in[0,\delta]\times\mathbb{R}^d$, it is clear that
\begin{equation} \label{E1}
        \begin{aligned}& \bigl| \overleftarrow{u}(s,s,y,y) \bigr| \\
        \leq \; & C \int_0^s \int_{\mathbb{R}^d} (s-\tau)^{-\frac{d}{2}} \exp \{ -c \varpi \} \bigl( |s-\tau| + |y-\xi| \bigr) [ \overrightarrow{u} ]^{(2+\alpha)}_{[0, \delta]} \, d\xi d\tau + Cs \lVert f \rVert^{(\alpha)}_{[0, \delta]} \\
        \leq \; & C (\delta^2 + \delta^{\frac{3}{2}}) [ \overrightarrow{u} ]^{(2+\alpha)}_{[0, \delta]} + C \delta \lVert f \rVert^{(\alpha)}_{[0, \delta]}
    \end{aligned}
\end{equation}

\ \ 

\noindent\textbf{(Estimate for $\boldsymbol{E_2}$)} For any fixed $(t,x)\in[0,\delta]\times\mathbb{R}^d$ and $|I|=1,2$, we have 
\begin{equation*}
        \begin{aligned}& \bigl| \partial_I \overleftarrow{u}(t,s,x,y) \bigr| \\
        \leq \;& C \int_0^s \int_{\mathbb{R}^d} (s-\tau)^{-\frac{d+|I|}{2}} \exp \bigl\{ -c \varpi(s,\tau,y,\xi) \bigr\} \\
        & \qquad \times \sum_{|I|\leq 2} \left| (\overleftarrow{B}^I\mathcal{I}^I)(t,\tau,x,\xi) - (\overleftarrow{B}^I\mathcal{I}^I)(t,\tau,x,y) \right| d\xi d\tau \\
        & + C \int_0^s (s-\tau)^{-\frac{|I|-\alpha}{2}} \sum_{|I|\leq 2} \bigl| \mathcal{I}^I(t,\tau,x,y) \bigr| d\tau + Cs^{\frac{2-|I|+\alpha}{2}} \bigl| \overleftarrow{f}(t,s,x,y) \bigr|^{(\alpha)}_{(s,y) \in [0,\delta] \times \mathbb{R}^d} \\
        \leq \;& C \int_0^s \int_{\mathbb{R}^d} (s-\tau)^{-\frac{d+|I|}{2}} \exp \bigl\{ -c \varpi(s,\tau,y,\xi) \bigr\} \\
        & \qquad \times \sum_{|I|\leq 2} \left| \overleftarrow{B}^I(t,\tau,x,\xi) - \overleftarrow{B}^I(t,\tau,x,y) \right| \bigl| \mathcal{I}^I(t,\tau,x,\xi) \bigr| d\xi d\tau \\
        & + C \int_0^s \int_{\mathbb{R}^d} (s-\tau)^{-\frac{d+|I|}{2}} \exp \bigl\{ -c \varpi(s,\tau,y,\xi) \bigr\} \\
        & \qquad \times \sum_{|I|\leq 2} \left| \mathcal{I}^I(t,\tau,x,\xi) - \mathcal{I}^I(t,\tau,x,y) \right| \bigl| \overleftarrow{B}^I(t,\tau,x,y) \bigr| d\xi d\tau \\
        & + C \int_0^s (s-\tau)^{-\frac{|I|-\alpha}{2}} \sum_{|I|\leq 2} \bigl| \mathcal{I}^I(t,\tau,x,y) \bigr| d\tau + Cs^{\frac{2-|I|+\alpha}{2}} \bigl| \overleftarrow{f}(t,s,x,y) \bigr|^{(\alpha)}_{(s,y) \in [0,\delta] \times \mathbb{R}^d}
    \end{aligned}
\end{equation*}
Hence, when $t=s$ and $x=y$, it holds that  
\begin{equation*}
		\begin{aligned}& \bigl| \partial_I \overleftarrow{u}(s,s,x,y) \bigr|_{x=y} \bigr| \\
		\leq \;& C \int_0^s \int_{\mathbb{R}^d} (s-\tau)^{-\frac{d+|I|}{2}} \exp \bigl\{ -c \varpi(s, \tau, y, \xi) \bigr\} |y-\xi|^\alpha \bigl( |s-\tau| + |y-\xi| \bigr) [ \overrightarrow{u} ]^{(2+\alpha)}_{[0, \delta]} \, d\xi d\tau \\
		& + C \int_0^s \int_{\mathbb{R}^d} (s-\tau)^{-\frac{d+|I|}{2}} \exp \bigl\{ -c \varpi(s, \tau, y, \xi) \bigr\} \\
		& \quad \times \sum_{|I| \leq 2} \Biggl\{ \left| \int_\tau^s \partial_I \Bigl( \frac{\partial u}{\partial t} \Bigr) (\theta_t, \tau, x, \xi) \bigr|_{x=y} \, d\theta_t - \int_\tau^s \partial_I \Bigl( \frac{\partial u}{\partial t} \Bigr) (\theta_t, \tau, x, y) \bigr|_{x=y} \, d\theta_t \right| \\
	\end{aligned}
\end{equation*}
\begin{equation*}
		\begin{aligned}
		& \quad + \sum_{1 \leq i \leq d} \left| \int_{\xi_i}^{y_i} \partial_I \Bigl( \frac{\partial u}{\partial x_i} \Bigr) (\tau, \tau, x_1, \dots, \theta_i, \dots, x_d, \xi) \Big|_{ \begin{subarray}{c} x_j=\xi_j, j<i \\ x_j=y_j, j>i \end{subarray} } d\theta_i - 0 \right| \Biggr\} d\xi d\tau \\
		& + C \int_0^s (s-\tau)^{-\frac{|I|-\alpha}{2}} \sum_{|I| \leq 2} \bigl| \mathcal{I}^I(t, \tau, x, y) \bigr| \, d\tau + Cs^{\frac{2-|I|+\alpha}{2}} \lVert f \rVert^{(\alpha)}_{[0, \delta]}
	\end{aligned}
\end{equation*}
Subsequently, it follows that   
\begin{equation} \label{E2}
		\begin{aligned}& \bigl| \partial_I \overleftarrow{u}(s,s,x,y) \big|_{x=y} \bigr| \\
		\leq \;& C \int_0^s \Bigl( (s-\tau)^{-\frac{|I|-\alpha-2}{2}} + (s-\tau)^{-\frac{|I|-\alpha-1}{2}} \Bigr) [ \overrightarrow{u} ]^{(2+\alpha)}_{[0,\delta]} \, d\tau \\
		& + C \int_0^s \int_{\mathbb{R}^d} (s-\tau)^{-\frac{d+|I|}{2}} \exp \bigl\{ -c\varpi(s,\tau,y,\xi) \bigr\} \Bigl\{ (s-\tau)|y-\xi|^\alpha + |y-\xi| \Bigr\} [ \overrightarrow{u} ]^{(2+\alpha)}_{[0,\delta]} \, d\xi d\tau \\
		& + C \int_0^s (s-\tau)^{-\frac{|I|-\alpha}{2}} (s-\tau) [ \overrightarrow{u} ]^{(2+\alpha)}_{[0,\delta]} \, d\tau + Cs^{\frac{2-|I|+\alpha}{2}} \lVert f \rVert^{(\alpha)}_{[0,\delta]} \\
		\leq \;& C \int_0^s \Bigl( (s-\tau)^{-\frac{|I|-\alpha-2}{2}} + (s-\tau)^{-\frac{|I|-\alpha-1}{2}} \Bigr) [ \overrightarrow{u} ]^{(2+\alpha)}_{[0,\delta]} \, d\tau \\
		& + C \int_0^s \Bigl( (s-\tau)^{-\frac{|I|-\alpha-2}{2}} + (s-\tau)^{-\frac{|I|-1}{2}} \Bigr) [ \overrightarrow{u} ]^{(2+\alpha)}_{[0,\delta]} \, d\tau \\
		& + C \int_0^s (s-\tau)^{-\frac{|I|-\alpha}{2}} (s-\tau) [ \overrightarrow{u} ]^{(2+\alpha)}_{[0,\delta]} \, d\tau + Cs^{\frac{2-|I|+\alpha}{2}} \lVert f \rVert^{(\alpha)}_{[0,\delta]} \\
		\leq \;& C \Bigl( \delta^{\frac{4-|I|+\alpha}{2}} + \delta^{\frac{3-|I|+\alpha}{2}} + \delta^{\frac{3-|I|}{2}} \Bigr) [ \overrightarrow{u} ]^{(2+\alpha)}_{[0,\delta]} + C \delta^{\frac{2-|I|+\alpha}{2}} \lVert f \rVert^{(\alpha)}_{[0,\delta]}
	\end{aligned}
\end{equation}

With the estimates \eqref{E1} and \eqref{E2} of $|\partial_I\overleftarrow{u}(s,s,x,y)|_{x=y}|$ for $|I|=0,1,2$, we find that $|\partial_I u(s,s,x,y)|_{x=y}|$ is bounded by $[\overrightarrow{u}]^{(2+\alpha)}_{[0,\delta]}$ and $\lVert f\rVert^{(\alpha)}_{[0,\delta]}$. Moreover, the coefficient $\delta$ in front of $[\overrightarrow{u}]^{(2+\alpha)}_{[0,\delta]}$ in \eqref{E1} and \eqref{E2} could be significant while suitably small $\delta\in(0,T]$ to establish $\frac{1}{2}$ of \eqref{Estimateofu}.

\ \ 

Next, we will also show that $E_3-E_8$ possess similar properties.  


\ \ 

\noindent\textbf{(Estimate for $\boldsymbol{E_3}$)} For $|I|=0$, by \eqref{Integralforoverleftarrowu}, we have 
\begin{equation} \label{E3}
		\begin{aligned}& \bigl| \overleftarrow{u}(\eta_t, s', \eta_x, y') \bigr| \\
		\leq \;& C \int_0^{s'} \int_{\mathbb{R}^d} (s'-\tau)^{-\frac{d}{2}} \exp \bigl\{ -c\varpi(s', \tau, y', \xi) \bigr\} \\
		& \quad \times \sum_{|I| \leq 2} \Biggl\{ \left| \int_\tau^{\eta_t} \partial_{I} \Bigl( \frac{\partial u}{\partial t} \Bigr) (\theta_t, \tau, x, \xi) \Big|_{x=\eta_x} d\theta_t \right| \\
		& \quad + \sum_{1 \leq i \leq d} \left| \int_{\xi_i}^{(\eta_x)_i} \partial_{I} \Bigl( \frac{\partial u}{\partial x_i} \Bigr) (\tau, \tau, \dots, \theta_i, \dots, \xi) \Big|_{ \begin{subarray}{c} x_j=\xi_j, j<i, \\ x_j=(\eta_x)_j, j>i \end{subarray} } d\theta_i \right| \Biggr\} \, d\xi d\tau + Cs' \lVert f \rVert^{(\alpha)}_{[0, \delta]} \\
	\end{aligned}
\end{equation}
\begin{equation*}
		\begin{aligned}
		\leq \;& C \int_0^{s'} \int_{\mathbb{R}^d} (s'-\tau)^{-\frac{d}{2}} \exp \bigl\{ -c\varpi(s', \tau, y', \xi) \bigr\} \bigl( |\eta_t-\tau| + |\eta_x-\xi| \bigr) [ \overrightarrow{u} ]^{(2+\alpha)}_{[0, \delta]} \, d\xi d\tau + Cs' \lVert f \rVert^{(\alpha)}_{[0, \delta]} \\
		\leq \;& C \int_0^{s'} \int_{\mathbb{R}^d} (s'-\tau)^{-\frac{d}{2}} \exp \bigl\{ -c\varpi(s', \tau, y', \xi) \bigr\} \\
		& \quad \times \bigl( |\eta_t-s'| + |s'-\tau| + |\eta_x-y'| + |y'-\xi| \bigr) [ \overrightarrow{u} ]^{(2+\alpha)}_{[0, \delta]} \, d\xi d\tau + Cs' \lVert f \rVert^{(\alpha)}_{[0, \delta]} \\
		\leq \;& C \bigl( |\eta_t-s'||s'| + |s'|^2 + |\eta_x-y'||s'| + |s'|^{\frac{3}{2}} \bigr) [ \overrightarrow{u} ]^{(2+\alpha)}_{[0, \delta]} + Cs' \lVert f \rVert^{(\alpha)}_{[0, \delta]} \\
		\leq \;& C \bigl( \delta^2 + \delta + \delta^{\frac{3}{2}} \bigr) [ \overrightarrow{u} ]^{(2+\alpha)}_{[0, \delta]} + C\delta \lVert f \rVert^{(\alpha)}_{[0, \delta]}
	\end{aligned}
\end{equation*}

\ \ 

\noindent\textbf{(Estimate for $\boldsymbol{E_4}$)} For $|I|=1,2$, the representation \eqref{Integralforpartialoverleftarrowu} implies that  
\begin{equation*}
		\begin{aligned}& \bigl| \partial_I \overleftarrow{u}(\eta_t, s', x, y') \bigr|_{x=\eta_x} \bigr| \\
		\leq \;& C \int_0^{s'} \int_{\mathbb{R}^d} (s'-\tau)^{-\frac{d+|I|}{2}} \exp \bigl\{ -c \varpi(s', \tau, y', \xi) \bigr\} \\
		& \quad \times \sum_{|I| \leq 2} \left| \bigl(\overleftarrow{B}^I \mathcal{I}^I \bigr)(\eta_t, \tau, \eta_x, \xi) - \bigl(\overleftarrow{B}^I \mathcal{I}^I \bigr)(\eta_t, \tau, \eta_x, y') \right| d\xi d\tau \\
		& + C \int_0^{s'} (s'-\tau)^{-\frac{|I|-\alpha}{2}} \sum_{|I| \leq 2} \bigl| \mathcal{I}^I(\eta_t, \tau, \eta_x, y') \bigr| d\tau + C |s'|^{\frac{2-|I|+\alpha}{2}} \lVert f \rVert^{(\alpha)}_{[0, \delta]} \\
		\leq \;& C \int_0^{s'} \int_{\mathbb{R}^d} (s'-\tau)^{-\frac{d+|I|}{2}} \exp \bigl\{ -c \varpi(s', \tau, y', \xi) \bigr\} \\
		& \quad \times \sum_{|I| \leq 2} \left| \overleftarrow{B}^I(\eta_t, \tau, \eta_x, \xi) - \overleftarrow{B}^I(\eta_t, \tau, \eta_x, y') \right| \bigl| \mathcal{I}^I(\eta_t, \tau, \eta_x, \xi) \bigr| d\xi d\tau \\
		& + C \int_0^{s'} \int_{\mathbb{R}^d} (s'-\tau)^{-\frac{d+|I|}{2}} \exp \bigl\{ -c \varpi(s', \tau, y', \xi) \bigr\} \\
		& \quad \times \sum_{|I| \leq 2} \left| \mathcal{I}^I(\eta_t, \tau, \eta_x, \xi) - \mathcal{I}^I(\eta_t, \tau, \eta_x, y') \right| \bigl| \overleftarrow{B}^I(\eta_t, \tau, \eta_x, y') \bigr| d\xi d\tau \\
		& + C \int_0^{s'} (s'-\tau)^{-\frac{|I|-\alpha}{2}} \sum_{|I| \leq 2} \bigl| \mathcal{I}^I(\eta_t, \tau, \eta_x, y') \bigr| d\tau + C |s'|^{\frac{2-|I|+\alpha}{2}} \lVert f \rVert^{(\alpha)}_{[0, \delta]}
	\end{aligned}
\end{equation*}
By making use of the representation of \eqref{Intergralrepresentations}, we have 
\begingroup\small
\begin{equation*}
        \begin{aligned}& \bigl| \partial_I \overleftarrow{u}(\eta_t, s', x, y') \bigr|_{x=\eta_x} \bigr| \\
        \leq \;& C \int_0^{s'} \int_{\mathbb{R}^d} (s'-\tau)^{-\frac{d+|I|}{2}} \exp \bigl\{ -c \varpi(s', \tau, y', \xi) \bigr\} |y'-\xi|^\alpha \bigl( |\eta_t-\tau| + |\eta_x-\xi| \bigr) [ \overrightarrow{u} ]^{(2+\alpha)}_{[0, \delta]} \, d\xi d\tau \\
        & + C \int_0^{s'} \int_{\mathbb{R}^d} (s'-\tau)^{-\frac{d+|I|}{2}} \exp \bigl\{ -c \varpi(s', \tau, y', \xi) \bigr\} \\
        & \quad \times \sum_{|I| \leq 2} \Biggl\{ \left| \int_\tau^{\eta_t} \partial_I \Bigl( \frac{\partial u}{\partial t} \Bigr) (\theta_t, \tau, x, \xi) \Big|_{x=\eta_x} d\theta_t - \int_\tau^{\eta_t} \partial_I \Bigl( \frac{\partial u}{\partial t} \Bigr) (\theta_t, \tau, x, y') \Big|_{x=\eta_x} d\theta_t \right| \\
    \end{aligned}
\end{equation*}
\begin{equation*}
        \begin{aligned}
        & \quad + \sum_{1 \leq i \leq d} \left| \int_{\xi_i}^{(\eta_x)_i} \partial_{I} \Bigl( \frac{\partial u}{\partial x_i} \Bigr) (\dots) \Big|_{ \begin{subarray}{c} x_j=\xi_j, j<i \\ x_j=(\eta_x)_j, j>i \end{subarray} } d\theta_i - \int_{y'_i}^{(\eta_x)_i} \partial_{I} \Bigl( \frac{\partial u}{\partial x_i} \Bigr) (\dots) \Big|_{ \begin{subarray}{c} x_j=y'_j, j<i \\ x_j=(\eta_x)_j, j>i \end{subarray} } d\theta_i \right| \Biggr\} \, d\xi d\tau \\
        & + C \int_0^{s'} (s'-\tau)^{-\frac{|I|-\alpha}{2}} \bigl( |\eta_t-\tau| + |\eta_x-y'| \bigr) [ \overrightarrow{u} ]^{(2+\alpha)}_{[0, \delta]} \, d\tau + C|s'|^{\frac{2-|I|+\alpha}{2}} \lVert f \rVert^{(\alpha)}_{[0, \delta]}
    \end{aligned}
\end{equation*}
\endgroup
Moreover, we have  
\begingroup\small
\begin{equation*}
		\begin{aligned}& \left| \int_{\xi_i}^{(\eta_x)_i} \partial_I \Bigl( \frac{\partial u}{\partial x_i} \Bigr) (\tau, \tau, \dots, \theta_i, \dots, \xi) \Big|_{ \begin{subarray}{c} x_j=\xi_j, j<i \\ x_j=(\eta_x)_j, j>i \end{subarray} } d\theta_i \right. \\
		& \qquad \left. - \int_{y'_i}^{(\eta_x)_i} \partial_I \Bigl( \frac{\partial u}{\partial x_i} \Bigr) (\tau, \tau, \dots, \theta_i, \dots, y') \Big|_{ \begin{subarray}{c} x_j=y'_j, j<i \\ x_j=(\eta_x)_j, j>i \end{subarray} } d\theta_i \right| \\
		\leq \; & \left| \int_{\xi_i}^{(\eta_x)_i} \partial_I \Bigl( \frac{\partial u}{\partial x_i} \Bigr) (\tau, \tau, \dots, \xi) \Big|_{ \begin{subarray}{c} x_j=\xi_j, j<i \\ x_j=(\eta_x)_j, j>i \end{subarray} } d\theta_i \right. \\
		& \qquad \left. - \int_{y'_i}^{(\eta_x)_i} \partial_I \Bigl( \frac{\partial u}{\partial x_i} \Bigr) (\tau, \tau, \dots, \xi) \Big|_{ \begin{subarray}{c} x_j=\xi_j, j<i \\ x_j=(\eta_x)_j, j>i \end{subarray} } d\theta_i \right| \\
		& + \left| \int_{y'_i}^{(\eta_x)_i} \partial_I \Bigl( \frac{\partial u}{\partial x_i} \Bigr) (\tau, \tau, \dots, \xi) \Big|_{ \begin{subarray}{c} x_j=\xi_j, j<i \\ x_j=(\eta_x)_j, j>i \end{subarray} } d\theta_i \right. \\
		& \qquad \left. - \int_{y'_i}^{(\eta_x)_i} \partial_I \Bigl( \frac{\partial u}{\partial x_i} \Bigr) (\tau, \tau, \dots, \xi) \Big|_{ \begin{subarray}{c} x_j=y'_j, j<i \\ x_j=(\eta_x)_j, j>i \end{subarray} } d\theta_i \right| \\
		& + \left| \int_{y'_i}^{(\eta_x)_i} \partial_I \Bigl( \frac{\partial u}{\partial x_i} \Bigr) (\tau, \tau, \dots, \xi) \Big|_{ \begin{subarray}{c} x_j=y'_j, j<i \\ x_j=(\eta_x)_j, j>i \end{subarray} } d\theta_i \right. \\
		& \qquad \left. - \int_{y'_i}^{(\eta_x)_i} \partial_I \Bigl( \frac{\partial u}{\partial x_i} \Bigr) (\tau, \tau, \dots, y') \Big|_{ \begin{subarray}{c} x_j=y'_j, j<i \\ x_j=(\eta_x)_j, j>i \end{subarray} } d\theta_i \right| \\
		\leq \; & \Bigl( |y' - \xi| + |\eta_x - y'||y' - \xi| + |\eta_x - y'||y' - \xi|^\alpha \Bigr) [ \overrightarrow{u} ]^{(2+\alpha)}_{[0, \delta]}
	\end{aligned}
\end{equation*}
\endgroup
Consequently, we obtain
\begin{equation*}
		\begin{aligned}& \bigl| \partial_I \overleftarrow{u}(\eta_t, s', x, y') \bigr|_{x=\eta_x} \bigr| \\
		\leq \;& C \int_0^{s'} \int_{\mathbb{R}^d} (s'-\tau)^{-\frac{d+|I|}{2}} \exp \bigl\{ -c \varpi(s', \tau, y', \xi) \bigr\} (y'-\xi)^\alpha \bigl( |\eta_t-\tau| + |\eta_x-\xi| \bigr) [ \overrightarrow{u} ]^{(2+\alpha)}_{[0, \delta]} \, d\xi d\tau \\
		& + C \int_0^{s'} \int_{\mathbb{R}^d} (s'-\tau)^{-\frac{d+|I|}{2}} \exp \bigl\{ -c \varpi(s', \tau, y', \xi) \bigr\} \\
		& \quad \times \Bigl\{ |\eta_t-\tau| |y'-\xi|^\alpha + |y'-\xi| + |\eta_x-y'||y'-\xi| + |\eta_x-y'||y'-\xi|^\alpha \Bigr\} [ \overrightarrow{u} ]^{(2+\alpha)}_{[0, \delta]} \, d\xi d\tau \\
		& + C \int_0^{s'} (s'-\tau)^{-\frac{|I|-\alpha}{2}} \bigl( |\eta_t-\tau| + |\eta_x-y'| \bigr) [ \overrightarrow{u} ]^{(2+\alpha)}_{[0, \delta]} \, d\tau + C |s'|^{\frac{2-|I|+\alpha}{2}} \lVert f \rVert^{(\alpha)}_{[0, \delta]} \\
	\end{aligned}
\end{equation*}
\begin{equation*}
		\begin{aligned}
		\leq \;& C \int_0^{s'} \int_{\mathbb{R}^d} (s'-\tau)^{-\frac{d+|I|}{2}} \exp \bigl\{ -c \varpi(s', \tau, y', \xi) \bigr\} \\
		& \quad \times (y'-\xi)^\alpha \bigl( |\eta_t-s'| + |s'-\tau| + |\eta_x-y'| + |y'-\xi| \bigr) [ \overrightarrow{u} ]^{(2+\alpha)}_{[0, \delta]} \, d\xi d\tau \\
		& + C \int_0^{s'} \int_{\mathbb{R}^d} (s'-\tau)^{-\frac{d+|I|}{2}} \exp \bigl\{ -c \varpi(s', \tau, y', \xi) \bigr\} \\
		& \quad \times \Bigl\{ \bigl( |\eta_t-s'| + |s'-\tau| \bigr) |y'-\xi|^\alpha + |y'-\xi| + |\eta_x-y'||y'-\xi| + |\eta_x-y'||y'-\xi|^\alpha \Bigr\} [ \overrightarrow{u} ]^{(2+\alpha)}_{[0, \delta]} \, d\xi d\tau \\
		& + C \int_0^{s'} (s'-\tau)^{-\frac{|I|-\alpha}{2}} \bigl( |\eta_t-s'| + |s'-\tau| + |\eta_x-y'| \bigr) [ \overrightarrow{u} ]^{(2+\alpha)}_{[0, \delta]} \, d\tau + C |s'|^{\frac{2-|I|+\alpha}{2}} \lVert f \rVert^{(\alpha)}_{[0, \delta]}
	\end{aligned}
\end{equation*}
Simple calculation yields that 
\begin{equation} \label{E4}
		\begin{aligned}& \bigl| \partial_I \overleftarrow{u}(\eta_t, s', x, y') \bigr|_{x=\eta_x} \bigr| \\
		\leq \;& C \Bigl( |\eta_t-s'| |s'|^{\frac{4-|I|+\alpha}{2}} + |s'|^{\frac{4-|I|+\alpha}{2}} + |\eta_x-y'| |s'|^{\frac{2-|I|+\alpha}{2}} + |s'|^{\frac{3-|I|+\alpha}{2}} \Bigr) [ \overrightarrow{u} ]^{(2+\alpha)}_{[0, \delta]} \\
		& + C \Bigl( |\eta_t-s'| |s'|^{\frac{2-|I|+\alpha}{2}} + |s'|^{\frac{4-|I|+\alpha}{2}} + |s'|^{\frac{3-|I|}{2}} + |\eta_x-y'| |s'|^{\frac{3-|I|}{2}} + |\eta_x-y'| |s'|^{\frac{2-|I|+\alpha}{2}} \Bigr) [ \overrightarrow{u} ]^{(2+\alpha)}_{[0, \delta]} \\
		& + C \Bigl( |\eta_t-s'| |s'|^{\frac{2-|I|+\alpha}{2}} + |s'|^{\frac{4-|I|+\alpha}{2}} + |\eta_x-y'| |s'|^{\frac{2-|I|+\alpha}{2}} \Bigr) [ \overrightarrow{u} ]^{(2+\alpha)}_{[0, \delta]} + C |s'|^{\frac{2-|I|+\alpha}{2}} \lVert f \rVert^{(\alpha)}_{[0, \delta]} \\
		\leq \;& C \Bigl( \delta^{\frac{6-|I|+\alpha}{2}} + \delta^{\frac{4-|I|+\alpha}{2}} + \delta^{\frac{3-|I|+\alpha}{2}} + \delta^{\frac{2-|I|+\alpha}{2}} + \delta^{\frac{3-|I|}{2}} \Bigr) [ \overrightarrow{u} ]^{(2+\alpha)}_{[0, \delta]} + C \delta^{\frac{2-|I|+\alpha}{2}} \lVert f \rVert^{(\alpha)}_{[0, \delta]}
	\end{aligned}
\end{equation}

\ \ 

After the analyses of $E_3$-$E_4$, namely \eqref{E3} and \eqref{E4}, we turn to investigate the difference quotient of \eqref{Analysisofdifference}, i.e. $\frac{|\partial_I u(s,s^\prime,x,y^\prime)|_{x=y}-\partial_I u(s,s,x,y)|_{x=y}|}{|s^\prime-s|^\frac{\alpha}{2}+|y^\prime-y|^\alpha}$. In what follows, we need to study the difference between $\partial_I u(s,s^\prime,x,y^\prime)|_{x=y}$ and $\partial_I u(s,s,x,y)|_{x=y}$ for $|I|=0,1,2$. 

\ \ 

\noindent\textbf{(Estimate for $\boldsymbol{E_5}$)} First of all, we consider the difference in the case that $s\leq \rho^2$. Then, for $|I|=0$, the estimate of $E_3$ tells us that  
\begin{equation*}
    \bigl| \overleftarrow{u}(s,s',y,y') \bigr| \leq C \Bigl( |s-s'| |s'| + |s'|^2 + |y-y'| |s'| + |s'|^{\frac{3}{2}} \Bigr) [ \overrightarrow{u} ]^{(2+\alpha)}_{[0, \delta]} + C |s'| \lVert f \rVert^{(\alpha)}_{[0, \delta]}
\end{equation*}
In addition, from $E_1$, we also have 
\begin{equation*}
		\begin{aligned}& \bigl| \overleftarrow{u}(s,s,y,y) \bigr| \\
		\leq \;& C \int_0^s \int_{\mathbb{R}^d} (s-\tau)^{-\frac{d}{2}} \exp \bigl\{ -c\varpi \bigr\} \bigl( |s-\tau| + |y-\xi| \bigr) [ \overrightarrow{u} ]^{(2+\alpha)}_{[0,\delta]} \, d\xi d\tau + Cs \lVert f \rVert^{(\alpha)}_{[0,\delta]} \\
		\leq \;& C (s^2 + s^{\frac{3}{2}}) [ \overrightarrow{u} ]^{(2+\alpha)}_{[0,\delta]} + Cs \lVert f \rVert^{(\alpha)}_{[0,\delta]}
	\end{aligned}
\end{equation*}
Consequently, it follows that 
\begin{equation} \label{E5}
		\begin{aligned}& \bigl| \overleftarrow{u}(s,s',y,y') - \overleftarrow{u}(s,s,y,y) \bigr| \\
		\leq \;& C \Bigl( |s-s'| |s'| + |s'|^2 + |y-y'| |s'| + |s'|^{\frac{3}{2}} \Bigr) [ \overrightarrow{u} ]^{(2+\alpha)}_{[0, \delta]} + C |s'| \lVert f \rVert^{(\alpha)}_{[0, \delta]} \\
		& + C \bigl( s^2 + s^{\frac{3}{2}} \bigr) [ \overrightarrow{u} ]^{(2+\alpha)}_{[0, \delta]} + Cs \lVert f \rVert^{(\alpha)}_{[0, \delta]} \\
		\leq \;& C \Bigl( |s-s'| |s'| + |s'| |s'-s+s| + |y-y'| |s'| + |s'|^{\frac{1}{2}} |s'-s+s| \Bigr) [ \overrightarrow{u} ]^{(2+\alpha)}_{[0, \delta]} \\
		& + C \bigl( s^2 + s^{\frac{3}{2}} \bigr) [ \overrightarrow{u} ]^{(2+\alpha)}_{[0, \delta]} + Cs \lVert f \rVert^{(\alpha)}_{[0, \delta]} + C |s'-s+s| \lVert f \rVert^{(\alpha)}_{[0, \delta]} \\
		\leq \;& C \delta^{\frac{1}{2}} \rho^\alpha [ \overrightarrow{u} ]^{(2+\alpha)}_{[0, \delta]} + C \rho^\alpha \lVert f \rVert^{(\alpha)}_{[0, \delta]}
	\end{aligned}
\end{equation}

\ \

\noindent\textbf{(Estimate for $\boldsymbol{E_6}$)} Similarly, for $s\leq \rho^2$ and $|I|=1,2$, the term $E_4$ yields that 
\begin{equation*}
		\begin{aligned}& \bigl| \partial_I \overleftarrow{u}(s, s', x, y') \big|_{x=y} \bigr| \\
		\leq \;& C \Bigl( |s-s'| |s'|^{\frac{4-|I|+\alpha}{2}} + |s'|^{\frac{4-|I|+\alpha}{2}} + |y-y'| |s'|^{\frac{2-|I|+\alpha}{2}} + |s'|^{\frac{3-|I|+\alpha}{2}} \Bigr) [ \overrightarrow{u} ]^{(2+\alpha)}_{[0, \delta]} \\
		& + C \Bigl( |s-s'| |s'|^{\frac{2-|I|+\alpha}{2}} + |s'|^{\frac{4-|I|+\alpha}{2}} + |s'|^{\frac{3-|I|}{2}} + |y-y'| |s'|^{\frac{3-|I|}{2}} + |y-y'| |s'|^{\frac{2-|I|+\alpha}{2}} \Bigr) [ \overrightarrow{u} ]^{(2+\alpha)}_{[0, \delta]} \\
		& + C \Bigl( |s-s'| |s'|^{\frac{2-|I|+\alpha}{2}} + |s'|^{\frac{4-|I|+\alpha}{2}} + |y-y'| |s'|^{\frac{2-|I|+\alpha}{2}} \Bigr) [ \overrightarrow{u} ]^{(2+\alpha)}_{[0, \delta]} + C |s'|^{\frac{2-|I|+\alpha}{2}} \lVert f \rVert^{(\alpha)}_{[0, \delta]}
	\end{aligned}
\end{equation*}
Moreover, from $E_2$, it is clear that  
\begin{equation*}
		\begin{aligned}& \bigl| \partial_I \overleftarrow{u}(s, s, x, y) \big|_{x=y} \bigr| \\
		\leq \;& C \int_0^s \Bigl( (s-\tau)^{-\frac{|I|-\alpha-2}{2}} + (s-\tau)^{-\frac{|I|-\alpha-1}{2}} \Bigr) [ \overrightarrow{u} ]^{(2+\alpha)}_{[0, \delta]} \, d\tau \\
		& + C \int_0^s \Bigl( (s-\tau)^{-\frac{|I|-\alpha-2}{2}} + (s-\tau)^{-\frac{|I|-1}{2}} \Bigr) [ \overrightarrow{u} ]^{(2+\alpha)}_{[0, \delta]} \, d\tau \\
		& + C \int_0^s (s-\tau)^{-\frac{|I|-\alpha}{2}} (s-\tau) [ \overrightarrow{u} ]^{(2+\alpha)}_{[0, \delta]} \, d\tau + C s^{\frac{2-|I|+\alpha}{2}} \lVert f \rVert^{(\alpha)}_{[0, \delta]} \\
		\leq \;& C \Bigl( s^{\frac{4-|I|+\alpha}{2}} + s^{\frac{3-|I|+\alpha}{2}} + s^{\frac{3-|I|}{2}} \Bigr) [ \overrightarrow{u} ]^{(2+\alpha)}_{[0, \delta]} + C s^{\frac{2-|I|+\alpha}{2}} \lVert f \rVert^{(\alpha)}_{[0, \delta]}
	\end{aligned}
\end{equation*}
Hence, it holds that 
\begingroup\small
\begin{equation*} \label{E6}
		\begin{aligned}& \bigl| \partial_I \overleftarrow{u}(s,s',x,y') \big|_{x=y} - \partial_I \overleftarrow{u}(s,s,x,y) \big|_{x=y} \bigr| \\
		\leq \;& \bigl| \partial_I \overleftarrow{u}(s,s',x,y') \big|_{x=y} \bigr| + \bigl| \partial_I \overleftarrow{u}(s,s,x,y) \big|_{x=y} \bigr| \\
		\leq \;& C \Bigl( |s-s'| |s'|^{\frac{4-|I|+\alpha}{2}} + |s'|^{\frac{4-|I|}{2}} |s'-s+s|^{\frac{\alpha}{2}} + |y-y'| |s'|^{\frac{2-|I|+\alpha}{2}} + |s'|^{\frac{3-|I|}{2}} |s'-s+s|^{\frac{\alpha}{2}} \Bigr) [ \overrightarrow{u} ]^{(2+\alpha)}_{[0, \delta]} \\
		& + C \Bigl( |s-s'| |s'|^{\frac{2-|I|+\alpha}{2}} + |s'|^{\frac{4-|I|}{2}} |s'-s+s|^{\frac{\alpha}{2}} + |s'|^{\frac{3-|I|-\alpha}{2}} |s'-s+s|^{\frac{\alpha}{2}} \\
		& \quad \qquad + |y-y'| |s'|^{\frac{3-|I|}{2}} + |y-y'| |s'|^{\frac{2-|I|+\alpha}{2}} \Bigr) [ \overrightarrow{u} ]^{(2+\alpha)}_{[0, \delta]} \\
	\end{aligned}
\end{equation*}
\begin{equation*} \label{E6}
		\begin{aligned}
		& + C \Bigl( |s-s'| |s'|^{\frac{2-|I|+\alpha}{2}} + |s'|^{\frac{4-|I|}{2}} |s'-s+s|^{\frac{\alpha}{2}} + |y-y'| |s'|^{\frac{2-|I|+\alpha}{2}} \Bigr) [ \overrightarrow{u} ]^{(2+\alpha)}_{[0, \delta]} \\
		& + C |s'-s+s|^{\frac{2-|I|+\alpha}{2}} \lVert f \rVert^{(\alpha)}_{[0, \delta]} + C \Bigl( s^{\frac{4-|I|}{2}} s^{\frac{\alpha}{2}} + s^{\frac{3-|I|}{2}} s^{\frac{\alpha}{2}} + s^{\frac{3-|I|-\alpha}{2}} s^{\frac{\alpha}{2}} \Bigr) [ \overrightarrow{u} ]^{(2+\alpha)}_{[0, \delta]} + Cs^{\frac{2-|I|+\alpha}{2}} \lVert f \rVert^{(\alpha)}_{[0, \delta]} \\
		\leq \;& C \Bigl( \delta^{\frac{3-|I|-\alpha}{2}} + \delta^{\frac{2-|I|+\alpha}{2}} \Bigr) \rho^\alpha [ \overrightarrow{u} ]^{(2+\alpha)}_{[0, \delta]} + C \delta^{\frac{2-|I|}{2}} \rho^\alpha \lVert f \rVert^{(\alpha)}_{[0, \delta]}
	\end{aligned}
\end{equation*}

\ \ 

\noindent\textbf{(Estimate for $\boldsymbol{E_7}$)} Next, we consider $\rho^2<s$. We examine the difference between $\partial_I u(s,s^\prime,x,y^\prime)|_{x=y}$ and $\partial_I u(s,s,x,y)|_{x=y}$ for $|I|=0$. By \eqref{Integralforoverleftarrowu}, we have   
\begin{equation*}
		\begin{aligned}& \bigl| \overleftarrow{u}(t,s',x,y') - \overleftarrow{u}(t,s,x,y) \bigr| \\
		\leq \;& \Biggl| \int_0^{s'} \int_{\mathbb{R}^d} Z(s',\tau,y',\xi;t,x) \sum_{|I|\leq 2} \overleftarrow{B}^I(t,\tau,x,\xi) \mathcal{I}^I(t,\tau,x,\xi) \, d\xi d\tau \\
		& \quad - \int_0^s \int_{\mathbb{R}^d} Z(s,\tau,y,\xi;t,x) \sum_{|I|\leq 2} \overleftarrow{B}^I(t,\tau,x,\xi) \mathcal{I}^I(t,\tau,x,\xi) \, d\xi d\tau \Biggr| \\
		& + \Biggl| \int_0^{s'} \int_{\mathbb{R}^d} Z(s',y') \overleftarrow{f}(t,\tau,x,\xi) \, d\xi d\tau - \int_0^s \int_{\mathbb{R}^d} Z(s,y) \overleftarrow{f}(t,\tau,x,\xi) \, d\xi d\tau \Biggr| \\
		=: \;& T_1 + T_2
	\end{aligned}
\end{equation*}
the first term $T_1$ of which is analyzed as follows:
\begin{equation*}
		\begin{aligned}|T_1| \leq \;& \Biggl| \int_0^{s'} \int_{\mathbb{R}^d} Z(s',\tau,y',\xi;t,x) \sum_{|I|\leq 2} \overleftarrow{B}^I(t,\tau,x,\xi) \mathcal{I}^I(t,\tau,x,\xi) \, d\xi d\tau \\
		& \quad - \int_0^{s} \int_{\mathbb{R}^d} Z(s',\tau,y',\xi;t,x) \sum_{|I|\leq 2} \overleftarrow{B}^I(t,\tau,x,\xi) \mathcal{I}^I(t,\tau,x,\xi) \, d\xi d\tau \Biggr| \\
		& + \Biggl| \int_0^{s} \int_{\mathbb{R}^d} Z(s',\tau,y',\xi;t,x) \sum_{|I|\leq 2} \overleftarrow{B}^I(t,\tau,x,\xi) \mathcal{I}^I(t,\tau,x,\xi) \, d\xi d\tau \\
		& \quad - \int_0^{s} \int_{\mathbb{R}^d} Z(s,\tau,y,\xi;t,x) \sum_{|I|\leq 2} \overleftarrow{B}^I(t,\tau,x,\xi) \mathcal{I}^I(t,\tau,x,\xi) \, d\xi d\tau \Biggr|
	\end{aligned}
\end{equation*}
Subsequently, we have
\begin{equation*}
		\begin{aligned}|T_1| \leq \;& C \int_s^{s'} \int_{\mathbb{R}^d} (s'-\tau)^{-\frac{d}{2}} \exp \bigl\{ -c\varpi(s',\tau,y',\xi) \bigr\} \sum_{|I|\leq 2} \bigl| \mathcal{I}^I(t,\tau,x,\xi) \bigr| \, d\xi d\tau \\
		& + \int_0^{s} \int_{\mathbb{R}^d} \bigl| Z(s,\tau,y,\xi;t,x) - Z(s,\tau,y',\xi;t,x) \bigr| \sum_{|I|\leq 2} \bigl| \mathcal{I}^I(t,\tau,x,\xi) \bigr| \, d\xi d\tau \\
		& + \int_0^{s} \int_{\mathbb{R}^d} \bigl| Z(s,\tau,y',\xi;t,x) - Z(s',\tau,y',\xi;t,x) \bigr| \sum_{|I|\leq 2} \bigl| \mathcal{I}^I(t,\tau,x,\xi) \bigr| \, d\xi d\tau \\
		\leq \;& C \int_s^{s'} \int_{\mathbb{R}^d} (s'-\tau)^{-\frac{d}{2}} \exp \bigl\{ -c\varpi(s',\tau,y',\xi) \bigr\} \sum_{|I|\leq 2} \bigl| \mathcal{I}^I(t,\tau,x,\xi) \bigr| \, d\xi d\tau \\
		& + C |y'-y|^\alpha \int_0^{s} \int_{\mathbb{R}^d} (s-\tau)^{-\frac{d+\alpha}{2}} \exp \bigl\{ -c\varpi(s,\tau,y,\xi) \bigr\} \sum_{|I|\leq 2} \bigl| \mathcal{I}^I(t,\tau,x,\xi) \bigr| \, d\xi d\tau \\
	\end{aligned}
\end{equation*}
\begin{equation*}
		\begin{aligned}
		& + C |s'-s|^{\frac{\alpha}{2}} \int_0^{s} \int_{\mathbb{R}^d} (s-\tau)^{-\frac{d+\alpha}{2}} \exp \bigl\{ -c\varpi(s,\tau,y',\xi) \bigr\} \sum_{|I|\leq 2} \bigl| \mathcal{I}^I(t,\tau,x,\xi) \bigr| \, d\xi d\tau
	\end{aligned}
\end{equation*}
Similarly, we can obtain the estimate for $T_2$. Then, when $(t,x)=(s,y)$, it follows that  
\begin{equation*}
		\begin{aligned}& \bigl| \overleftarrow{u}(s,s',y,y') - \overleftarrow{u}(s,s,y,y) \bigr| \\
		\leq \;& C \int_s^{s'} \int_{\mathbb{R}^d} (s'-\tau)^{-\frac{d}{2}} e^{-c\varpi(s',\tau,y',\xi)} \bigl( |s-\tau| + |y-\xi| \bigr) [ \overrightarrow{u} ]^{(2+\alpha)}_{[0, \delta]} \, d\xi d\tau \\
		& + C |y'-y|^\alpha \int_0^s \int_{\mathbb{R}^d} (s-\tau)^{-\frac{d+\alpha}{2}} e^{-c\varpi(s,\tau,y,\xi)} \bigl( |s-\tau| + |y-\xi| \bigr) [ \overrightarrow{u} ]^{(2+\alpha)}_{[0, \delta]} \, d\xi d\tau \\
		& + C |s'-s|^{\frac{\alpha}{2}} \int_0^s \int_{\mathbb{R}^d} (s-\tau)^{-\frac{d+\alpha}{2}} e^{-c\varpi(s,\tau,y',\xi)} \bigl( |s-\tau| + |y-\xi| \bigr) [ \overrightarrow{u} ]^{(2+\alpha)}_{[0, \delta]} \, d\xi d\tau \\
		& + C \int_s^{s'} \int_{\mathbb{R}^d} (s'-\tau)^{-\frac{d}{2}} e^{-c\varpi(s',\tau,y',\xi)} \lVert f \rVert^{(\alpha)}_{[0, \delta]} \, d\xi d\tau \\
		& + C |y'-y|^\alpha \int_0^s \int_{\mathbb{R}^d} (s-\tau)^{-\frac{d+\alpha}{2}} e^{-c\varpi(s,\tau,y,\xi)} \lVert f \rVert^{(\alpha)}_{[0, \delta]} \, d\xi d\tau \\
		& + C |s'-s|^{\frac{\alpha}{2}} \int_0^s \int_{\mathbb{R}^d} (s-\tau)^{-\frac{d+\alpha}{2}} e^{-c\varpi(s,\tau,y',\xi)} \lVert f \rVert^{(\alpha)}_{[0, \delta]} \, d\xi d\tau
	\end{aligned}
\end{equation*}
Consequently, we have
\begin{equation*}
		\begin{aligned}& \bigl| \overleftarrow{u}(s,s',y,y') - \overleftarrow{u}(s,s,y,y) \bigr| \\
		\leq \;& C \Bigl( |s'-s|^2 + |y'-y| |s'-s| + |s'-s|^{\frac{3}{2}} \Bigr) [ \overrightarrow{u} ]^{(2+\alpha)}_{[0, \delta]} \\
		& + C \Bigl( |y'-y| s^{\frac{4-\alpha}{2}} + |y'-y| s^{\frac{3-\alpha}{2}} \Bigr) [ \overrightarrow{u} ]^{(2+\alpha)}_{[0, \delta]} \\
		& + C \Bigl( |s'-s|^{\frac{\alpha}{2}} s^{\frac{4-\alpha}{2}} + |s'-s|^{\frac{\alpha}{2}} |y'-y| s^{\frac{2-\alpha}{2}} + |s'-s|^{\frac{\alpha}{2}} s^{\frac{3-\alpha}{2}} \Bigr) [ \overrightarrow{u} ]^{(2+\alpha)}_{[0, \delta]} \\
		& + C \rho^\alpha s^{\frac{2-\alpha}{2}} \lVert f \rVert^{(\alpha)}_{[0, \delta]} \leq \; C \rho^\alpha \Bigl( \delta^{\frac{4-\alpha}{2}} + \delta + \delta^{\frac{3-\alpha}{2}} \Bigr) [ \overrightarrow{u} ]^{(2+\alpha)}_{[0, \delta]} + C \rho^\alpha \delta^{\frac{2-\alpha}{2}} \lVert f \rVert^{(\alpha)}_{[0, \delta]}
	\end{aligned}
\end{equation*}

\ \ 

\noindent\textbf{(Estimate for $\boldsymbol{E_8}$)} Next, we consider the cases where $\rho^2<s$ and $|I|=1,2$. According to the representation \eqref{Integralforpartialoverleftarrowu} of $\partial_I\overleftarrow{u}(t,s,x,y)$, we have
\begin{equation*}
		\begin{aligned}& \triangle_{s,y}\partial_I\overleftarrow{u}(t,s,x,y) := \partial_I\overleftarrow{u}(t,s,x,y) - \partial_I\overleftarrow{u}(t,s^\prime,x,y^\prime) \\
		= \;& \int_0^{s-\lambda} \int_{\mathbb{R}^d} \triangle_{s,y}\partial_IZ(s,\tau,y,\xi;t,x) \sum_{|I|\leq 2} \bigl[ (\overleftarrow{B}^I\mathcal{I}^I)(t,\tau,x,\xi) - (\overleftarrow{B}^I\mathcal{I}^I)(t,\tau,x,y) \bigr] \, d\xi d\tau \\
		& + \int_{s-\lambda}^{s} \int_{\mathbb{R}^d} \partial_IZ \sum_{|I|\leq 2} \bigl[ (\overleftarrow{B}^I\mathcal{I}^I)(t,\tau,x,\xi) - (\overleftarrow{B}^I\mathcal{I}^I)(t,\tau,x,y) \bigr] \, d\xi d\tau \\
		& + \int_0^{s-\lambda} \left( \triangle_{s,y}\partial_I \int_{\mathbb{R}^d} Z \, d\xi \right) \sum_{|I|\leq 2} (\overleftarrow{B}^I\mathcal{I}^I)(t,\tau,x,y) \, d\tau \\
		& + \int_{s-\lambda}^s \partial_I \int_{\mathbb{R}^d} Z \, d\xi \sum_{|I|\leq 2} (\overleftarrow{B}^I\mathcal{I}^I)(t,\tau,x,y) \, d\tau \\
		& - \int_{s-\lambda}^{s^\prime} \int_{\mathbb{R}^d} \partial_{I,y^\prime}Z(s^\prime,\tau,y^\prime,\xi;t,x) \sum_{|I|\leq 2} \bigl[ (\overleftarrow{B}^I\mathcal{I}^I)(t,\tau,x,\xi) - (\overleftarrow{B}^I\mathcal{I}^I)(t,\tau,x,y^\prime) \bigr] \, d\xi d\tau \\
		& - \int_{s-\lambda}^{s^\prime} \partial_{I,y^\prime} \int_{\mathbb{R}^d} Z(s^\prime,\tau,y^\prime,\xi;t,x) \, d\xi \sum_{|I|\leq 2} (\overleftarrow{B}^I\mathcal{I}^I)(t,\tau,x,y^\prime) \, d\tau + J_7 =: \sum_{i=1}^{7} J_i
	\end{aligned}
\end{equation*}
where $\lambda=\frac{1}{2}\rho^2$ and the last term $J_7$ is equal to the sum of the first six terms ($J_1$-$J_6$) with all $\sum_{|I|\leq 2}\big(\overleftarrow{B}^I\mathcal{I}^I\big)$ by $\overleftarrow{f}$.

\ \ 

\noindent\textbf{($\boldsymbol{J_1}$-term)} We have 
\begin{equation*}
		\begin{aligned}|J_1| \leq \;& \int_0^{s-\lambda} \int_{\mathbb{R}^d} \Bigl( \bigl| \triangle_{y-y',y} \partial_I Z(s,\tau,y,\xi;t,x) \bigr| + \bigl| \triangle_{s-s',s} \partial_{I,y'} Z(s,\tau,y', \xi;t,x) \bigr| \Bigr) \\
		& \quad \times \sum_{|I|\leq 2} \bigl| (\overleftarrow{B}^I\mathcal{I}^I)(t,\tau,x,\xi) - (\overleftarrow{B}^I\mathcal{I}^I)(t,\tau,x,y) \bigr| \, d\xi d\tau \\
		\leq \;& \int_0^{s-\lambda} \int_{\mathbb{R}^d} \bigl| \triangle_{y-y',y} \partial_I Z(s,\tau,y,\xi;t,x) \bigr| \\
		& \quad \times \sum_{|I|\leq 2} \Bigl\{ \bigl| \overleftarrow{B}^I(t,\tau,x,\xi) - \overleftarrow{B}^I(t,\tau,x,y) \bigr| \bigl| \mathcal{I}^I(t,\tau,x,\xi) \bigr|  \\
		& \qquad\qquad + \bigl| \mathcal{I}^I(t,\tau,x,\xi) - \mathcal{I}^I(t,\tau,x,y) \bigr| \bigl| \overleftarrow{B}^I(t,\tau,x,y) \bigr| \Bigr\} \, d\xi d\tau \\
		& + \int_0^{s-\lambda} \int_{\mathbb{R}^d} \bigl| \triangle_{s-s',s} \partial_{I,y'} Z(s,\tau,y', \xi;t,x) \bigr| \\
		& \quad \times \sum_{|I|\leq 2} \Bigl\{ \bigl| (\overleftarrow{B}^I\mathcal{I}^I)(t,\tau,x,\xi) - (\overleftarrow{B}^I\mathcal{I}^I)(t,\tau,x,y') \bigr| \\
		& \qquad\qquad + \bigl| (\overleftarrow{B}^I\mathcal{I}^I)(t,\tau,x,y') - (\overleftarrow{B}^I\mathcal{I}^I)(t,\tau,x,y) \bigr| \Bigr\} \, d\xi d\tau
	\end{aligned}
\end{equation*}
Moreover, it holds that  
\begin{equation*}
		\begin{aligned}|J_1| \leq \;& \int_0^{s-\lambda} \int_{\mathbb{R}^d} \bigl| \triangle_{y-y',y} \partial_I Z(s,\tau,y,\xi;t,x) \bigr| \\
		& \quad \times \sum_{|I|\leq 2} \Bigl\{ \bigl| \overleftarrow{B}^I(t,\tau,x,\xi) - \overleftarrow{B}^I(t,\tau,x,y) \bigr| \bigl| \mathcal{I}^I(t,\tau,x,\xi) \bigr|  \\
		& \qquad\qquad + \bigl| \mathcal{I}^I(t,\tau,x,\xi) - \mathcal{I}^I(t,\tau,x,y) \bigr| \bigl| \overleftarrow{B}^I(t,\tau,x,y) \bigr| \Bigr\} \, d\xi d\tau \\
		& + \int_0^{s-\lambda} \int_{\mathbb{R}^d} \bigl| \triangle_{s-s',s} \partial_{I,y'} Z(s,\tau,y', \xi;t,x) \bigr| \\
		& \quad \times \sum_{|I|\leq 2} \Biggl[ \Bigl\{ \bigl| \overleftarrow{B}^I(t,\tau,x,\xi) - \overleftarrow{B}^I(t,\tau,x,y') \bigr| \bigl| \mathcal{I}^I(t,\tau,x,\xi) \bigr|  \\
		& \qquad\qquad\quad + \bigl| \mathcal{I}^I(t,\tau,x,\xi) - \mathcal{I}^I(t,\tau,x,y') \bigr| \bigl| \overleftarrow{B}^I(t,\tau,x,y') \bigr| \Bigr\} \\
		& \qquad\qquad + \Bigl\{ \bigl| \overleftarrow{B}^I(t,\tau,x,y) - \overleftarrow{B}^I(t,\tau,x,y') \bigr| \bigl| \mathcal{I}^I(t,\tau,x,y) \bigr| \\
		& \qquad\qquad\quad + \bigl| \mathcal{I}^I(t,\tau,x,y) - \mathcal{I}^I(t,\tau,x,y') \bigr| \bigl| \overleftarrow{B}^I(t,\tau,x,y') \bigr| \Bigr\} \Biggr] \, d\xi d\tau
	\end{aligned}
\end{equation*}
From the regularities of the fundamental solution $Z$ and $B^I$, we have  
\begin{equation*}
		\begin{aligned}|J_1| \leq \;& C |y'-y|^\alpha \int_0^{s-\lambda} \int_{\mathbb{R}^d} (s-\tau)^{-\frac{d+|I|+\alpha}{2}} e^{-c\varpi(s,\tau,y,\xi)} \\
		& \quad \times \sum_{|I|\leq 2} \Bigl\{ |y-\xi|^\alpha \bigl| \mathcal{I}^I(t,\tau,x,\xi) \bigr| + \bigl| \mathcal{I}^I(t,\tau,x,\xi) - \mathcal{I}^I(t,\tau,x,y) \bigr| \Bigr\} \, d\xi d\tau \\
		& + C |s'-s|^{\frac{\alpha}{2}} \int_0^{s-\lambda} \int_{\mathbb{R}^d} (s-\tau)^{-\frac{d+|I|+\alpha}{2}} e^{-c\varpi(s,\tau,y',\xi)} \\
		& \quad \times \sum_{|I|\leq 2} \Bigl\{ |y'-\xi|^\alpha \bigl| \mathcal{I}^I(t,\tau,x,\xi) \bigr| + \bigl| \mathcal{I}^I(t,\tau,x,\xi) - \mathcal{I}^I(t,\tau,x,y') \bigr| \\
		& \qquad\qquad + |y-y'|^\alpha \bigl| \mathcal{I}^I(t,\tau,x,y) \bigr| + \bigl| \mathcal{I}^I(t,\tau,x,y) - \mathcal{I}^I(t,\tau,x,y') \bigr| \Bigr\} \, d\xi d\tau
	\end{aligned}
\end{equation*}
Consequently, when $t=s$ and $x=y$, we have  
\begin{equation*}
		\begin{aligned}|J_1| \leq \;& C |y'-y|^\alpha \int_0^{s-\lambda} \int_{\mathbb{R}^d} (s-\tau)^{-\frac{d+|I|+\alpha}{2}} e^{-c\varpi(s,\tau,y,\xi)} \\
		& \quad \times \Bigl\{ |y-\xi|^\alpha \bigl( |s-\tau| + |y-\xi| \bigr) + (s-\tau) |y-\xi|^\alpha + |y-\xi| \Bigr\} [ \overrightarrow{u} ]^{(2+\alpha)}_{[0, \delta]} \, d\xi d\tau \\
		& + C |s'-s|^{\frac{\alpha}{2}} \int_0^{s-\lambda} \int_{\mathbb{R}^d} (s-\tau)^{-\frac{d+|I|+\alpha}{2}} e^{-c\varpi(s,\tau,y',\xi)} \\
		& \quad \times \Biggl\{ |y'-\xi|^\alpha \bigl( |s-\tau| + |y-y'| + |y'-\xi| \bigr) \\
		& \qquad\quad + \bigl( (s-\tau) |y'-\xi|^\alpha + |y'-\xi| + |y'-y| |y'-\xi| + |y'-y| |y'-\xi|^\alpha \bigr) \\
		& \qquad\quad + |y-y'|^\alpha (s-\tau) + (s-\tau) |y'-y|^\alpha + |y'-y| \Biggr\} [ \overrightarrow{u} ]^{(2+\alpha)}_{[0, \delta]} \, d\xi d\tau
	\end{aligned}
\end{equation*}
Then 
\begingroup\small
\begin{equation*}
		\begin{aligned}|J_1| \leq \;& C |y'-y|^\alpha \int_0^{s-\lambda} \Bigl( (s-\tau)^{-\frac{|I|-2}{2}} + (s-\tau)^{-\frac{|I|-1}{2}} \Bigr) [\overrightarrow{u}]^{(2+\alpha)}_{[0, \delta]} \, d\tau \\
		& + C |y'-y|^\alpha \int_0^{s-\lambda} \Bigl( (s-\tau)^{-\frac{|I|-2}{2}} + (s-\tau)^{-\frac{|I|+\alpha-1}{2}} \Bigr) [\overrightarrow{u}]^{(2+\alpha)}_{[0, \delta]} \, d\tau \\
		& + C |s'-s|^{\frac{\alpha}{2}} \int_0^{s-\lambda} \Bigl( (s-\tau)^{-\frac{|I|-2}{2}} + (s-\tau)^{-\frac{|I|-1}{2}} + (s-\tau)^{-\frac{|I|}{2}} |y'-y| \Bigr) [\overrightarrow{u}]^{(2+\alpha)}_{[0, \delta]} \, d\tau \\
		& + C |s'-s|^{\frac{\alpha}{2}} \int_0^{s-\lambda} \Bigl( (s-\tau)^{-\frac{|I|-2}{2}} + (s-\tau)^{-\frac{|I|+\alpha-1}{2}} (1 + |y'-y|) \\
		& \qquad\qquad\qquad + (s-\tau)^{-\frac{|I|}{2}} |y'-y| \Bigr) [\overrightarrow{u}]^{(2+\alpha)}_{[0, \delta]} \, d\tau \\
		& + C |s'-s|^{\frac{\alpha}{2}} \int_0^{s-\lambda} (s-\tau)^{-\frac{|I|+\alpha-2}{2}} |y'-y|^\alpha [\overrightarrow{u}]^{(2+\alpha)}_{[0, \delta]} \, d\tau \\
		& + C |s'-s|^{\frac{\alpha}{2}} \int_0^{s-\lambda} \Bigl( (s-\tau)^{-\frac{|I|+\alpha-2}{2}} |y'-y|^\alpha + (s-\tau)^{-\frac{|I|+\alpha}{2}} |y'-y| \Bigr) [\overrightarrow{u}]^{(2+\alpha)}_{[0, \delta]} \, d\tau
	\end{aligned}
\end{equation*}
\endgroup
It is noteworthy that $|y^\prime-y|\leq d=(2\lambda)^\frac{1}{2}\leq 2^\frac{1}{2}(s-\tau)^\frac{1}{2}$ if $\tau\in[0,s-\lambda]$. To simplify the results, we also extend the upper bound of integrals from $s-\lambda$ to $s$ since all integrands are positive in the interval. Consequently, we have 
\begin{equation*}
		\begin{aligned}|J_1| \leq \;& C \rho^\alpha \Bigl( s^{\frac{4-|I|}{2}} + s^{\frac{3-|I|}{2}} + s^{\frac{3-|I|-\alpha}{2}} + s^{\frac{4-|I|-\alpha}{2}} \Bigr) [ \overrightarrow{u} ]^{(2+\alpha)}_{[0, \delta]} \\
		\leq \;& C \rho^\alpha \Bigl( \delta^{\frac{4-|I|}{2}} + \delta^{\frac{3-|I|}{2}} + \delta^{\frac{3-|I|-\alpha}{2}} + \delta^{\frac{4-|I|-\alpha}{2}} \Bigr) [ \overrightarrow{u} ]^{(2+\alpha)}_{[0, \delta]}
	\end{aligned}
\end{equation*}

\ \ 

\noindent\textbf{($\boldsymbol{J_2}$-term)} Next, we investigate the second term $J_2$. 
\begin{equation*}
		\begin{aligned}|J_2| \leq \;& \int_{s-\lambda}^{s} \int_{\mathbb{R}^d} |\partial_I Z| \sum_{|I|\leq 2} \bigl| (\overleftarrow{B}^I\mathcal{I}^I)(t,\tau,x,\xi) - (\overleftarrow{B}^I\mathcal{I}^I)(t,\tau,x,y) \bigr| \, d\xi d\tau \\
		\leq \;& \int_{s-\lambda}^{s} \int_{\mathbb{R}^d} |\partial_I Z| \sum_{|I|\leq 2} \Bigl\{ \bigl| \overleftarrow{B}^I(t,\tau,x,\xi) - \overleftarrow{B}^I(t,\tau,x,y) \bigr| \bigl| \mathcal{I}^I(t,\tau,x,\xi) \bigr| \\
		& \qquad\qquad\qquad + \bigl| \mathcal{I}^I(t,\tau,x,\xi) - \mathcal{I}^I(t,\tau,x,y) \bigr| \bigl| \overleftarrow{B}^I(t,\tau,x,y) \bigr| \Bigr\} \, d\xi d\tau \\
		\leq \;& C \int_{s-\lambda}^{s} \int_{\mathbb{R}^d} (s-\tau)^{-\frac{d+|I|}{2}} e^{-c\varpi(s,\tau,y,\xi)} \\
		& \quad \times \sum_{|I|\leq 2} \Bigl\{ |y-\xi|^\alpha \bigl| \mathcal{I}^I(t,\tau,x,\xi) \bigr| + \bigl| \mathcal{I}^I(t,\tau,x,\xi) - \mathcal{I}^I(t,\tau,x,y) \bigr| \Bigr\} \, d\xi d\tau
	\end{aligned}
\end{equation*}
Moreover, in the case that $t=s$ and $x=y$, it follows that  
\begin{equation*}
		\begin{aligned}|J_2| \leq \;& C \int_{s-\lambda}^{s} \int_{\mathbb{R}^d} (s-\tau)^{-\frac{d+|I|}{2}} e^{-c\varpi(s,\tau,y,\xi)} \Bigl\{ |y-\xi|^\alpha \bigl( (s-\tau) + |y-\xi| \bigr) \\
		& \qquad\qquad\quad + \bigl( (s-\tau) |y-\xi|^\alpha + |y-\xi| \bigr) \Bigr\} [\overrightarrow{u}]^{(2+\alpha)}_{[0, \delta]} \, d\xi d\tau \\
		\leq \;& C \Bigl( \lambda^{\frac{4-|I|+\alpha}{2}} + \lambda^{\frac{3-|I|+\alpha}{2}} + \lambda^{\frac{3-|I|}{2}} \Bigr) [\overrightarrow{u}]^{(2+\alpha)}_{[0, \delta]} \\
		\leq \;& C \rho^\alpha \Bigl( \delta^{\frac{4-|I|}{2}} + \delta^{\frac{3-|I|}{2}} + \delta^{\frac{3-|I|-\alpha}{2}} \Bigr) [\overrightarrow{u}]^{(2+\alpha)}_{[0, \delta]}
	\end{aligned}
\end{equation*}

\ \ 

\noindent\textbf{($\boldsymbol{J_5}$-term)} Now, we study $J_5$. 
\begin{equation*}
		\begin{aligned}|J_5| \leq \;& \int_{s-\lambda}^{s'} \int_{\mathbb{R}^d} \bigl| \partial_{I,y'} Z(s',\tau,y',\xi;t,x) \bigr| \sum_{|I|\leq 2} \bigl| (\overleftarrow{B}^I\mathcal{I}^I)(t,\tau,x,\xi) - (\overleftarrow{B}^I\mathcal{I}^I)(t,\tau,x,y') \bigr| \, d\xi d\tau \\
		\leq \;& \int_{s-\lambda}^{s'} \int_{\mathbb{R}^d} \bigl| \partial_{I,y'} Z(s',\tau,y',\xi;t,x) \bigr| \sum_{|I|\leq 2} \Bigl\{ \bigl| \overleftarrow{B}^I(t,\tau,x,\xi) - \overleftarrow{B}^I(t,\tau,x,y') \bigr| \bigl| \mathcal{I}^I(t,\tau,x,\xi) \bigr| \\
		& \qquad\qquad\qquad\qquad + \bigl| \mathcal{I}^I(t,\tau,x,\xi) - \mathcal{I}^I(t,\tau,x,y') \bigr| \bigl| \overleftarrow{B}^I(t,\tau,x,y') \bigr| \Bigr\} \, d\xi d\tau \\
		\leq \;& C \int_{s-\lambda}^{s'} \int_{\mathbb{R}^d} (s'-\tau)^{-\frac{d+|I|}{2}} e^{-c\varpi(s',\tau,y',\xi)} \\
		& \quad \times \sum_{|I|\leq 2} \Bigl\{ |y'-\xi|^\alpha \bigl| \mathcal{I}^I(t,\tau,x,\xi) \bigr| + \bigl| \mathcal{I}^I(t,\tau,x,\xi) - \mathcal{I}^I(t,\tau,x,y') \bigr| \Bigr\} \, d\xi d\tau
	\end{aligned}
\end{equation*}
Hence, when $t=s$ and $x=y$, we have 
\begin{equation*}
		\begin{aligned}|J_5| \leq \;& C \int_{s'-\lambda}^{s'} \int_{\mathbb{R}^d} (s'-\tau)^{-\frac{d+|I|}{2}} e^{-c\varpi(s',\tau,y',\xi)} \Biggl\{ |y'-\xi|^\alpha \Bigl( |s-s'| + |s'-\tau| + |y-y'| + |y'-\xi| \Bigr) \\
		& \qquad\quad + |y'-\xi|^\alpha \bigl( |s-s'| + |s'-\tau| \bigr) + |y'-\xi| \\
		& \qquad\quad + |y-y'| \bigl( |y'-\xi| + |y'-\xi|^\alpha \bigr) \Biggr\} [\overrightarrow{u}]^{(2+\alpha)}_{[0, \delta]} \, d\xi d\tau \\
		\leq \;& C \rho^\alpha \Bigl( \delta^{\frac{2-|I|+\alpha}{2}} + \delta^{\frac{4-|I|}{2}} + \delta^{\frac{3-|I|}{2}} + \delta^{\frac{3-|I|-\alpha}{2}} \Bigr) [\overrightarrow{u}]^{(2+\alpha)}_{[0, \delta]}
	\end{aligned}
\end{equation*}

\ \ 

\noindent\textbf{($\boldsymbol{J_4}$-term and $\boldsymbol{J_6}$-term)} The estimates of $J_4$ and $J_6$ are similar, which are evaluated as follows:  
\begin{equation*}
		\begin{aligned}|J_4| \leq \;& C \int_{s-\lambda}^s \int_{\mathbb{R}^d} (s-\tau)^{-\frac{d+|I|}{2}} e^{-c\varpi(s,\tau,y,\xi)} (s-\tau) [\overrightarrow{u}]^{(2+\alpha)}_{[0,\delta]} \, d\xi d\tau \\
		\leq \;& C \rho^\alpha \delta^{\frac{4-|I|}{2}} [\overrightarrow{u}]^{(2+\alpha)}_{[0,\delta]}
	\end{aligned}
\end{equation*}
and 
\begin{equation*}
		\begin{aligned}|J_6| \leq \;& C \int_{s'-\lambda}^{s'} \int_{\mathbb{R}^d} (s'-\tau)^{-\frac{d+|I|}{2}} e^{-c\varpi(s',\tau,y',\xi)} \bigl( |s-s'| + |s'-\tau| + |y-y'| \bigr) [\overrightarrow{u}]^{(2+\alpha)}_{[0, \delta]} \, d\xi d\tau \\
		\leq \;& C \rho^\alpha \Bigl( \delta^{\frac{2-|I|+\alpha}{2}} + \delta^{\frac{4-|I|}{2}} \Bigr) [\overrightarrow{u}]^{(2+\alpha)}_{[0, \delta]}
	\end{aligned}
\end{equation*}

\ \ 

\noindent\textbf{($\boldsymbol{J_3}$-term)} Next, we analyze the $J_3$-term. 
\begin{equation*}
	|J_3| \leq \int_0^{s-\lambda} \left| \triangle_{s,y} \partial_I \int_{\mathbb{R}^d} Z \, d\xi \right| \sum_{|I|\leq 2} \bigl| \overleftarrow{B}^I\mathcal{I}^I(t,\tau,x,y) \bigr| \, d\tau
\end{equation*}
Then 
\begin{equation*}
		\begin{aligned}& \left| \triangle_{s,y} \partial_I \int_{\mathbb{R}^d} Z \, d\xi \right| \\
		\leq \;& \left| \triangle_{y'-y,y} \partial_I \int_{\mathbb{R}^d} Z(s,\tau,y,\xi;t,x) \, d\xi \right| + \left| \triangle_{s'-s,s} \partial_{I,y'} \int_{\mathbb{R}^d} Z(s,\tau,y',\xi;t,x) \, d\xi \right| \\
		\leq \;& \left| \triangle_{y'-y,y} \partial_I \int_{\mathbb{R}^d} Z_0(s-\tau, y-\xi, \tau, \xi; t, x) \, d\xi \right| + \left| \triangle_{y'-y,y} \partial_I \int_{\mathbb{R}^d} W(s, \tau, y, \xi; t, x) \, d\xi \right| \\
		& + \left| \triangle_{s'-s,s} \partial_{I,y'} \int_{\mathbb{R}^d} Z_0(s-\tau, y'-\xi, \tau, \xi; t, x) \, d\xi \right| + \left| \triangle_{s'-s,s} \partial_{I,y'} \int_{\mathbb{R}^d} W(s, \tau, y', \xi; t, x) \, d\xi \right| \\
		\leq \;& C \sum_{i=1}^d |y'_i - y_i| (s-\tau)^{-\frac{|k|+1-\alpha}{2}} + C |y'-y|^\alpha (s-\tau)^{-\frac{|k|}{2}} \\
		& + C (s'-s) (s-\tau)^{-\frac{|k|+2-\alpha}{2}} + C (s'-s)^{\frac{\alpha}{2}} (s-\tau)^{-\frac{|k|}{2}}
	\end{aligned}
\end{equation*}
Since $|y^\prime_i-y_i|\leq d\leq 2^\frac{1}{2}(s-\tau)^\frac{1}{2}$ and $s^\prime-s\leq d^2\leq 2(s-\tau)$ in the interval $\tau\in[0,s-\lambda]$. Consequently, when $t=s$ and $x=y$, we have  
\begin{equation*}
		\begin{aligned}|J_3| \leq \;& C \bigl( |y'-y|^\alpha + d^\alpha \bigr) \int_0^s (s-\tau)^{-\frac{|I|-2}{2}} [\overrightarrow{u}]^{(2+\alpha)}_{[0, \delta]} \, d\tau \\
		& + C |s'-s|^{\frac{\alpha}{2}} \int_0^s (s-\tau)^{-\frac{|I|-2}{2}} [\overrightarrow{u}]^{(2+\alpha)}_{[0, \delta]} \, d\tau \\
		\leq \;& C \rho^\alpha \delta^{\frac{4-|I|}{2}} [\overrightarrow{u}]^{(2+\alpha)}_{[0, \delta]}
	\end{aligned}
\end{equation*}

Finally, according to the classical theory of parabolic linear systems, we can find that $|J_7|\leq C\rho^\alpha\lVert f\rVert^{(\alpha)}_{[0,\delta]}$. From the analyses of $J_1$-$J_7$, it follows that in the case of $(t,x)=(s,y)$, $|\triangle_{s,y}\partial_I\overleftarrow{u}(t,s,x,y)|$ is bounded by $\lVert u\rVert^{(2+\alpha)}_{[0,\delta]}$ and $\lVert f\rVert^{(\alpha)}_{[0,\delta]}$ and that the coefficient in front of $\lVert u\rVert^{(2+\alpha)}_{[0,\delta]}$ could be arbitrarily small by choosing a suitable $\delta$.  

\ \ 

After showing the estimates of $E_1$-$E_8$ in Table \ref{TableforlinearPDE}, for any $(t,x)\in[0,\delta]\times\mathbb{R}^d$, there exists a small enough $\delta\in(0,T]$, independent of $(t,s,x,y)$, such that 
\begin{equation*}
		\begin{aligned}\Bigl\| \Bigl( u, \frac{\partial u}{\partial t}, \frac{\partial u}{\partial x}, \frac{\partial^2 u}{\partial x^2} \Bigr) (t,s,x,y) \Bigr\|^{(2+\alpha)}_{(s,y) \in [0,\delta] \times \mathbb{R}^d} 
		&\leq \frac{1}{2} [\overrightarrow{u}]^{(2+\alpha)}_{[0,\delta]} + C \Bigl( \| f \|^{(\alpha)}_{[0,\delta]} + \| g \|^{(2+\alpha)}_{[0,\delta]} \Bigr) \\
		&\leq \frac{1}{2} \| u \|^{(2+\alpha)}_{[0,\delta]} + C \Bigl( \| f \|^{(\alpha)}_{[0,\delta]} + \| g \|^{(2+\alpha)}_{[0,\delta]} \Bigr)
	\end{aligned}
\end{equation*}
Hence, we have 
\begin{equation*}
	\| u \|^{(2+\alpha)}_{[0, \delta]} \leq C \Bigl( \| f \|^{(\alpha)}_{[0, \delta]} + \| g \|^{(2+\alpha)}_{[0, \delta]} \Bigr)
\end{equation*}

\section{Estimate of \texorpdfstring{$\boldsymbol{\lVert\varphi\rVert^{(\alpha)}_{[0,\delta]}}$}{Estimate of varphi norm on [0,delta]}}
\label{App:B}
In this section, we will evaluate the nonhomogeneous term $\varphi(t,s,x,y)$ of \eqref{varphi}, and show it holds that  
\begin{equation*}
	\| \varphi \|^{(\alpha)}_{[0, \delta]} \leq C(R) \delta^{\frac{\alpha}{2}} \| u - \widehat{u} \|^{(2+\alpha)}_{[0, \delta]}
\end{equation*}
such that the mapping $\Lambda:u\mapsto \Lambda(u)$ defined by \eqref{DefinitionofLamda} is a $\frac{1}{2}$-contraction over the closed ball $\mathcal{U}$ centered at $g$ with radius $R$ for a small enough $\delta\in(0,T]$. In order to establish the inequality, we need to evaluate the terms of $K_1$-$K_{12}$ in Table \ref{tab:table5}.

\ \ 

\noindent \textbf{(Estimates $\boldsymbol{K_1}$-$\boldsymbol{K_3}$ of $\boldsymbol{|\varphi(t,s,x,y)|^{(\alpha)}_{(s,y)\in[0,\delta]\times\mathbb{R}^d}}$).} Let us consider $|\varphi(t,s,x,y)-\varphi(t,s^\prime,x,y)|$ for any $0\leq s< s^\prime\leq \delta\leq T$, $t\in[0,\delta]$, and $x,y\in\mathbb{R}^d$. By making use of \eqref{Integralrepresentationofvarphi}, it is convenient to add and subtract 
\begin{equation*}
		\begin{aligned}& \int_0^1 \sum_{|I|\leq 2} \partial_I F \bigl( t, s', x, y, \theta_\sigma(t, s', x, y) \bigr) \times \partial_I (u - \widehat{u})(t, s, x, y) \, d\sigma \\
		+ & \int_0^1 \sum_{|I|\leq 2} \partial_I \overline{F} \bigl( t, s', x, y, \theta_\sigma(t, s', x, y) \bigr) \times \partial_I (u - \widehat{u})(s, s, x, y) \big|_{x=y} \, d\sigma
	\end{aligned}
\end{equation*}
Subsequently, we need to estimate 
\begin{equation*}
		\begin{aligned}& \bigl| \partial_I F \bigl( t, s, x, y, \theta_\sigma(t, s, x, y) \bigr) - \partial_I F \bigl( t, s', x, y, \theta_\sigma(t, s', x, y) \bigr) \bigr|, \\
		& \bigl| \partial_I \overline{F} \bigl( t, s, x, y, \theta_\sigma(t, s, x, y) \bigr) - \partial_I \overline{F} \bigl( t, s', x, y, \theta_\sigma(t, s', x, y) \bigr) \bigr|, \\
		& \bigl| \partial_I F \bigl( t, s', x, y, \theta_\sigma(t, s', x, y) \bigr) - \partial_I F \bigl( t, 0, x, y, \theta_0(t, x, y) \bigr) \bigr|, \\
		\text{and } & \bigl| \partial_I \overline{F} \bigl( t, s', x, y, \theta_\sigma(t, s', x, y) \bigr) - \partial_I \overline{F} \bigl( t, 0, x, y, \theta_0(t, x, y) \bigr) \bigr|.
	\end{aligned}
\end{equation*}
Next, we have 
\begin{equation*}
		\begin{aligned}& \bigl| \partial_I F \bigl( t,s,x,y,\theta_\sigma(t,s,x,y) \bigr) - \partial_I F \bigl( t,s',x,y,\theta_\sigma(t,s',x,y) \bigr) \bigr| \\
		\leq \;& K(s'-s)^{\frac{\alpha}{2}} + L \biggl\{ \Bigl( \|u(t, \cdot, x, \cdot) \|^{(2+\alpha)}_{[0,\delta]} + \|\widehat{u}(t, \cdot, x, \cdot) \|^{(2+\alpha)}_{[0,\delta]} \Bigr) (s'-s)^{\frac{\alpha}{2}} \\
	\end{aligned}
\end{equation*}
\begin{equation*}
		\begin{aligned}
		& \quad + \sup_{\overline{s} \in (s,s')} |u_t(\overline{s}, \cdot, x, \cdot)|^{(2+\alpha)}_{[0,\overline{s}]} (s'-s) + \sup_{\overline{s} \in (s,s')} |\widehat{u}_t(\overline{s}, \cdot, x, \cdot)|^{(2+\alpha)}_{[0,\overline{s}]} (s'-s) \\
		& \quad + \Bigl( \|u(s', \cdot, x, \cdot) \|^{(2+\alpha)}_{[0,s']} + \|\widehat{u}(s', \cdot, x, \cdot) \|^{(2+\alpha)}_{[0,s']} \Bigr) (s'-s)^{\frac{\alpha}{2}} \biggr\} \\
		\leq \;& \Bigl( K + L \bigl( \| u \|^{(2+\alpha)}_{[0,\delta]} + \| \widehat{u} \|^{(2+\alpha)}_{[0,\delta]} \bigr) \Bigr) (s'-s)^{\frac{\alpha}{2}} \\
		\leq \;& C_1(R) (s'-s)^{\frac{\alpha}{2}}
	\end{aligned}
\end{equation*}
and
\begin{equation*}
		\begin{aligned}& \bigl| \partial_I F \bigl( t,s',x,y,\theta_\sigma(t,s',x,y) \bigr) - \partial_I F \bigl( t,0,x,y,\theta_0(t,x,y) \bigr) \bigr| \\
		\leq \;& K(s')^{\frac{\alpha}{2}} + L \biggl\{ \Bigl( \|(u-g)(t, \cdot, x, \cdot) \|^{(2+\alpha)}_{[0,\delta]} + \|(\widehat{u}-g)(t, \cdot, x, \cdot) \|^{(2+\alpha)}_{[0,\delta]} \Bigr) (s')^{\frac{\alpha}{2}} \\
		& \quad + \sup_{\overline{s} \in (0,s')} \|g_t(\overline{s}, x, \cdot) \|^{(2+\alpha)}_{\mathbb{R}^d} s' + \Bigl( \|u(s', \cdot, x, \cdot) \|^{(2+\alpha)}_{[0,s']} + \|\widehat{u}(s', \cdot, x, \cdot) \|^{(2+\alpha)}_{[0,s']} \Bigr) (s')^{\frac{\alpha}{2}} \biggr\} \\
		\leq \;& \Bigl( K + L \bigl( \| u-g \|^{(2+\alpha)}_{[0,\delta]} + \| \widehat{u}-g \|^{(2+\alpha)}_{[0,\delta]} + \| g \|^{(2+\alpha)}_{[0,\delta]} \bigr) \Bigr) (s')^{\frac{\alpha}{2}} \\
		\leq \;& C_2(R) \delta^{\frac{\alpha}{2}}
	\end{aligned}
\end{equation*}
where $L>0$ is a constant which can be different from line to line and the subscripts of $C$ are to represent different constant values within the derivation. In a similar manner, we can obtain 
\begin{equation*}
	\bigl| \partial_I \overline{F} \bigl( t, s, x, y, \theta_\sigma(t, s, x, y) \bigr) - \partial_I \overline{F} \bigl( t, s', x, y, \theta_\sigma(t, s', x, y) \bigr) \bigr| \leq C_3(R) (s' - s)^{\frac{\alpha}{2}},
\end{equation*}
and 
\begin{equation*}
	\bigl| \partial_I \overline{F} \bigl( t,s',x,y,\theta_\sigma(t,s',x,y) \bigr) - \partial_I \overline{F} \bigl( t,0,x,y,\theta_0(t,x,y) \bigr) \bigr| \leq C_4(R) \delta^{\frac{\alpha}{2}}.
\end{equation*}

\ \ 

\noindent \textbf{($\boldsymbol{K_2}$-term)} Consequently, for $K_2$, it holds that 
\begin{equation*}
		\begin{aligned}&\big| \varphi(t,s,x,y) - \varphi(t,s^\prime,x,y) \big| \\
		\leq \;& \int^1_0 \sum_{|I|\leq 2} \Big| \partial_I F\big(t,s,x,y,\theta_\sigma(t,s,x,y)\big) - \partial_I F\big(t,s^\prime,x,y,\theta_\sigma(t,s^\prime,x,y)\big) \Big| \\
		& \qquad \times \Big| \partial_I(u-\hat{u})(t,s,x,y) \Big| d\sigma \\
		& + \int^1_0 \sum_{|I|\leq 2} \Big| \partial_I \bar{F}\big(t,s,x,y,\theta_\sigma(t,s,x,y)\big) - \partial_I \bar{F}\big(t,s^\prime,x,y,\theta_\sigma(t,s^\prime,x,y)\big) \Big| \\
		& \qquad \times \Big| \partial_I(u-\hat{u})(s,s,x,y)\big|_{x=y} \Big| d\sigma \\
		& + \int^1_0 \sum_{|I|\leq 2} \Big| \partial_I F\big(t,s^\prime,x,y,\theta_\sigma(t,s^\prime,x,y)\big) - \partial_I F\big(t,0,x,y,\theta_0(t,x,y)\big) \Big| \\
		& \qquad \times \Big| \partial_I(u-\hat{u})(t,s,x,y) - \partial_I(u-\hat{u})(t,s^\prime,x,y) \Big| d\sigma \\
		& + \int^1_0 \sum_{|I|\leq 2} \Big| \partial_I \bar{F}\big(t,s^\prime,x,y,\theta_\sigma(t,s^\prime,x,y)\big) - \partial_I \bar{F}\big(t,0,x,y,\theta_0(t,x,y)\big) \Big| \\
		& \qquad \times \Big| \partial_I(u-\hat{u})(s,s,x,y)\big|_{x=y} - \partial_I(u-\hat{u})(s^\prime,s^\prime,x,y)\big|_{x=y} \Big| d\sigma
	\end{aligned}
\end{equation*}
Furthermore, we have 
\begin{equation} \label{Holdercontinuityofvarphiwithrespecttos}
		\begin{aligned}& \bigl| \varphi(t,s,x,y) - \varphi(t,s',x,y) \bigr| \\
		\leq \;& C_1(R) (s'-s)^{\frac{\alpha}{2}} \delta^{\frac{\alpha}{2}} \| u-\widehat{u} \|^{(2+\alpha)}_{[0,\delta]} + C_2(R) \delta^{\frac{\alpha}{2}} (s'-s)^{\frac{\alpha}{2}} \| u-\widehat{u} \|^{(2+\alpha)}_{[0,\delta]} \\
		& + C_3(R) (s'-s)^{\frac{\alpha}{2}} \delta^{\frac{\alpha}{2}} \| u-\widehat{u} \|^{(2+\alpha)}_{[0,s]} \\
		& + C_4(R) \delta^{\frac{\alpha}{2}} (s'-s)^{\frac{\alpha}{2}} \biggl( \sup_{\overline{s} \in (s,s')} |(u-\widehat{u})_t|^{(2+\alpha)}_{[0,\overline{s}]} + \| u-\widehat{u} \|^{(2+\alpha)}_{[0,s']} \biggr) \\
		\leq \;& C_5(R) \delta^{\frac{\alpha}{2}} (s'-s)^{\frac{\alpha}{2}} \| u-\widehat{u} \|^{(2+\alpha)}_{[0,\delta]}
	\end{aligned}
\end{equation}

\ \ 

\noindent \textbf{($\boldsymbol{K_1}$-term)} It directly implies $K_1$ by noting that $\varphi(t,0,x,y)\equiv 0$,
\begin{equation} \label{Boundnessofvarphi} 
	\| \varphi(t,\cdot,x,\cdot) \|^{(0)}_{[0,\delta] \times \mathbb{R}^d} \leq C_5(R) \delta \| u-\widehat{u} \|^{(2+\alpha)}_{[0,\delta]}.
\end{equation}

\ \ 

\noindent \textbf{($\boldsymbol{K_3}$-term)} Next, to estimate $|\varphi(t,s,x,y)-\varphi(t,s,x,y^\prime)|$ for any $y,y^\prime\in\mathbb{R}^d$ with $0<|y-y^\prime|\leq 1$, it is convenient to add and subtract 
\begin{equation*}
		\begin{aligned}& \int_0^1 \sum_{|I|\leq 2} \Bigl( \partial_I F \bigl( t,s,x,y', \theta_\sigma(t,s,x,y') \bigr) - \partial_I F \bigl( t,0,x,y', \theta_0(t,x,y') \bigr) \Bigr) \cdot \partial_I (u-\widehat{u})(t,s,x,y) \, d\sigma \\
		+ & \int_0^1 \sum_{|I|\leq 2} \Bigl( \partial_I \overline{F} \bigl( t,s,x,y', \theta_\sigma(t,s,x,y') \bigr) - \partial_I \overline{F} \bigl( t,0,x,y', \theta_0(t,x,y') \bigr) \Bigr) \cdot \partial_I (u-\widehat{u})(s,s,x,y) \big|_{x=y} \, d\sigma
	\end{aligned}
\end{equation*}
Note that
\begin{equation*}
		\begin{aligned}&\left| \partial_I F\big(t,s,x,y,\theta_\sigma(t,s,x,y)\big) - \partial_I F\big(t,s,x,y^\prime,\theta_\sigma(t,s,x,y^\prime)\big) \right| \\
		&\quad + \left| \partial_I F\big(t,0,x,y,\theta_0(t,x,y)\big) - \partial_I F\big(t,0,x,y^\prime,\theta_0(t,x,y^\prime)\big) \right| \\
		\leq \;& 2K|y-y^\prime|^\alpha + L \left| \theta_\sigma(t,s,x,y) - \theta_\sigma(t,s,x,y^\prime) \right| + L \left| \theta_0(t,x,y) - \theta_0(t,x,y^\prime) \right| \\
		\leq \;& 2K|y-y^\prime|^\alpha + L |y-y^\prime|^\alpha \Biggl( |u(t,\cdot,x,\cdot)|^{(2+\alpha)}_{[0,\delta]\times\mathbb{R}^d} + |u(s,\cdot,y^\prime, \cdot)|^{(2+\alpha)}_{[0,\delta]\times\mathbb{R}^d} + \sup_{\bar{y} \in (y,y^\prime)} |u_x(s,\cdot,\bar{y},\cdot)|^{(2+\alpha)}_{[0,\delta]\times\mathbb{R}^d} \\
		& + |\hat{u}(t,\cdot,x,\cdot)|^{(2+\alpha)}_{[0,\delta]\times\mathbb{R}^d} + |\hat{u}(s,\cdot,y^\prime,\cdot)|^{(2+\alpha)}_{[0,\delta]\times\mathbb{R}^d} + \sup_{\bar{y} \in (y,y^\prime)} |\hat{u}_x(s,\cdot,\bar{y},\cdot)|^{(2+\alpha)}_{[0,\delta]\times\mathbb{R}^d} \\
		& + |g(t,x,\cdot)|^{(2+\alpha)}_{\mathbb{R}^d} + |g(0,y^\prime,\cdot)|^{(2+\alpha)}_{\mathbb{R}^d} + \sup_{\bar{y} \in (y,y^\prime)} |g_x(0,\bar{y},\cdot)|^{(2+\alpha)}_{\mathbb{R}^d} \Biggr) \\
		\leq \;& \left( 2K + L \left( \| u \|^{(2+\alpha)}_{[0,\delta]} + \| \hat{u} \|^{(2+\alpha)}_{[0,\delta]} + \| g \|^{(2+\alpha)}_{[0,\delta]} \right) \right) |y-y^\prime|^\alpha \\
		\leq \;& C_6(R) |y-y^\prime|^\alpha.
	\end{aligned}
\end{equation*}
and  
\begin{equation*}
		\begin{aligned}& \bigl| \partial_I F \bigl( t,s,x,y,\theta_\sigma(t,s,x,y) \bigr) - \partial_I F \bigl( t,0,x,y,\theta_0(t,x,y) \bigr) \bigr| \\
		\leq \;& Ks^{\frac{\alpha}{2}} + L \biggl\{ \Bigl( \|u-g\|^{(2+\alpha)}_{[0,\delta]} + \|\widehat{u}-g\|^{(2+\alpha)}_{[0,\delta]} \Bigr) s^{\frac{\alpha}{2}} + \sup_{\overline{s} \in (0,s)} \|g_t\|^{(2+\alpha)}_{\mathbb{R}^d} s \\
		& \quad + \Bigl( \|u\|^{(2+\alpha)}_{[0,s]} + \|\widehat{u}\|^{(2+\alpha)}_{[0,s]} \Bigr) s^{\frac{\alpha}{2}} \biggr\} \\
		\leq \;& \Bigl( K + L \bigl( \| u-g \|^{(2+\alpha)}_{[0,\delta]} + \| \widehat{u}-g \|^{(2+\alpha)}_{[0,\delta]} + \| g \|^{(2+\alpha)}_{[0,\delta]} \bigr) \Bigr) s^{\frac{\alpha}{2}} \\
		\leq \;& C_7(R) \delta^{\frac{\alpha}{2}}
	\end{aligned}
\end{equation*}

Similarly, we also have 
\begin{equation*}
		\begin{aligned}& \bigl| \partial_I \overline{F} \bigl( t,s,x,y,\theta_\sigma(t,s,x,y) \bigr) - \partial_I \overline{F} \bigl( t,s,x,y', \theta_\sigma(t,s,x,y') \bigr) \bigr| \\
		+ & \bigl| \partial_I \overline{F} \bigl( t,0,x,y,\theta_0(t,x,y) \bigr) - \partial_I \overline{F} \bigl( t,0,x,y', \theta_0(t,x,y') \bigr) \bigr| \leq C_8(R) |y-y'|^\alpha.
	\end{aligned}
\end{equation*}
and 
\begin{equation*}
	\bigl| \partial_I \overline{F} \bigl( t,s,x,y,\theta_\sigma(t,s,x,y) \bigr) - \partial_I \overline{F} \bigl( t,0,x,y,\theta_0(t,x,y) \bigr) \bigr| \leq C_{9}(R) \delta^{\frac{\alpha}{2}}.
\end{equation*}

Hence, for $K_3$, we have 
\begin{equation} \label{Holdercontinuityofvarphiwithrespecttoy}
		\begin{aligned}& \bigl| \varphi(t,s,x,y) - \varphi(t,s,x,y') \bigr| \\
		\leq \;& C_6(R) |y-y'|^\alpha \delta^{\frac{\alpha}{2}} \| u-\widehat{u} \|^{(2+\alpha)}_{[0,\delta]} + C_8(R) |y-y'|^\alpha \delta^{\frac{\alpha}{2}} \| u-\widehat{u} \|^{(2+\alpha)}_{[0,\delta]} \\
		& + C_7(R) \delta^{\frac{\alpha}{2}} |y-y'|^\alpha \| u-\widehat{u} \|^{(2+\alpha)}_{[0,\delta]} \\
		& + C_9(R) \delta^{\frac{\alpha}{2}} |y-y'|^\alpha \biggl( \| u-\widehat{u} \|^{(2+\alpha)}_{[0,\delta]} + \sup_{\overline{y} \in (y,y')} |(u-\widehat{u})_x|^{(2+\alpha)}_{[0,\delta]} \biggr) \\
		\leq \;& C_{10}(R) \delta^{\frac{\alpha}{2}} |y-y'|^\alpha \| u-\widehat{u} \|^{(2+\alpha)}_{[0,\delta]}.
	\end{aligned}
\end{equation}

Consequently, from \eqref{Holdercontinuityofvarphiwithrespecttos}, \eqref{Boundnessofvarphi}, and \eqref{Holdercontinuityofvarphiwithrespecttoy}, for any $(t,x)\in[0,\delta]\times\mathbb{R}^d$, we obtain
\begin{equation} \label{Estimateofvarphi}
	\| \varphi(t,\cdot,x,\cdot) \|^{(\alpha)}_{[0,\delta] \times \mathbb{R}^d} \leq C_{11}(R) \delta^{\frac{\alpha}{2}} \| u-\widehat{u} \|^{(2+\alpha)}_{[0,\delta]}.
\end{equation}

\ \ 

\noindent \textbf{(Estimates $\boldsymbol{K_4}$-$\boldsymbol{K_6}$ of $\boldsymbol{|\varphi_t(t,s,x,y)|^{(\alpha)}_{(s,y)\in[0,\delta]\times\mathbb{R}^d}}$).} We now analyze the H\"{o}lder continuity of $\varphi_t(t,\cdot,x,\cdot)$ with respect to $s$ and $y$ in $[0,\delta]\times\mathbb{R}^d$. According to the integral representation \eqref{Integralrepresentationofvarphi} of $\varphi(t,s,x,y)$, its first derivative in $t$ satisfies 
\begin{equation*}
		\begin{aligned}&\varphi_t(t,s,x,y) \\
		= \;& \int^1_0 \sum_{|I|\leq 2} \Biggl[ \frac{\partial \Bigl( \partial_I F\big(t,s,x,y,\theta_\sigma(t,s,x,y)\big) - \partial_I F\big(t,0,x,y,\theta_0(t,x,y)\big) \Bigr)}{\partial t} \partial_I(u-\hat{u})(t,s,x,y) \\
		& \quad + \Bigl( \partial_I F\big(t,s,x,y,\theta_\sigma(t,s,x,y)\big) - \partial_I F\big(t,0,x,y,\theta_0(t,x,y)\big) \Bigr) \partial_I(u_t-\hat{u}_t)(t,s,x,y) \Biggr] d\sigma \\
		& + \int^1_0 \sum_{|I|\leq 2} \frac{\partial \Bigl( \partial_I \bar{F}\big(t,s,x,y,\theta_\sigma(t,s,x,y)\big) - \partial_I \bar{F}\big(t,0,x,y,\theta_0(t,x,y)\big) \Bigr)}{\partial t} \partial_I(u-\hat{u})(s,s,x,y)\big|_{x=y} d\sigma
	\end{aligned}
\end{equation*}

By the product rule and chain rule, we have 
\begin{equation} \label{varphi_t}
\begin{split}
    &\varphi_t(t,s,x,y) \\
    &= \int^1_0 \sum_{|I|\leq 2} \Bigl[ \partial^2_{tI} F\bigl(t,s,x,y,\theta_\sigma(t,s,x,y)\bigr) - \partial^2_{tI} F\bigl(t,0,x,y,\theta_0(t,x,y)\bigr) \Bigr] \\
    &\qquad\qquad\qquad \times \partial_I (u-\widehat{u})(t,s,x,y) \, d\sigma \\
    &\quad + \int^1_0 \sum_{|I|\leq 2} \Biggl[ \sum_{|J|\leq 2} \biggl( \partial^2_{IJ}F\bigl(t,s,x,y,\theta_\sigma(t,s,x,y)\bigr) \cdot \bigl( \sigma\partial_Ju_t(t,s,x,y) + (1-\sigma)\partial_J\widehat{u}_t(t,s,x,y) \bigr) \\
    &\qquad\qquad\qquad - \partial^2_{IJ}F\bigl(t,0,x,y,\theta_0(t,x,y)\bigr) \cdot \partial_Jg_t(t,x,y) \biggr) \Biggr] \times \partial_I (u-\widehat{u})(t,s,x,y) \, d\sigma \\
\end{split}
\end{equation}
\begin{equation*} 
\begin{split}
    &\quad + \int^1_0 \sum_{|I|\leq 2} \Bigl[ \partial_I F\bigl(t,s,x,y,\theta_\sigma(t,s,x,y)\bigr) - \partial_I F\bigl(t,0,x,y,\theta_0(t,x,y)\bigr) \Bigr] \\
    &\qquad\qquad\qquad \times \partial_I (u_t-\widehat{u}_t)(t,s,x,y) \, d\sigma \\
    &\quad + \int^1_0 \sum_{|I|\leq 2} \Bigl[ \partial^2_{tI} \overline{F}\bigl(t,s,x,y,\theta_\sigma(t,s,x,y)\bigr) - \partial^2_{tI} \overline{F}\bigl(t,0,x,y,\theta_0(t,x,y)\bigr) \Bigr] \\
    &\qquad\qquad\qquad \times \partial_I (u-\widehat{u})(s,s,x,y)\big|_{x=y} \, d\sigma \\
    &\quad + \int^1_0 \sum_{|I|\leq 2} \Biggl[ \sum_{|J|\leq 2} \biggl( \partial^2_{IJ}\overline{F}\bigl(t,s,x,y,\theta_\sigma(t,s,x,y)\bigr) \cdot \bigl( \sigma\partial_Ju_t(t,s,x,y) + (1-\sigma)\partial_J\widehat{u}_t(t,s,x,y) \bigr) \\
    &\qquad\qquad\qquad - \partial^2_{IJ}\overline{F}\bigl(t,0,x,y,\theta_0(t,x,y)\bigr) \cdot \partial_Jg_t(t,x,y) \biggr) \Biggr] \times \partial_I (u-\widehat{u})(s,s,x,y)\big|_{x=y} \, d\sigma \\
    &:= M_1 + M_2 + M_3 + M_4 + M_5
\end{split}
\end{equation*}
It is easy to see that the estimates of $M_1$, $M_3$, and $M_4$ are similar to the terms of $|\varphi(t,\cdot,x,\cdot)|^{(\alpha)}_{[0,\delta]\times\mathbb{R}^d}$. Hence, we focus on the remaining two terms $M_2$ and $M_5$. We denote $\eta(t,s,x,y)=M_2+M_5$. 

\ \ 

\noindent \textbf{(H\"{o}lder continuity of $\boldsymbol{\eta}$ in $\boldsymbol{s}$)} In order to estimate $|\eta(t,s,x,y)-\eta(t,s^\prime,x,y)|$ for $0\leq s< s^\prime\leq \delta\leq T$ and any $x,y\in\mathbb{R}^d$, it is convenient to add and subtract  
\begin{equation*}
    \begin{aligned}&\int^1_0 \sum_{|I|\leq 2} \Biggl[ \sum_{|J|\leq 2} \biggl( \partial^2_{IJ}F\bigl(t,s',x,y,\theta_\sigma(t,s',x,y)\bigr) \cdot \Bigl( \sigma\partial_J u_t(t,s',x,y) + (1-\sigma)\partial_J\widehat{u}_t(t,s',x,y) \Bigr) \biggr) \Biggr] \\
    &\qquad\qquad\qquad \times \partial_I (u-\widehat{u})(t,s,x,y) \, d\sigma \\
    &+ \int^1_0 \sum_{|I|\leq 2} \Biggl[ \sum_{|J|\leq 2} \biggl( \partial^2_{IJ}\overline{F}\bigl(t,s',x,y,\theta_\sigma(t,s',x,y)\bigr) \cdot \Bigl( \sigma\partial_J u_t(t,s',x,y) + (1-\sigma)\partial_J\widehat{u}_t(t,s',x,y) \Bigr) \biggr) \Biggr] \\
    &\qquad\qquad\qquad \times \partial_I (u-\widehat{u})(s,s,x,y)\big|_{x=y} \, d\sigma
\end{aligned}
\end{equation*}

Subsequently, we need to estimate 
\begin{equation} \label{eq:1stF-1stterm}
		\begin{aligned}\biggl| & \partial^2_{IJ}F \bigl( t,s,x,y,\theta_\sigma(t,s,x,y) \bigr) \cdot \bigl( \sigma \partial_J u_t + (1-\sigma) \partial_J \widehat{u}_t \bigr) (t,s,x,y) \\
		- & \partial^2_{IJ}F \bigl( t,s',x,y,\theta_\sigma(t,s',x,y) \bigr) \cdot \bigl( \sigma \partial_J u_t + (1-\sigma) \partial_J \widehat{u}_t \bigr) (t,s',x,y) \biggr|
	\end{aligned}
\end{equation}
\begin{equation} \label{eq:1stbarF-1stterm}
		\begin{aligned}\biggl| & \partial^2_{IJ}\overline{F} \bigl( t,s,x,y,\theta_\sigma(t,s,x,y) \bigr) \cdot \bigl( \sigma \partial_J u_t + (1-\sigma) \partial_J \widehat{u}_t \bigr) (t,s,x,y) \\
		- & \partial^2_{IJ}\overline{F} \bigl( t,s',x,y,\theta_\sigma(t,s',x,y) \bigr) \cdot \bigl( \sigma \partial_J u_t + (1-\sigma) \partial_J \widehat{u}_t \bigr) (t,s',x,y) \biggr|
	\end{aligned}
\end{equation}
and 
\begin{equation} \label{eq:1stF-2ndterm}
		\begin{aligned}\biggl| & \partial^2_{IJ}F \bigl( t,s',x,y,\theta_\sigma(t,s',x,y) \bigr) \cdot \bigl( \sigma \partial_J u_t + (1-\sigma) \partial_J \widehat{u}_t \bigr) (t,s',x,y) \\
		- & \partial^2_{IJ}F \bigl( t,0,x,y,\theta_0(t,x,y) \bigr) \cdot \partial_J g_t(t,x,y) \biggr|
	\end{aligned}
\end{equation}
\begin{equation} \label{eq:1stbarF-2ndterm}
		\begin{aligned}\biggl| & \partial^2_{IJ}\overline{F} \bigl( t,s',x,y,\theta_\sigma(t,s',x,y) \bigr) \cdot \bigl( \sigma \partial_J u_t + (1-\sigma) \partial_J \widehat{u}_t \bigr) (t,s',x,y) \\
		- & \partial^2_{IJ}\overline{F} \bigl( t,0,x,y,\theta_0(t,x,y) \bigr) \cdot \partial_J g_t(t,x,y) \biggr|.
	\end{aligned}
\end{equation}

Note that
\begin{equation*}
    \begin{aligned}\eqref{eq:1stF-1stterm} 
    &\leq \Biggl| \partial^2_{IJ}F\bigl(t,s,x,y,\theta_\sigma(t,s,x,y)\bigr) \cdot \Bigl( \sigma\partial_J u_t(t,s,x,y) + (1-\sigma)\partial_J\widehat{u}_t(t,s,x,y) \Bigr) \\
    &\qquad - \partial^2_{IJ}F\bigl(t,s',x,y,\theta_\sigma(t,s',x,y)\bigr) \cdot \Bigl( \sigma\partial_J u_t(t,s,x,y) + (1-\sigma)\partial_J\widehat{u}_t(t,s,x,y) \Bigr) \Biggr| \\
    &\quad + \Biggl| \partial^2_{IJ}F\bigl(t,s',x,y,\theta_\sigma(t,s',x,y)\bigr) \cdot \Bigl( \sigma\partial_J u_t(t,s,x,y) + (1-\sigma)\partial_J\widehat{u}_t(t,s,x,y) \Bigr) \\
    &\qquad - \partial^2_{IJ}F\bigl(t,s',x,y,\theta_\sigma(t,s',x,y)\bigr) \cdot \Bigl( \sigma\partial_J u_t(t,s',x,y) + (1-\sigma)\partial_J\widehat{u}_t(t,s',x,y) \Bigr) \Biggr| \\
    &\leq C_{12}(R)(s' - s)^{\frac{\alpha}{2}}
\end{aligned}
\end{equation*}

and 
\begin{equation*}
    \begin{aligned}\eqref{eq:1stF-2ndterm} 
    &\leq \Biggl| \partial^2_{IJ}F\bigl(t,s',x,y,\theta_\sigma(t,s',x,y)\bigr) \cdot \Bigl( \sigma\partial_J u_t(t,s',x,y) + (1-\sigma)\partial_J\widehat{u}_t(t,s',x,y) \Bigr) \\
    &\qquad - \partial^2_{IJ}F\bigl(t,s',x,y,\theta_\sigma(t,s',x,y)\bigr) \cdot \partial_J g_t(t,x,y) \Biggr| \\
    &\quad + \Biggl| \partial^2_{IJ} F\bigl(t,s',x,y,\theta_\sigma(t,s',x,y)\bigr) \cdot \partial_J g_t(t,x,y) \\
    &\qquad - \partial^2_{IJ}F\bigl(t,0,x,y,\theta_0(t,x,y)\bigr) \cdot \partial_J g_t(t,x,y) \Biggr| \\
    &\leq C_{13}(R)\delta^{\frac{\alpha}{2}}
\end{aligned}
\end{equation*}

Similarly, we also have 
\begin{equation*}
	\eqref{eq:1stbarF-1stterm} \leq C_{14}(R) (s' - s)^{\frac{\alpha}{2}}, \quad \eqref{eq:1stbarF-2ndterm} \leq C_{15}(R) \delta^{\frac{\alpha}{2}}.
\end{equation*}

Hence, we obtain that 
\begin{equation} \label{Holderofetains}
		\begin{aligned}& \bigl| \eta(t,s,x,y) - \eta(t,s',x,y) \bigr| \\
		\leq \;& C_{12}(R) (s' - s)^{\frac{\alpha}{2}} \delta^{\frac{\alpha}{2}} \| u-\widehat{u} \|^{(2+\alpha)}_{[0,\delta]} + C_{14}(R) (s' - s)^{\frac{\alpha}{2}} \delta^{\frac{\alpha}{2}} \| u-\widehat{u} \|^{(2+\alpha)}_{[0,\delta]} \\
		& + C_{13}(R) \delta^{\frac{\alpha}{2}} (s' - s)^{\frac{\alpha}{2}} \| u-\widehat{u} \|^{(2+\alpha)}_{[0,\delta]} \\
		& + C_{15}(R) \delta^{\frac{\alpha}{2}} (s' - s)^{\frac{\alpha}{2}} \biggl( \sup_{\overline{s} \in (s,s')} \| (u-\widehat{u})_t \|^{(2+\alpha)}_{[0,\delta]} + \| u-\widehat{u} \|^{(2+\alpha)}_{[0,\delta]} \biggr) \\
		\leq \;& C_{16}(R) \delta^{\frac{\alpha}{2}} (s' - s)^{\frac{\alpha}{2}} \| u-\widehat{u} \|^{(2+\alpha)}_{[0,\delta]}.
	\end{aligned}
\end{equation}

\ \ 

\noindent \textbf{(Boundedness of $\boldsymbol{\eta}$)} Then \eqref{Holderofetains} implies the following by noting that $\eta(t,0,x,y)\equiv 0$, 
\begin{equation} \label{Boundnessofeta}
	\| \eta(t,\cdot,x,\cdot) \|^{\infty}_{[0,\delta] \times \mathbb{R}^d} \leq C_{16}(R) \delta \| u-\widehat{u} \|^{(2+\alpha)}_{[0,\delta]}.
\end{equation}

\ \ 

\noindent \textbf{(H\"{o}lder continuity of $\boldsymbol{\eta}$ in $\boldsymbol{y}$)} In order to estimate $|\eta(t,s,x,y)-\eta(t,s,x,y^\prime)|$, it is convenient to add and subtract 
\begin{equation*}
    \begin{aligned}&\int^1_0 \sum_{|I|\leq 2} \Biggl[ \sum_{|J|\leq 2} \biggl( \partial^2_{IJ}F\bigl(t,s,x,y',\theta_\sigma(t,s,x,y')\bigr) \cdot \Bigl( \sigma\partial_J u_t(t,s,x,y') + (1-\sigma)\partial_J\widehat{u}_t(t,s,x,y') \Bigr) \\
    &\qquad\qquad\qquad - \partial^2_{IJ}F\bigl(t,0,x,y',\theta_0(t,x,y')\bigr) \cdot \partial_Jg_t(t,x,y') \biggr) \Biggr] \times \partial_I (u-\widehat{u})(t,s,x,y) \, d\sigma \\
    &+ \int^1_0 \sum_{|I|\leq 2} \Biggl[ \sum_{|J|\leq 2} \biggl( \partial^2_{IJ}\overline{F}\bigl(t,s,x,y',\theta_\sigma(t,s,x,y')\bigr) \cdot \Bigl( \sigma\partial_J u_t(t,s,x,y') + (1-\sigma)\partial_J\widehat{u}_t(t,s,x,y') \Bigr) \\
    &\qquad\qquad\qquad - \partial^2_{IJ}\overline{F}\bigl(t,0,x,y',\theta_0(t,x,y')\bigr) \cdot \partial_J g_t(t,x,y') \biggr) \Biggr] \times \partial_I (u-\widehat{u})(s,s,x,y)\big|_{x=y} \, d\sigma
\end{aligned}
\end{equation*}

Then we need to evaluate the estimates (for $F$)
\begin{equation} \label{eq:2ndF-1stterm}
		\begin{aligned}\Biggl| & \biggl[ \partial^2_{IJ}F\big(t,s,x,y,\theta_\sigma(t,s,x,y)\big) \cdot \Big( \sigma\partial_J u_t(t,s,x,y) + (1-\sigma)\partial_J\widehat{u}_t(t,s,x,y) \Big) \\
		& \quad - \partial^2_{IJ}F\big(t,0,x,y,\theta_0(t,x,y)\big) \cdot \partial_J g_t(t,x,y) \biggr] \\
		& - \biggl[ \partial^2_{IJ}F\big(t,s,x,y', \theta_\sigma(t,s,x,y')\big) \cdot \Big( \sigma\partial_J u_t(t,s,x,y') + (1-\sigma)\partial_J\widehat{u}_t(t,s,x,y') \Big) \\
		& \quad - \partial^2_{IJ}F\big(t,0,x,y', \theta_0(t,x,y')\big) \cdot \partial_J g_t(t,x,y') \biggr] \Biggr|
	\end{aligned}
\end{equation}
as well as the estimates (for $\overline{F}$)
\begin{equation} \label{eq:2ndbarF-1stterm}
		\begin{aligned}\Biggl| & \biggl[ \partial^2_{IJ}\overline{F}\big(t,s,x,y,\theta_\sigma(t,s,x,y)\big) \cdot \Big( \sigma\partial_J u_t(t,s,x,y) + (1-\sigma)\partial_J\widehat{u}_t(t,s,x,y) \Big) \\
		& \quad - \partial^2_{IJ}\overline{F}\big(t,0,x,y,\theta_0(t,x,y)\big) \cdot \partial_J g_t(t,x,y) \biggr] \\
		& - \biggl[ \partial^2_{IJ}\overline{F}\big(t,s,x,y', \theta_\sigma(t,s,x,y')\big) \cdot \Big( \sigma\partial_J u_t(t,s,x,y') + (1-\sigma)\partial_J\widehat{u}_t(t,s,x,y') \Big) \\
		& \quad - \partial^2_{IJ}\overline{F}\big(t,0,x,y', \theta_0(t,x,y')\big) \cdot \partial_J g_t(t,x,y') \biggr] \Biggr|.
	\end{aligned}
\end{equation}

Moreover, we also need to estimate 
\begin{equation} \label{eq:2ndF-2ndterm}
		\begin{aligned}\biggl| & \partial^2_{IJ}F\big(t,s,x,y,\theta_\sigma(t,s,x,y)\big) \cdot \Big( \sigma\partial_J u_t(t,s,x,y) + (1-\sigma)\partial_J\widehat{u}_t(t,s,x,y) \Big) \\
		- & \partial^2_{IJ}F\big(t,0,x,y,\theta_0(t,x,y)\big) \cdot \partial_J g_t(t,x,y) \biggr|
	\end{aligned}
\end{equation}
and 
\begin{equation} \label{eq:2ndbarF-2ndterm}
		\begin{aligned}\biggl| & \partial^2_{IJ}\overline{F}\big(t,s,x,y,\theta_\sigma(t,s,x,y)\big) \cdot \Big( \sigma\partial_J u_t(t,s,x,y) + (1-\sigma)\partial_J\widehat{u}_t(t,s,x,y) \Big) \\
		- & \partial^2_{IJ}\overline{F}\big(t,0,x,y,\theta_0(t,x,y)\big) \cdot \partial_J g_t(t,x,y) \biggr|.
	\end{aligned}
\end{equation}

Note that
\begin{equation*}
		\begin{aligned}\eqref{eq:2ndF-1stterm} \leq \;& \biggl| \partial^2_{IJ}F\big(t,s,x,y,\theta_\sigma(t,s,x,y)\big) \cdot \Big( \sigma\partial_J u_t(t,s,x,y) + (1-\sigma)\partial_J\widehat{u}_t(t,s,x,y) \Big) \\
		& - \partial^2_{IJ}F\big(t,s,x,y', \theta_\sigma(t,s,x,y')\big) \cdot \Big( \sigma\partial_J u_t(t,s,x,y') + (1-\sigma)\partial_J\widehat{u}_t(t,s,x,y') \Big) \biggr| \\
		& + \biggl| \partial^2_{IJ}F\big(t,0,x,y', \theta_0(t,x,y')\big) \cdot \partial_J g_t(t,x,y') - \partial^2_{IJ}F\big(t,0,x,y,\theta_0(t,x,y)\big) \cdot \partial_J g_t(t,x,y) \biggr| \\
		=: \;& N_1 + N_2.
	\end{aligned}
\end{equation*}

For $N_1$, it holds that 
\begin{equation*}
		\begin{aligned}N_1 \leq \;& \biggl| \partial^2_{IJ}F\big(t,s,x,y,\theta_\sigma(t,s,x,y)\big) \cdot \Bigl[ \big( \sigma\partial_J u_t(t,s,x,y) + (1-\sigma)\partial_J\widehat{u}_t(t,s,x,y) \big) \\
		& \qquad - \big( \sigma\partial_J u_t(t,s,x,y') + (1-\sigma)\partial_J\widehat{u}_t(t,s,x,y') \big) \Bigr] \biggr| \\
		& + \biggl| \Bigl[ \partial^2_{IJ}F\big(t,s,x,y,\theta_\sigma(t,s,x,y)\big) - \partial^2_{IJ}F\big(t,s,x,y', \theta_\sigma(t,s,x,y')\big) \Bigr] \\
		& \qquad \cdot \big( \sigma\partial_J u_t(t,s,x,y') + (1-\sigma)\partial_J\widehat{u}_t(t,s,x,y') \big) \biggr| \\
		\leq \;& C_{17}(R)|y-y'|^\alpha.
	\end{aligned}
\end{equation*}
For $N_2$, 
\begin{equation*}
		\begin{aligned}N_2 \leq \;& \biggl| \Bigl[ \partial^2_{IJ}F\big(t,0,x,y', \theta_0(t,x,y')\big) - \partial^2_{IJ}F\big(t,0,x,y, \theta_0(t,x,y)\big) \Bigr] \cdot \partial_J g_t(t,x,y') \biggr| \\
		& + \biggl| \partial^2_{IJ}F\big(t,0,x,y, \theta_0(t,x,y)\big) \cdot \Bigl[ \partial_J g_t(t,x,y') - \partial_J g_t(t,x,y) \Bigr] \biggr| \\
		\leq \;& C_{18}(R)|y-y'|^\alpha.
	\end{aligned}
\end{equation*}
From the estimates of $N_1$ and $N_2$, we have 
\begin{equation*}
	\eqref{eq:2ndF-1stterm} \leq  C_{19}(R)|y-y^\prime|^\alpha
\end{equation*}

Moroever, we have 
\begin{equation*}
		\begin{aligned}\eqref{eq:2ndF-2ndterm} \leq \;& \biggl| \partial^2_{IJ}F\big(t,s,x,y,\theta_\sigma(t,s,x,y)\big) \cdot \Bigl[ \big( \sigma\partial_J u_t(t,s,x,y) + (1-\sigma)\partial_J\widehat{u}_t(t,s,x,y) \big) - \partial_J g_t(t,x,y) \Bigr] \biggr| \\
		& + \biggl| \Bigl[ \partial^2_{IJ}F\big(t,s,x,y,\theta_\sigma(t,s,x,y)\big) - \partial^2_{IJ}F\big(t,0,x,y,\theta_0(t,x,y)\big) \Bigr] \cdot \partial_J g_t(t,x,y) \biggr| \\
		\leq \;& C_{20}(R) \delta^{\frac{\alpha}{2}}.
	\end{aligned}
\end{equation*}

Similarly, for $\overline{F}$, we have
\begin{equation*}
	\eqref{eq:2ndbarF-1stterm} \leq C_{21}(R)|y-y^\prime|^\alpha, \quad
	\eqref{eq:2ndbarF-2ndterm} \leq C_{22}(R)\delta^\frac{\alpha}{2}.
\end{equation*}

Hence, we have 
\begin{equation} \label{Holderofetainy}
		\begin{aligned}& \bigl| \eta(t,s,x,y) - \eta(t,s,x,y') \bigr| \\
		\leq \;& C_{19}(R) |y-y'|^\alpha \delta^{\frac{\alpha}{2}} \left| (u-\widehat{u})(t,\cdot,x,\cdot) \right|^{(2+\alpha)}_{[0,\delta] \times \mathbb{R}^d} \\
		& + C_{21}(R) |y-y'|^\alpha \delta^{\frac{\alpha}{2}} \left| (u-\widehat{u})(s,\cdot,y,\cdot) \right|^{(2+\alpha)}_{[0,\delta] \times \mathbb{R}^d} \\
		& + C_{20}(R) \delta^{\frac{\alpha}{2}} |y-y'|^\alpha \left| \partial_I (u-\widehat{u})(t,\cdot,x,\cdot) \right|^{(2+\alpha)}_{[0,\delta] \times \mathbb{R}^d} \\
		& + C_{22}(R) \delta^{\frac{\alpha}{2}} |y-y'|^\alpha \biggl( \left| (u-\widehat{u})(s,\cdot,y', \cdot) \right|^{(2+\alpha)}_{[0,\delta] \times \mathbb{R}^d} + \sup_{\overline{y} \in (y,y')} \left| (u-\widehat{u})_x(s,\cdot,\overline{y},\cdot) \right|^{(2+\alpha)}_{[0,\delta] \times \mathbb{R}^d} \biggr) \\
		\leq \;& C_{23}(R) \delta^{\frac{\alpha}{2}} |y-y'|^\alpha \lVert u-\widehat{u} \rVert^{(2+\alpha)}_{[0,\delta]}.
	\end{aligned}
\end{equation}

Therefore, together with \eqref{Holderofetains}, \eqref{Boundnessofeta}, and \eqref{Holderofetainy}, we have 
\begin{equation} \label{Holderofeta}
	|\eta(t,\cdot,x,\cdot)|^{(\alpha)}_{[0,\delta]\times\mathbb{R}^d}\leq C_{24}(R)\delta^{\frac{\alpha}{2}}\lVert u-\widehat{u}\rVert^{(2+\alpha)}_{[0,\delta]}.
\end{equation}

Since $M_1$, $M_3$ and $M_4$ of \eqref{varphi_t} satisfy the same estimates, the estimates of $K_4$, $K_5$ and $K_6$ hold as well. Hence, we have  
\begin{equation} \label{Estimateofvarphi_t}
	|\varphi_t(t,\cdot,x,\cdot)|^{(\alpha)}_{[0,\delta]\times\mathbb{R}^d}\leq C_{25}(R)\delta^{\frac{\alpha}{2}}\lVert u-\widehat{u}\rVert^{(2+\alpha)}_{[0,\delta]}.
\end{equation}

Thanks to the symmetry between $t$ and $x$, we also have the estimates of $K_7$, $K_8$ and $K_9$. Then, for any $(t,x)\in[0,\delta]\times\mathbb{R}^d$, it holds that   
\begin{equation} \label{Estimateofvarphi_x}
	|\varphi_x(t,\cdot,x,\cdot)|^{(\alpha)}_{[0,\delta]\times\mathbb{R}^d}\leq C_{25}(R)\delta^{\frac{\alpha}{2}}\lVert u-\widehat{u}\rVert^{(2+\alpha)}_{[0,\delta]}.
\end{equation}

Similarly, we can also acquire an integral representation for $\varphi_{xx}(t,s,x,y)$. Due to the chain rule and the product rule of derivatives, it is clear that the same estimate ($K_{10}$-$K_{12}$) holds for the term $|\varphi_{xx}(t,\cdot,x,\cdot)|^{(\alpha)}_{[0,\delta]\times\mathbb{R}^d}$, i.e.  
\begin{equation} \label{Estimateofvarphi_xx}
	|\varphi_{xx}(t,\cdot,x,\cdot)|^{(\alpha)}_{[0,\delta]\times\mathbb{R}^d}\leq C_{25}(R)\delta^{\frac{\alpha}{2}}\lVert u-\widehat{u}\rVert^{(2+\alpha)}_{[0,\delta]}.
\end{equation}

Consequently, together with \eqref{Estimateofvarphi}, \eqref{Estimateofvarphi_t}, \eqref{Estimateofvarphi_x}, and \eqref{Estimateofvarphi_xx}, we have 
\begin{equation}
	\lVert\varphi\rVert^{(\alpha)}_{[0,\delta]}\leq C(R)\delta^{\frac{\alpha}{2}}\lVert u-\widehat{u}\rVert^{(2+\alpha)}_{[0,\delta]},  
\end{equation}
Furthermore, it follows that 
\begin{equation}
	\lVert U-\widehat{U}\rVert^{(2+\alpha)}_{[0,\delta]}\leq C\lVert\varphi\rVert^{(\alpha)}_{[0,\delta]}\leq C(R)\delta^{\frac{\alpha}{2}}\lVert u-\widehat{u}\rVert^{(2+\alpha)}_{[0,\delta]}\leq \frac{1}{2}\lVert u-\widehat{u}\rVert^{(2+\alpha)}_{[0,\delta]}, 
\end{equation}
for a small enough $\delta\in(0,T]$.

\section{Probabilistic Representation for Nonlocal PDEs} \label{app:probrep}
In this appendix, we provide a probabilistic representation for the solutions of nonlocal fully nonlinear PDEs on top of their well-posedness. With such a representation, it is promising to combine the Monte Carlo simulations and deep learning techniques to devise a numerical scheme of solving the nonlocal PDEs (even in a high-dimensional setting); see \cite{Han2018}. Let us consider a nonlocal PDE \eqref{BackwardnonlocalfullynonlinearPDE} with a terminal condition:
\begin{equation} \label{BackwardnonlocalfullynonlinearPDE}
    \left\{
    \begin{aligned}
        & u_s(t,s,x,y) + F\Bigl(t,s,x,y,\bigl(\partial_I u\bigr)_{|I|\leq 2}(t,s,x,y),\, \bigl(\partial_I u\bigr)_{|I|\leq 2}(s,s,x,y)\big|_{x=y}\Bigr) = 0, \\
        & u(t,T,x,y) = g(t,x,y), \quad t,s \in [0,T], \quad x,y \in \mathbb{R}^d.
    \end{aligned}
    \right.
\end{equation}
Consequently, one can obtain the following conclusions. 

\begin{theorem} \label{thm:probrep}
	Suppose that $\sigma(s,y)\in C^{1,2}([0,T]\times\mathbb{R}^d)$ and \eqref{BackwardnonlocalfullynonlinearPDE} admits a unique solution $u(t,s,x,y)$ that is first-order continuously differentiable in $s$ and third-order continuously differentiable with respect to $y$ in $\nabla[0,T]\times\mathbb{R}^d$. Furthermore, let 
    \begin{equation*}
		\begin{aligned}Y(t,s) &:= u(t,s,X(t),X(s)), & Z(t,s) &:= \bigl(\sigma^\top u_y\bigr)(t,s,X(t),X(s)), \\
		\Gamma(t,s) &:= \bigl(\sigma^\top(\sigma^\top u_y)_y\bigr)(t,s,X(t),X(s)), & A(t,s) &:= \mathcal{D}\bigl(\sigma^\top u_y\bigr)(t,s,X(t),X(s)).
	\end{aligned}
\end{equation*}
	where $\left(\sigma^\top u_y\right)(t,s,x,y)=\sigma^\top(s,y)u_y(t,s,x,y)$ and the operator $\mathcal{D}$ is defined by 
    \begin{equation*}
    \mathcal{D}\varphi = \varphi_s + \frac{1}{2} \sum_{i,j=1}^d (\sigma\sigma^\top)_{ij} \frac{\partial^2\varphi}{\partial y_i \partial y_j} + \sum_{i=1}^d b_i \frac{\partial \varphi}{\partial y_i}.
\end{equation*}
	then the family of random fields $\left(X(\cdot),Y(\cdot,\cdot),Z(\cdot,\cdot),\Gamma(\cdot,\cdot),A(\cdot,\cdot)\right)$ is an adapted solution of the following flow of 2FBSDEs: 
    \begin{equation} \label{Flowof2FBSDEs}
	\begin{aligned}
		X(s) = \;& X(t) + \int^s_{t} b(\tau,X(\tau)) d\tau + \int^s_{t} \sigma(\tau,X(\tau)) dW(\tau), \\
		Y(t,s) = \;& g(t,X(t),X(T)) - \int^T_s Z^\top(t,\tau) dW(\tau) \\
		& + \int^T_s \mathscr{F} \Bigl( t,\tau,X(t),X(\tau),Y(t,\tau),Y(\tau,\tau),Z(t,\tau),Z(\tau,\tau),\Gamma(t,\tau),\Gamma(\tau,\tau) \Bigr) d\tau, \\
		Z(t,s) = \;& Z(t,t) + \int^s_{t} A(t,\tau) d\tau + \int^s_{t} \Gamma(t,\tau) dW(\tau), \quad 0 \leq t \leq s \leq T.
	\end{aligned}
\end{equation}
	where $\mathscr{F}$ is defined by 
    \begin{equation*}
		\begin{aligned}& \mathscr{F}\bigl(t,\tau,X(t),X(\tau),Y(t,\tau),Y(\tau,\tau),Z(t,\tau),Z(\tau,\tau),\Gamma(t,\tau),\Gamma(\tau,\tau)\bigr) \\
		= \;& \overline{F}\Bigl(t,\tau,X(t),X(\tau),\bigl(\partial_I u\bigr)_{|I|\leq 2}(t,\tau,X(t),X(\tau)), \bigl(\partial_I u\bigr)_{|I|\leq 2}(\tau,\tau,X(\tau),X(\tau))\Bigr)
	\end{aligned}
\end{equation*}
	with the definition of $\overline{F}$
    \begin{equation*}
		\begin{aligned}& \overline{F}\Bigl(t,\tau,x,y,\bigl(\partial_I u\bigr)_{|I|\leq 2}(t,\tau,x,y),\, \bigl(\partial_I u\bigr)_{|I|\leq 2}(\tau,\tau,x,y)\big|_{x=y}\Bigr) \\
		:= \;& F\Bigl(t,\tau,x,y,\bigl(\partial_I u\bigr)_{|I|\leq 2}(t,\tau,x,y),\, \bigl(\partial_I u\bigr)_{|I|\leq 2}(\tau,\tau,x,y)\big|_{x=y}\Bigr) \\
		& - \frac{1}{2}\sum_{i,j=1}^d (\sigma\sigma^\top)_{ij}(\tau,y) \frac{\partial^2 u}{\partial y_i \partial y_j}(t,\tau,x,y) - \sum_{i=1}^d b_i(\tau,y) \frac{\partial u}{\partial y_i}(t,\tau,x,y).
	\end{aligned}
\end{equation*}
\end{theorem}

The results come directly from the application of It\^{o}'s lemma. We refer the readers to the similar claims and proofs in \cite{Wang2019,Lei2023,Lei2021}. We make three important observations about the stochastic system \eqref{Flowof2FBSDEs}: (I) when the generator $\mathscr{F}$ is independent of diagonal terms, i.e. $Y(\tau,\tau)$, $Z(\tau,\tau)$, and $\Gamma(\tau,\tau)$, the flow of FBSDEs \eqref{Flowof2FBSDEs} are reduced to a family of 2FBSDEs parameterized by $(t,X(t))$, which is exactly the 2FBSDE in \cite{Kong2015} and equivalent to the ones in \cite{Cheridito2007} for any fixed $t$; (II) \eqref{Flowof2FBSDEs} is more general than related results in the previous literature \cite{Peng1992,Ma2002,Wang2019,Wang2020,Hamaguchi2020,Lei2023} since it allows for the nonlinearity of $(Y(t,\tau),Z(t,\tau),\Gamma(t,\tau))$ by introducing an additional SDE and also contains their diagonal terms $(Y(\tau,\tau),Z(\tau,\tau),\Gamma(\tau,\tau))$ in almost arbitrary way; (III) inspired by \cite{Cheridito2007} and \cite{Soner2011}, it is interesting to establish the well-posedness of \eqref{Flowof2FBSDEs} in the theoretical framework of SDEs. However, it is beyond the scope of this paper while we put it into our research agenda.

\section{Proof of Theorem 4.2} \label{App:ProofThm4.2}

	We refer the readers to \cite{Lei2021} for the similar claim and proof. Assume that $\lim_{s\to\tau}u(t,s,x,y)=u(t,\tau,x,y)\in\mathcal{O}$. To obtain a global solution, the maximally defined solution in $[0,\tau(g))$ has to be extended into the bigger interval $[0,\tau(g)]$ such that we can update the initial data with $u(\cdot,\tau,\cdot,\cdot)\in\Omega^{(2+\alpha)}_{[0,T]}$. It requires that the mapping $u:s\mapsto u(t,s,x,y)$ from $[0,\tau)$ to $\Omega^{(2+\alpha)}_{[0,T]}$ is at least uniformly continuous. By the estimate \eqref{Upperboundedforglobalexistence}, we have
	\begin{equation*}
		u\in B\big([0,\tau);\Omega^{(2+\alpha+\epsilon)}_{[0,T]}\big), \quad u_s\in B\big([0,\tau);\Omega^{(\alpha+\epsilon)}_{[0,T]}\big),
	\end{equation*}
	where $B([a,b);X)$ denotes the space of bounded functions over $[a,b)$ valued in the Banach space $X$. By an interpolation result for $\theta\in[0,1]$ (see \cite[Proposition 2.7]{Sinestrari1985}), it follows that $u\in C^{1-\theta}\big([0,\sigma];\Omega^{(\alpha+\epsilon+2\theta)}_{[0,T]}\big)$ for every $\sigma\in(0,\tau)$ with a H\"{o}lder constant independent of $\sigma$. By choosing $\theta=1-\frac{\epsilon}{2}$, we have $u\in C^{\frac{\epsilon}{2}}\big([0,\sigma];\Omega^{(2+\alpha)}_{[0,T]}\big)$. Consequently, $u$ can be continued at $s=\tau(g)$ in such a way that the extension belongs to $u(\cdot,\tau,\cdot,\cdot)\in\Omega^{(2+\alpha)}_{[0,T]}$. Then, by Theorem \ref{Localwell-posednessoffullynonlinearPDE}, \eqref{NonlocalfullynonlinearPDE} admits a unique solution $u\in\Omega^{(2+\alpha)}_{[0,\tau+\tau_1]}$ for some $\tau_1>0$, which contradicts the definition of $\tau(g)$. Therefore, we have $\tau(g)=T$.

\section{Proof of Theorem 4.3} \label{App:ProofThm4.3}
	We leverage Theorem \ref{Localwell-posednessoffullynonlinearPDE} for the analysis of \eqref{QuasilinearPDE}. It is clear that \eqref{QuasilinearPDE} admits a unique solution in $\Omega^{(2+\alpha)}_{[0,T]}$ in $[0,\tau]^2\times\mathbb{R}^{d;d}$. If $\tau=T$, the proof is completed. Otherwise, we ought to examine if the maximally defined solution can be extended uniquely into $[0,T]^2$. 

	Following the proof of Theorem \ref{Localwell-posednessoffullynonlinearPDE} and the definition of $\Lambda$, we can find that the argument $R$ of $C(R)\delta^\frac{\alpha}{2}$ in \eqref{Contraction} depends on $\lVert \cdot\rVert^{(1+\alpha)}_{[0,\delta]}$-norm of $u$ and $g$ rather than their $\lVert \cdot\rVert^{(2+\alpha)}_{[0,\delta]}$-norm, which is a key difference between \eqref{NonlocalfullynonlinearPDE} and \eqref{QuasilinearPDE}.
	In the case of \eqref{QuasilinearPDE}, we only need to control the behavior of solutions in the $\lVert \cdot\rVert^{(1+\alpha)}_{[0,\delta]}$-topology and show that the mapping $u:s\mapsto u(t,s,x,y)$ from $[0,\tau)$ to $\Omega^{(1+\alpha)}_{[0,T]}$ is uniformly continuous.    

	By \eqref{QuasilinearPDE} restricted in $[0,\tau)^2\times\mathbb{R}^{d;d}$, we can differentiate the equation once and twice with respect to $x_i$, $i=0,1,\cdots,d$, then 
    \begin{equation} \label{QuasilinearPDE1}
    \left\{
    \begin{aligned}
        \left(\frac{\partial u}{\partial x_i}\right)_s(t,s,x,y) &= \sum_{|I|= 2} A^I(s,y) \partial_I \left(\frac{\partial u}{\partial x_i}\right)(t,s,x,y) \\
        &\quad + \sum_{|I|\leq 1} Q_{p^I}(u) \partial_I\left(\frac{\partial u}{\partial x_i}\right)(t,s,x,y) + Q_{x_i}(u), \\
        \left(\frac{\partial u}{\partial x_i}\right)(t,0,x,y) &= g_{x_i}(t,x,y), \quad t,s\in[0,\tau), \quad x,y\in\mathbb{R}^d, \quad i=0,\dots,d
    \end{aligned}
    \right. 
\end{equation}
	and 
    \begin{equation} \label{QuasilinearPDE2}
    \left\{
    \begin{aligned}
        \left(\frac{\partial^2 u}{\partial x_i\partial x_j}\right)_s(t,s,x,y) &= \sum_{|I|= 2} A^I(s,y) \partial_I \left(\frac{\partial^2 u}{\partial x_i\partial x_j}\right)(t,s,x,y) \\
        &\quad + \sum_{|I|\leq 1} Q_{p^I}(u) \partial_I\left(\frac{\partial^2 u}{\partial x_i\partial x_j}\right)(t,s,x,y) \\
        &\quad + \sum_{|I|\leq 1,|J|\leq 1} Q_{p^Ip^J}(u) \left(\partial_I\left(\frac{\partial u}{\partial x_i}\right) \partial_J\left(\frac{\partial u}{\partial x_j}\right)\right)(t,s,x,y) \\
        &\quad + \sum_{|I|\leq 1} Q_{p^Ix_j}(u) \partial_I\left(\frac{\partial u}{\partial x_i}\right)(t,s,x,y) \\
        &\quad + \sum_{|I|\leq 1} Q_{x_ip^I}(u) \partial_I\left(\frac{\partial u}{\partial x_j}\right)(t,s,x,y) + Q_{x_ix_j}(u), \\
        \left(\frac{\partial^2 u}{\partial x_i\partial x_j}\right)(t,0,x,y) &= g_{x_ix_j}(t,x,y), \quad t,s\in[0,\tau), \quad i,j=1,\dots,d. 
    \end{aligned}
    \right. 
\end{equation}
	With the conditions \eqref{Growthcondition}-\eqref{Boundednesscondition} and the Gr\"{o}nwall--Bellman inequality, it is clear from \eqref{QuasilinearPDE}, \eqref{QuasilinearPDE1}, and \eqref{QuasilinearPDE2} that $u\in\Omega^{(1+\alpha)}_{[0,\tau)}$ and there exists a constant $K^\prime$ such that $\lVert u\rVert^{(1+\alpha)}_{[0,\tau)}\leq K^\prime$. Consequently, the nonlinearity $Q$ of \eqref{QuasilinearPDE} belongs to $\Omega^{(\alpha)}_{[0,\tau)}$ as well. Furthermore, Theorem \ref{Well-posednessofL} implies that \eqref{QuasilinearPDE} admits a unique solution $u\in\Omega^{(2+\alpha)}_{[0,\tau)}$ in $[0,\tau)^2\times\mathbb{R}^{d;d}$ and $\lVert u\rVert^{(2+\alpha)}_{[0,\tau)}\leq K^\prime$, where $K^\prime$ could vary from line to line. Therefore, by Lemma 8.5.5 in \cite{Lunardi1995}, the mapping $u:s\mapsto u(t,s,x,y)$ from $[0,\tau)$ to $\Omega^{(1+\alpha)}_{[0,T]}$ has an analytic continuation at $s=\tau$ and $u(\cdot,\tau,\cdot,\cdot)\in\Omega^{(2+\alpha)}_{[0,T]}$. With the same spirit of the proof of Theorem \ref{Globalsolvability}, we can extend the maximally defined solution until $[0,T]^2$ by updating the initial condition.

\section{Sophisticated Player versus Na\"{i}ve Player} \label{app:diffvalue}
While the sophisticated controllers would think globally and act locally, it would be interesting to investigate for the ``na\"{i}ve" controllers, who think and act locally. Specifically, at each time $t$ and state $x$, the na\"{i}ve controllers fix their originally time-varying objectives as the one at $(t,x)$ throughout the stochastic control problem over $[t,T]\times \mathbb{R}$. Given $(t,x)$, the na\"{i}ve controllers only need to solve a time-consistent stochastic control problem, whose value function is denoted by $u^n(t,s,x,y)$ for $(s,y)\in [t,T]\times \mathbb{R}$. By the dynamic programming approach and the assumption of existing optimum of Hamiltonian \eqref{Hamiltonian}, it is not difficult to find that the PDE for $u^n$ is given by 
\begin{equation} \label{DifferenceEq1}
	\left\{
	\begin{aligned}
		u^n_s(t,s,x,y) = \;& \overline{\mathcal{H}}\Bigl( t,s,x,y, \psi\big(t,s,x,y,(\partial_I u^n)_{|I|\leq 2}(t,s,x,y)\big), \\
		& \qquad\quad (\partial_I u^n)_{|I|\leq 2}(t,s,x,y) \Bigr), \quad (t,s,x,y) \in \triangle[0,T] \times \mathbb{R}^{2d}, \\
		u^n(t,0,x,y) = \;& \overline{g}(t,x,y), \quad (t,x,y) \in [0,T] \times \mathbb{R}^{2d}.
	\end{aligned}
	\right.
\end{equation}
Its key difference from \eqref{EquilibriumHJBequation} is that there is no $u$-function terms with substitution of $(t,x)$ by $(s,y)$ in \eqref{DifferenceEq1}. Thus, the PDE \eqref{DifferenceEq1} is local as $(t,x)$ can be viewed as fixed parameters. Subsequently, the function, defined by $V^n(s,y)=u^n(s,s,y,y)$, naturally serves as a lower bound for the corresponding equilibrium value function $V$ since $V^n(s,y)$ is the local solution (in value) of problem \eqref{TICproblem}. However, we should note that there is no consistent control policy (in the face of TIC) that can give the value function $V^n(s,y)$. In other words, $u^n$ and $V^n$ are merely nominal while not achievable.

Our next proposition estimates the difference between the functions $V(s,y)$ and $V^n(s,y)$, which somehow implies the goodness of using the equilibrium strategy. The result builds on top of the well-posedness of the nonlocal fully nonlinear PDEs we obtained in the previous sections. To avoid repeated discussions on the global solvability, we leverage only the local well-posedness results and discuss only on a small time interval. Though the result is extendable, it may not bring additional insights.
\begin{proposition}
	Suppose that $\mathcal{H}$ and $\psi$ possess needed
	regularity.
	Then there exists a constant $C(R)>0$ such that 
	$$
	V^n(s,y)\leq V(s,y)\leq V^n(s,y)+C(R)((T-s)^2+(T-s)^\frac{3}{2})
	$$
	for any $s\in[T-\delta,T]$ and $y\in\mathbb{R}^d$, where $R$ and $\delta$ are determined by Theorem \ref{Localwell-posednessoffullynonlinearPDE}.  
\end{proposition}
\begin{proof}
	In order to evaluate the difference between $V^n$ and $V$, we study the following two nonlocal PDEs \eqref{DifferenceEq1} and \eqref{EquilibriumHJBequation} for $(t,s,x,y)\in\triangle[0,\delta]\times\mathbb{R}^{d;d}$.
	%
	Consequently, given the well-posedness of $u(t,s,x,y)$, we obtain a classical PDE for $(u^n-u)(t,s,x,y)$ of the following form
    \begin{equation} 
	\left\{
	\begin{aligned}
		(u^n-u)_s(t,s,x,y) = \;& \sum_{|I|\leq 2} M^I(t,s,x,y) \partial_I(u^n-u)(t,s,x,y) \\
		& + \sum_{|I|\leq 2} N^I(t,s,x,y) \mathcal{I}^I \left[ \frac{\partial u}{\partial t}, \frac{\partial u}{\partial x} \right](t,s,x,y) \\
		& + K(t,s,x,y), \quad (t,s,x,y) \in \triangle[0,\delta] \times \mathbb{R}^{2d}, \\
		(u^n-u)(t,0,x,y) = \;& 0, \quad (t,x,y) \in [0,\delta] \times \mathbb{R}^{2d}.
	\end{aligned}
	\right.
\end{equation}
	where 
    \begin{equation*} 
	\left\{
	\begin{aligned}
		M^I(t,s,x,y) = \;& \int^1_0 \left( \frac{\partial \overline{\mathcal{H}}}{\partial\psi} \frac{\partial\psi}{\partial(\partial_I u)} + \frac{\partial \overline{\mathcal{H}}}{\partial(\partial_I u)} \right) \bigl(s,y,\eta_\sigma(t,s,x,y)\bigr) d\sigma, \\ 
		N^I(t,s,x,y) = \;& \int^1_0 \frac{\partial \overline{\mathcal{H}}}{\partial\psi} \frac{\partial\psi}{\partial(\partial_I u)} \bigl(s,y,\eta_\sigma(t,s,x,y)\bigr) d\sigma, \\
		K(t,s,x,y) = \;& \int^1_0 \frac{\partial \overline{\mathcal{H}}}{\partial\psi} \frac{\partial\psi}{\partial(t,x)} \bigl(s,y,\eta_\sigma(t,s,x,y)\bigr) d\sigma \cdot \int^{(t,x)}_{(s,y)} 1 d\theta, \\
		\eta_\sigma(t,s,x,y) = \;& \sigma \cdot \bigl(t,x,(\partial_I u^n)_{|I|\leq 2}(t,s,x,y)\bigr) \\ 
		& + (1-\sigma) \cdot \bigl(s,y,(\partial_I u)_{|I|\leq 2}(s,s,x,y)\big|_{x=y}\bigr). 
	\end{aligned}
	\right.
\end{equation*}
	Then, by setting $(t,x)=(s,y)$ and following earlier analyses, we have
    \begin{equation*}
		\begin{aligned}0 \leq \;& (V-V^n)(T-s,y) = \bigl| (V^n-V)(T-s,y) \bigr| = \bigl| (u^n-u)(s,s,y,y) \bigr| \\
		\leq \;& C(R) \int^s_0 d\tau \int_{\mathbb{R}^d} (s-\tau)^{-\frac{d}{2}} \exp\bigl\{ -c(R) \varpi(s,\tau,y,\xi) \bigr\} \bigl( |s-\tau| + |y-\xi| \bigr) d\xi \\
		\leq \;& C(R) (s^2 + s^{\frac{3}{2}}), \quad s \in [0,\delta].
	\end{aligned}
\end{equation*}
	where $\varpi(s,\tau,y,\xi)$ is defined in Section \ref{App:A}. The proof is completed.
\end{proof}

Compared with the nonlocal PDE \eqref{EquilibriumHJBequation}, the equation \eqref{DifferenceEq1} can be considered as a family of classical PDEs parameterized by $(t,x)$. Hence, after solving \eqref{DifferenceEq1}, it is helpful to use the difference between $V$ and $V^n$ to estimate equilibrium value function $V$. Moreover, if \eqref{EquilibriumHJBequation} is globally solvable, the estimate can be extended to the whole time horizon.

\section{Another Financial Example: Power Utility}
\label{App:PowerU}
Although this alternative example with potentially degenerate coefficients falls outside our general framework, we make a targeted attempt to address it. By applying specific ansatzs for the solutions, we derive explicit expressions for the equilibrium policy and equilibrium value function, while the latter is characterized by a (nonlocal) ODE system. By proving global solvability of the ODE system, we establish the global solvability of the original problem.

In a similar market model as in the Section \ref{Sec:FinEx}, an investor needs to decide not only the amount $\alpha(\cdot)$ of money to invest in the risk asset but also the consumption amount \(c(\cdot)\). The TIC recursive utility process \((Y(\cdot), Z(\cdot))\) then satisfies the following controlled FBSDE:
\begin{equation} \label{ExampleFBSDE}
	\left\{
	\begin{aligned}
		dX(s) = \;& \bigl[ rX(s) + (\mu-r)\alpha(s) - c(s) \bigr] ds + \sigma \alpha(s) dW(s), & s \in [t,T], \\
		dY(s) = \;& -\bigl[ v(t,s)c(s)^\beta - w(t,s)Y(s) + z(t,s)x^\gamma \bigr] ds + Z(s) dW(s), & s \in [t,T], \\
		X(t) = \;& x, \quad Y(T) = g^{(1)}(t)X(T)^\beta + g^{(2)}(t)x^\gamma, & t \in [0,T].
	\end{aligned}
	\right.
\end{equation}
where $\beta,\gamma\in(0,1)$, $v$, $w$, $z$, $g^{(1)}$, and $g^{(2)}$ are all continuous and positive functions. Similarly, we define the recursive utility functional for the investor by $J(t,x;\alpha(\cdot),c(\cdot)):=Y(t;t,x,\alpha(\cdot),c(\cdot))$. Hence, the problem for the investor is to identify the optimal investment and consumption policy such that a sort of mixed power utilities of the instantaneous and terminal wealth is maximized. We consider the Hamiltonian function
\begin{equation*}
	\mathcal{H}(t,s,x,y,a,c,u,p,q) = \frac{1}{2}\sigma^2 a^2 q + \bigl[ ry + (\mu-r)a - c \bigr] p + \bigl[ v(t,s)c^\beta - w(t,s)u + z(t,s)x^\gamma \bigr].
\end{equation*}
Maximizing it with respect to $(a,c)$ yields the maxima (for $p>0$ and $q<0$) 
\begin{equation*}
	\bar{a} = -\frac{(\mu-r)p}{\sigma^2 q}, \quad \bar{c} = \left( \frac{p}{\beta v(t,s,x)} \right)^{\frac{1}{\beta-1}}.
\end{equation*}
Consequently, the equilibrium policy admits the forms 
\begin{equation*}
	\mathbbm{a}(s,y) = -\frac{(\mu-r) u_y(s,s,x,y)\big|_{x=y}}{\sigma^2 u_{yy}(s,s,x,y)\big|_{x=y}}, \quad \mathbbm{c}(s,y) = \left( \frac{u_y(s,s,x,y)\big|_{x=y}}{\beta v(s,s,y)} \right)^{\frac{1}{\beta-1}}
\end{equation*}
with $u(t,s,x,y)$ being the solution to the following equilibrium HJB equation 
\begin{equation} \label{PowerHJB}
	\left\{
	\begin{aligned}
		u_s(t,s,x,y) & + \frac{(\mu-r)^2 \bigl(u_y(s,s,x,y)\big|_{x=y}\bigr)^2}{2\sigma^2 \bigl(u_{yy}(s,s,x,y)\big|_{x=y}\bigr)^2} u_{yy}(t,s,x,y) \\
		& + \left[ ry - \frac{(\mu-r)^2 u_y(s,s,x,y)\big|_{x=y}}{\sigma^2 u_{yy}(s,s,x,y)\big|_{x=y}} - \left( \frac{u_y(s,s,x,y)\big|_{x=y}}{\beta v(s,s,y)} \right)^{\frac{1}{\beta-1}} \right] u_y(t,s,x,y) \\
		& + v(t,s) \left( \frac{u_y(s,s,x,y)\big|_{x=y}}{\beta v(s,s,y)} \right)^{\frac{\beta}{\beta-1}} - w(t,s)u(t,s,x,y) + z(t,s)x^\gamma = 0, \\
		u(t,t,x,y) & = g^{(1)}(t)y^\beta + g^{(2)}(t)x^\gamma, \quad 0 \leq t \leq s \leq T, \quad x,y \in (0, \infty).
	\end{aligned}
	\right.
\end{equation}
It is clear that the first-order derivative of the nonlinearity in \eqref{PowerHJB} with respect to \(u_{yy}(t,s,x,y)\) at \(u(t,T,x,y)\) is degenerate. Consequently, our well-posedness results cannot be directly applied to claim its solvability. Hence, this second example serves more as a source of inspiration and underscores the importance of addressing the general (degenerate) case.

By observing the terminal condition of \eqref{ExampleFBSDE}, we consider the following ansatz for $u$: 
\begin{equation*}
	u(t,s,x,y) = \varphi^{(1)}(t,s)y^\beta + \varphi^{(2)}(t,s)x^\gamma, \quad 0 \leq t \leq s \leq T, \quad x,y \in (0,\infty),
\end{equation*}
for $\varphi^{(1)}(t,s)$ and $\varphi^{(2)}(t,s)$ to be determined. Then, we have $\varphi^{(1)}(t,T)=g^{(1)}(t)$, $\varphi^{(2)}(t,T)=g^{(2)}(t)$, and
\begin{equation*}
		\begin{aligned}& \varphi^{(1)}_s(t,s)y^\beta + \varphi^{(2)}_s(t,s)x^\gamma + \frac{(\mu-r)^2 \bigl(\varphi^{(1)}(s,s)\beta y^{\beta-1}\bigr)^2}{2\sigma^2 \bigl(\varphi^{(1)}(s,s)\beta(\beta-1)y^{\beta-2}\bigr)^2} \varphi^{(1)}(t,s)\beta(\beta-1)y^{\beta-2} \\
		& + \left[ ry - \frac{(\mu-r)^2 \varphi^{(1)}(s,s)\beta y^{\beta-1}}{\sigma^2 \varphi^{(1)}(s,s)\beta(\beta-1)y^{\beta-2}} - \left( \frac{\varphi^{(1)}(s,s)\beta y^{\beta-1}}{\beta v(s,s,y)} \right)^{\frac{1}{\beta-1}} \right] \varphi^{(1)}(t,s)\beta y^{\beta-1} \\
		& + v(t,s) \left( \frac{\varphi^{(1)}(s,s)\beta y^{\beta-1}}{\beta v(s,s,y)} \right)^{\frac{\beta}{\beta-1}} - w(t,s)\varphi^{(1)}(t,s)y^\beta - w(t,s)\varphi^{(2)}(t,s)x^\gamma + z(t,s)x^\gamma \\
		= \;& \Biggl\{ \varphi^{(1)}_s(t,s) + \frac{(\mu-r)^2\beta}{2\sigma^2(\beta-1)}\varphi^{(1)}(t,s) + \left[ r\beta - \frac{(\mu-r)^2\beta }{\sigma^2(\beta-1)} - \beta \left( \frac{\varphi^{(1)}(s,s)}{v(s,s)} \right)^{\frac{1}{\beta-1}} \right] \varphi^{(1)}(t,s) \\
		& + v(t,s) \left( \frac{\varphi^{(1)}(s,s)}{ v(s,s)} \right)^{\frac{\beta}{\beta-1}} - w(t,s)\varphi^{(1)}(t,s) \Biggr\} y^\beta + \Bigl\{ \varphi^{(2)}_s(t,s) - w(t,s)\varphi^{(2)}(t,s) + z(t,s) \Bigr\} x^\gamma \\
		= \;& 0.
	\end{aligned}
\end{equation*}
Therefore, $\varphi^{(1)}(t,s)$ and $\varphi^{(2)}(t,s)$ satisfy the following ODEs: 
\begin{equation} \label{ODEsystem} 
	\left\{
	\begin{aligned}
		& \varphi^{(1)}_s(t,s) + \left[ k(t,s) - \beta \left( \frac{\varphi^{(1)}(s,s)}{v(s,s)} \right)^{\frac{1}{\beta-1}} \right] \varphi^{(1)}(t,s) + v(t,s) \left( \frac{\varphi^{(1)}(s,s)}{ v(s,s)} \right)^{\frac{\beta}{\beta-1}} = 0, \\
		& \varphi^{(2)}_s(t,s) - w(t,s)\varphi^{(2)}(t,s) + z(t,s) = 0, \quad 0 \leq t \leq s \leq T, \\
		& \varphi^{(1)}(t,T) = g^{(1)}(t), \quad \varphi^{(2)}(t,T) = g^{(2)}(t), \quad 0 \leq t \leq T.
	\end{aligned}
	\right.
\end{equation}
where $k(t,s):=r\beta-\frac{(\mu-r)^2\beta }{2\sigma^2(\beta-1)}-w(t,s)$. Then, by variation of constants method, we have
\begin{equation} \label{varphi1}
		\begin{aligned}\varphi^{(1)}(t,s) = \;& \exp \left\{ \int^T_s \left[ k(t,\tau) - \beta \left( \frac{\varphi^{(1)}(\tau,\tau)}{v(\tau,\tau)} \right)^{\frac{1}{\beta-1}} \right] d\tau \right\} g^{(1)}(t) \\
		& + \int^T_s \exp \left\{ \int^\lambda_s \left[ k(t,\tau) - \beta \left( \frac{\varphi^{(1)}(\tau,\tau)}{v(\tau,\tau)} \right)^{\frac{1}{\beta-1}} \right] d\tau \right\} v(t,\lambda) \left( \frac{\varphi^{(1)}(\lambda,\lambda)}{v(\lambda,\lambda)} \right)^{\frac{\beta}{\beta-1}} d\lambda
	\end{aligned}
\end{equation}
and 
\begin{equation} \label{varphi2}
	\varphi^{(2)}(t,s) = \exp \left\{ -\int^T_s w(t,\tau) d\tau \right\} g^{(2)}(t) + \int^T_s \exp \left\{ -\int^\lambda_s w(t,\tau) d\tau \right\} z(t,\lambda) d\lambda
\end{equation}
for $0\leq t\leq s\leq T$ and $x,y\in(0,\infty)$. It is clear that \eqref{varphi2} admits a unique solution since it can be considered as a family of classical ODEs parameterized by $t$. Taking $t=s$ in \eqref{varphi1} gives 
\begin{equation}
		\begin{aligned}\varphi^{(1)}(s,s) = \;& \exp \left\{ \int^T_s \left[ k(s,\tau) - \beta \left( \frac{\varphi^{(1)}(\tau,\tau)}{v(\tau,\tau)} \right)^{\frac{1}{\beta-1}} \right] d\tau \right\} g^{(1)}(s) \\
		& + \int^T_s \exp \left\{ \int^\lambda_s \left[ k(s,\tau) - \beta \left( \frac{\varphi^{(1)}(\tau,\tau)}{v(\tau,\tau)} \right)^{\frac{1}{\beta-1}} \right] d\tau \right\} v(s,\lambda) \left( \frac{\varphi^{(1)}(\lambda,\lambda)}{v(\lambda,\lambda)} \right)^{\frac{\beta}{\beta-1}} d\lambda
	\end{aligned}
\end{equation}
Let us denote by
\begin{equation*}
	\bar{\varphi}^{(1)}(s) = \frac{\varphi^{(1)}(s,s)}{v(s,s)}, \quad \bar{g}^{(1)}(s) = \frac{g^{(1)}(s)}{v(s,s)}, \quad \bar{v}(t,s) = \frac{v(t,s)}{v(t,t)}. 
\end{equation*}
Then, we obtain a nonlinear integral equation for $\overline{\varphi}^{(1)}(s)$:
\begin{equation} \label{Integralequation}
		\begin{aligned}\bar{\varphi}^{(1)}(s) = \;& \exp\left\{ \int^T_s \left[ k(s,\tau) - \beta \bar{\varphi}^{(1)}(\tau)^{\frac{1}{\beta-1}} \right] d\tau \right\} \bar{g}^{(1)}(s) \\
		& + \int^T_s \exp\left\{ \int^\lambda_s \left[ k(s,\tau) - \beta \bar{\varphi}^{(1)}(\tau)^{\frac{1}{\beta-1}} \right] d\tau \right\} \bar{v}(s,\lambda) \bar{\varphi}^{(1)}(\lambda)^{\frac{\beta}{\beta-1}} d\lambda
	\end{aligned}
\end{equation}
If \eqref{Integralequation} is solvable, it turns out that there exists a unique solution $\varphi^{(1)}(t,s)$ solving \eqref{varphi1}. Moreover, if the classical parameterized ODE for $\varphi^{(2)}$ in \eqref{ODEsystem} is solvable, the proof of global solvability of $\varphi^{(1)}$ can be obtained by similar arguments in \cite{Yong2012,Wei2017,Lei2021}. 

In what follows, we claim that the ODE system \eqref{ODEsystem} admits a unique positive solution $(\varphi^{(1)},\varphi^{(2)})(t,s)$ for $0\leq t\leq s\leq T$ such that the TIC stochastic control problem \eqref{Exampleutility}-\eqref{ExampleFBSDE} is solvable globally.
\begin{proposition} \label{prop:ex}
	For $v$, $w$, and $z:\nabla[0,T]\to(0,\infty)$ and $g^{(1)}$, $g^{(2)}:[0,T]\to(0,\infty)$ being continuous, and $s\mapsto v(t,s)$ being continuously differentiable, the ODE system \eqref{ODEsystem} admits a unique positive solution $(\varphi^{(1)},\varphi^{(2)})(t,s)$ in $0\leq t\leq s\leq T$. Consequently, the equilibrium value function and the equilibrium policy for problem \eqref{ExampleFBSDE}-\eqref{Exampleutility} are given by
    \begin{equation} \label{Solution} 
	\left\{
	\begin{aligned}
		V(s,y) = \;& \varphi^{(1)}(s,s)y^\beta + \varphi^{(2)}(s,s)y^\gamma, \\
		\mathbbm{a}(s,y) = \;& -\frac{(\mu-r)}{\sigma^2(\beta-1)}y,  \\
		\mathbbm{c}(s,y) = \;& \left( \frac{\varphi^{(1)}(s,s)}{ v(s,s,y)} \right)^{\frac{1}{\beta-1}} y, \quad (s,y) \in [0,T] \times (0,\infty).
	\end{aligned}
	\right.
\end{equation}
\end{proposition}
\begin{proof}
One could take advantage of fixed-point arguments to show the local well-posedness of \eqref{Integralequation}. Let us focus on proving the lower and upper bounds of $\varphi^{(1)}(s)$, which suffices to guarantee the global existence of solutions of \eqref{Integralequation} by an analytic continuation.

First, we let
\begin{equation*}
	\widehat{\varphi}^{(1)}(s)=\overline{\varphi}^{(1)}(s)\exp\left\{\beta\int^T_s\overline{\varphi}^{(1)}(\tau)^\frac{1}{\beta-1}d\tau\right\},
\end{equation*}
\begin{equation*}
	\widehat{g}(s)=\overline{g}(s)\exp\left\{\int^T_s k(s,\tau)d\tau\right\}, \qquad \widehat{v}(t,s)=\overline{v}(t,s)\exp\left\{\int^s_t k(t,\tau)d\tau\right\}.
\end{equation*}
Then, \eqref{Integralequation} can be rewritten as 
\begin{equation} \label{Integralequation2}
	\widehat{\varphi}^{(1)}(s)=\widehat{g}(s)+\int^T_s\widehat{\varphi}^{(1)}(\tau)\overline{\varphi}^{(1)}(\tau)^\frac{1}{\beta-1}\widehat{v}(s,\tau)d\tau, \quad s\in[0,T]. 
\end{equation}
Under some suitable conditions (see \cite[Propositon 7.1]{Wei2017}), the following inequalities hold:  
\begin{equation*}
	\widehat{h}(s)\geq g_0>0, \qquad \widehat{v}(t,s)\geq \exp\big\{-\varrho(s-t)\big\}
\end{equation*}
for some constants $g_0>0$ and $\varrho>0$. From \eqref{Integralequation2}, it is clear that 
\begin{equation*} 
	\widehat{\varphi}^{(1)}(s)\geq g_0+\int^T_s\exp\big\{-\varrho(\tau-s)\big\}\widehat{\varphi}^{(1)}(\tau)\overline{\varphi}^{(1)}(\tau)^\frac{1}{\beta-1}d\tau, \quad s\in[0,T],
\end{equation*}
which is equivalent to
\begin{equation*} 
	\widehat{\varphi}^{(1)}(s)\exp\big\{-\varrho s\big\}\geq g_0\exp\big\{-\varrho s\big\}+\int^T_s\Big[\widehat{\varphi}^{(1)}(\tau)\exp\big\{-\varrho\tau\big\}\Big]\overline{\varphi}^{(1)}(\tau)^\frac{1}{\beta-1}d\tau\equiv\xi(s).  
\end{equation*}
Since, 
\begin{equation*} 
	\begin{split}
		\xi^\prime(s)&=-\varrho g_0\exp\big\{-\varrho s\big\}-\Big[\widehat{\varphi}^{(1)}(s)\exp\big\{-\varrho s\big\}\Big]\overline{\varphi}^{(1)}(s)^\frac{1}{\beta-1} \\
		&\leq -\varrho g_0\exp\big\{-\varrho s\big\}-\xi(s)\overline{\varphi}^{(1)}(s)^\frac{1}{\beta-1},
	\end{split}    
\end{equation*}
we have 
\begin{equation*}
	\left[\xi(s)\exp\left\{-\int^T_s\overline{\varphi}^{(1)}(\tau)^\frac{1}{\beta-1}d\tau\right\}\right]^\prime\leq -\rho g_0\exp\left\{-\varrho s-\int^T_s\overline{\varphi}^{(1)}(\tau)^\frac{1}{\beta-1}d\tau\right\}. 
\end{equation*}
Integrating both sides over $[s,T]$, we obtain 
\begin{equation*}
		\begin{aligned}\xi(s) \geq \;& \exp \left\{ \int^T_s \bar{\varphi}^{(1)}(\tau)^{\frac{1}{\beta-1}} d\tau \right\} \\
		& \times \left[ g_0 e^{-\varrho T} + \rho g_0 \int^T_s \exp \left\{ -\varrho \lambda - \int^T_\lambda \bar{\varphi}^{(1)}(\tau)^{\frac{1}{\beta-1}} d\tau \right\} d\lambda \right].
	\end{aligned}
\end{equation*}
Consequently, 
\begin{equation*}
		\begin{aligned}\bar{\varphi}^{(1)}(s) &= \hat{\varphi}^{(1)}(s) \exp \left\{ -\beta \int^T_s \bar{\varphi}^{(1)}(\tau)^{\frac{1}{\beta-1}} d\tau \right\} \\
		&\geq \exp \left\{ -\beta \int^T_s \bar{\varphi}^{(1)}(\tau)^{\frac{1}{\beta-1}} d\tau + \varrho s \right\} \xi(s) \\
		&\geq \exp \left\{ (1-\beta) \int^T_s \bar{\varphi}^{(1)}(\tau)^{\frac{1}{\beta-1}} d\tau + \varrho s \right\} \\
		&\quad \times \left[ g_0 e^{-\varrho T} + \varrho g_0 \int^T_s \exp \left\{ -\varrho \lambda - \int^T_\lambda \bar{\varphi}^{(1)}(\tau)^{\frac{1}{\beta-1}} d\tau \right\} d\lambda \right] \\
		&\geq e^{-\varrho(T-s)} g_0 \geq \delta > 0.
	\end{aligned}
\end{equation*}
The lower bound of $\overline{\varphi}^{(1)}(s)$ leads to 
\begin{equation*}
	\overline{\varphi}^{(1)}(s)^\frac{\beta}{\beta-1}=\frac{1}{\overline{\varphi}^{(1)}(s)^\frac{\beta}{1-\beta}}\leq \frac{1}{\delta^\frac{\beta}{1-\beta}}\leq K, 
\end{equation*}
On the other hand, \eqref{Integralequation2} implies that 
\begin{equation*}
		\begin{aligned}\bar{\varphi}^{(1)}(s) = \;& \exp\left\{ \int^T_s \left[ k(s,\tau) - \beta \bar{\varphi}^{(1)}(\tau)^{\frac{1}{\beta-1}} \right] d\tau \right\} \bar{g}^{(1)}(s) \\
		& + \int^T_s \exp\left\{ \int^\lambda_s \left[ k(s,\tau) - \beta \bar{\varphi}^{(1)}(\tau)^{\frac{1}{\beta-1}} \right] d\tau \right\} \bar{v}(s,\lambda) \bar{\varphi}^{(1)}(\lambda)^{\frac{\beta}{\beta-1}} d\lambda \\
		\leq \;& \exp\left\{ \int^T_s k(s,\tau) d\tau \right\} \bar{g}^{(1)}(s) + \int^T_s \exp\left\{ \int^\lambda_s k(s,\tau) d\tau \right\} \bar{v}(s,\lambda) \frac{1}{\delta^{\frac{\beta}{1-\beta}}} d\lambda \\
		\leq \;& K.
	\end{aligned}
\end{equation*}
The proof is completed.
\end{proof}





 







\bibliography{references}

\end{document}